\documentclass{article}[12pts]

\usepackage{overpic}
\usepackage{amsmath}
\usepackage{xfrac}
\usepackage{blindtext}
\usepackage{comment}
\usepackage{faktor}
\usepackage[all]{xy}
\usepackage[french, main=english]{babel}
\usepackage[utf8x]{inputenc}
\usepackage{float}
\usepackage{graphicx}
\usepackage[colorinlistoftodos]{todonotes}
\usepackage{url}
\usepackage{hyperref}
\usepackage{array}
\usepackage{tabularx}
\usepackage{setspace}
\usepackage[T1]{fontenc}
\usepackage{subfig}
\usepackage[margin=1in]{geometry} 
\usepackage{fancyhdr}
\usepackage{rotating}
\usepackage{amsfonts,amssymb,amsthm,scrextend}
\usepackage{pinlabel}
\usepackage[nottoc,numbib]{tocbibind}
\usepackage{appendix}
\usepackage{lscape}

\newcommand{\N}{\mathbb{N}}

\newcommand{\Z}{\mathbb{Z}}

\newcommand{\R}{\mathbb{R}}

\newcommand{\F}{\mathcal{F}}
\newcommand{\G}{\mathcal{G}}

\newcommand{\Ai}{\mathcal{A_{\infty}}}

\newcommand{\deffct}[5]{#1: \begin{array}{ccc}
    #2 & \to & #3  \\
     #4 & \mapsto & \displaystyle #5
\end{array}}
\newcommand{\trajb}[1]{\overline{\mathcal{L}}(#1)}
\newcommand{\trajbi}[2]{\overline{\mathcal{L}}_{#1}(#2)}
\newcommand{\traji}[2]{\mathcal{L}_{#1}(#2)}
\newcommand{\traj}[1]{\mathcal{L}(#1)}
\newcommand{\wb}[2]{\overline{W}^{#1}(#2)}
\newcommand{\wbi}[3]{\overline{W}^{#1}_{#2}(#3)}

\newcommand{\Crit}{\textup{Crit}}
\newcommand{\Or}{\textup{Or}}
\newcommand{\Id}{\textup{Id}}
\newcommand{\Coor}{\textup{Coor}}
\newcommand{\Hom}{\textup{Hom}}
\newcommand{\End}{\textup{End}}
\newcommand{\EZ}{\textup{EZ}}
\newcommand{\ev}{\textup{ev}}
\newcommand{\red}[1]{\textcolor{red}{#1}}
\newcommand{\blue}[1]{\textcolor{blue}{#1}}

\newcommand{\Y}{\mathcal{Y}}

\newcommand{\Xq}{X/\Y}

\newcommand{\ftimes}[2]{\hspace{0.1em} _{#1}\!\times_{#2}}

\newcommand{\ls}[1]{\mathcal{L} #1}
\newcommand{\CS}{\textup{CS}_{DG}}

\newcommand{\m}{\mathbf{\tilde{m}}}
\newcommand{\mL}{\mathbf{\tilde{m}^L}}

\renewcommand{\qed}{\hfill$\blacksquare$}

\renewenvironment{proof}{\begin{addmargin}[1em]{0em}\begin{newproof}}{\end{newproof}\end{addmargin}\qed}

\newenvironment{myproof}[2]{ \begin{addmargin}[1em]{0em} \emph{Proof {#1} {#2}.}}{\end{addmargin}\qed}

\definecolor{darkgreen}{RGB}{0,168,0}

\newtheorem{thm}{Theorem}[section]
\newtheorem{notation}[thm]{Notation}
\newtheorem{defi}[thm]{Definition}
\newtheorem{rem}[thm]{Remark}
\newtheorem{prop}[thm]{Proposition}
\newtheorem{cor}[thm]{Corollary}
\newtheorem{lemme}[thm]{Lemma}

\newtheorem{theorem}{Theorem}

\title{A chain-level model for Chas-Sullivan products in Morse homology with differential graded coefficients}
\author{Robin Riegel\\
IRMA, Université de Strasbourg}
\date{}

\begin{document}
\selectlanguage{english}
\renewcommand{\contentsname}{Table of Contents}
\renewcommand{\proofname}{Proof}
\renewcommand{\refname}{References}
\renewcommand{\figurename}{Figure}
\setlength{\parindent}{0pt}

\maketitle

\begin{center}
    \textbf{Abstract}
\end{center}

We use the framework of Morse theory with differential graded coefficients to study certain operations on the total space of a fibration. More particularly, we focus in this paper on a chain-level description of the Chas-Sullivan product on the homology of the free loop space of an oriented, closed and connected manifold. The idea of "intersecting on the base" and "concatenating on the fiber" is well-adapted to this framework. We also give a Morse theoretical description of other products that follow the same principle. For this purpose, we develop functorial properties with respect to the coefficient in terms of morphisms of $\Ai$-modules and morphisms of fibrations. We also build a differential graded version of the Künneth formula and of the Pontryagin-Thom construction. 

\tableofcontents

\newpage

\section{Introduction and main results}

Given an oriented $n$-dimensional closed manifold $X$, Chas and Sullivan described in \cite{ChaSu99} a degree $-n$ product  $$H_i(\ls{X}) \otimes H_j(\ls{X}) \to H_{i+j-n}(\ls{X})$$ on the homology of the free loop space $\ls{X} = C^{0}(\mathbb{S}^1,X)$ of $X$  as an intersection product on $H_*(X)$ combined with the Pontryagin product on based loops. This multiplication is now a fundamental operation in string topology, the study of the homology of loop spaces.\\

\emph{State of the art.}

A geometric realisation of this product has been described by Cohen-Jones \cite{cohen2002homotopy} in terms of the Thom spectrum $\ls{X}^{-TX} = \Sigma^{-N-d} Th(\ev^*\nu^N)$ where $\ev : \ls{X} \to X$ is the evaluation at the basepoint and $\nu^N \to X$ is the $N$-dimensional normal bundle of a fixed embedding $X \to \R^{N+n}$. This approach has been generalized by Gruher-Salvatore \cite{GS07} to define a product on the homology of any total space of a fiber bundle $E \overset{p}{\to} X$ equipped with a "multiplication map" $m : E \ftimes{p}{p} E \to E$ and a "unit section" $s : X \to E$.
A chain-level description of the Chas-Sullivan product on the singular complex of $\ls{X}$ has been given by Hingston-Wahl in \cite{hingston2022product}. This description consists in intersecting two chains $A,B \in C_*(\ls{X})$ by considering a tubular neighborhood $\mathcal{U}$ of the diagonal $\Delta X \subset X^2$ and taking the cap-product with $A\times B$ of the pullback of a representative $\tau$ of the Thom-class of $\mathcal U \to \Delta X$ to obtain $(\ev \times \ev)^*\tau \cap (A \times B) \in C_*((\ev \times \ev)^{-1} \mathcal{U})$ and then retracting into the \textbf{figure-eight space} $\ls{X} \ftimes{\ev}{\ev} \ls{X}$ using "geodesic sticks" before concatenating.
Abbondandolo-Schwarz constructed in \cite{AS10} a loop product using the fact that the figure-eight space has a smooth tubular neighborhood in $\ls{X}\times \ls{X}$ and thus use an "Umkehr map" $e_! : H_*(\ls{X} \times \ls{X}) \to H_{*-n}(\ls{X} \ftimes{\ev}{\ev} \ls{X})$.
Laudenbach defines in \cite{Lau11} a loop-product in terms of "transverse bi-simplices" $u\times v$ in $\ls{X} \times \ls{X}$, seeing $\ls{X}$ as a simplicial set instead of a topological space.
Chataur introduced in \cite{Cha05} a geometric way to define the loop product from the point of view of geometric homology theory.
Abouzaid took in \cite{Abo15} the approach of approximating the loop space $\ls{X}$ by finite dimensional spaces of piecewise geodesic loops in order to use Morse theory.

Most of these constructions either use the Hilbert space of $H^1$-loops to represent the loop space $\ls{X}$ and use infinite dimensional considerations, or define the loop product only at the homology level. We construct in this paper a chain-level model for Chas-Sullivan products using finite dimensional arguments.
\\

\emph{Our construction.}

Let $(X^n,\star)$ be a smooth pointed, oriented, closed, connected manifold.
The first goal of this paper is to give a finite dimensional Morse theoretic chain-level model for the Chas-Sullivan product and generalize it to any Hurewicz fibration $F \hookrightarrow E \overset{\pi}{\to} X$ equipped with a \textbf{morphism of fibrations} $m : E\  _{\pi}\!\times_{\pi} E  \to E$, ie $m$ is continuous and preserves the fibrations $\pi \circ m = \pi$.

Our construction relies on the fact that there exists a notion of \emph{intersecting on the base} and \emph{multiplying on the fiber}  in the setting of \textbf{enriched Morse homology} defined and studied by Barraud-Damian-Humilière-Oancea \cite{BDHO23} using the original idea of Barraud and Cornea \cite{BC07}. The authors define a Morse homology theory with coefficients in any differential graded (DG) right-module $(\F_*, \partial_{\F})$ over $C_*(\Omega X)$, the cubical chains on the Moore based loop space over $X$ with algebra structure given by the Pontryagin product. Such a DG-module $\F$ is called a \textbf{DG local system}.

 Given a Morse-Smale pair $ (f, \xi)$, they build following \cite{BC07} a Maurer-Cartan element referred to as the \textbf{Barraud-Cornea twisting cocycle} $\{m_{x,y} \in C_{|x|-|y|-1}(\Omega X), \ x,y \in \Crit(f)\}$ satisfying

 \begin{equation}\label{eq : Maurer-Cartan}
     \partial m_{x,y} = \sum_{z \in \Crit(f)} (-1)^{|x|-|z|} m_{x,z} \cdot m_{z,y}
 \end{equation}
 
 and use it to define the \textbf{enriched Morse complex} or \textbf{Morse complex with DG coefficients} 

$$C_*(X, \F_*) = \F_* \otimes_{\Z} \Z\Crit(f) $$ with the twisted differential $$\partial(\alpha \otimes x) = \partial_{\F} \alpha \otimes x + (-1)^{|\alpha|} \sum_y \alpha \cdot m_{x,y} \otimes y.$$

One of their main results is what we will refer to as the \textbf{Fibration Theorem}.

Given a fibration $F \hookrightarrow E \overset{\pi}{\to} X$, a \textbf{transitive lifting function} $\Phi : E \ftimes{\pi}{\ev_0} \mathcal{P}X \to E$ is a map that lifts all paths in $X$ and respects concatenation in the sense that lifting a concatenation of two paths is the same as lifting them one after the other : $\Phi(\Phi(e,\gamma),\delta) = \Phi(e,\gamma \#\delta).$\\

\textbf{Fibration Theorem} (\cite{BDHO23} Theorem 7.2.1) : \emph{Let $F \hookrightarrow E \to X$ be a Hurewicz fibration and equip $C_*(F)$ with the $C_*(\Omega X)$-module structure induced by a transitive lifting function associated to this fibration (see \cite[Section 7.1]{BDHO23}). Then, there exists a quasi-isomorphism denoted $$\Psi_E : C_*(X, C_*(F)) \to C_*(E).$$}

This result was also proved by Charette in \cite[Theorem 1.1.1]{Char17} considering these complexes as $\Z/2\Z$-vector spaces instead of $\Z$-modules.

We will prove in Corollary \ref{cor : homotopy euqivalence if different lifting functions} that the complex $C_*(X,C_*(F))$ does not depend, up to homotopy equivalence, on the choice of transitive lifting function.  

This setting therefore provides a description of a chain complex on $X$ that computes the homology of the total space $E$ of a fibration in terms of Morse data on the base $X$ twisted by the cubical chain complex of the fiber $F$. It is particularly adapted to operations such as the Chas-Sullivan product. This setting also provides an operation that is a model for the \emph{intersection on the base} which is the shriek map

$$\xymatrix{C_*(X \times X, C_*(F \times F)) \ar[r]^-{\Delta_!} \ar[d]^{\Psi_{E^2}} & C_{*-n}(X, \Delta^*C_*(F \times F)) \ar[d]^{\Psi_{E\!\ftimes{\pi}{\pi} E}}\\
C_*(E \times E) \ar@{-->}[r] & C_{*-n}(E\ftimes{\pi}{\pi} E)}$$

induced by the diagonal $\Delta : X \to X\times X$ (see sections 9 and 10 of \cite{BDHO23}). \\

We now state our generalization of the Chas-Sullivan product.

\begin{theorem}[\ref{thm : A}] Let $F \hookrightarrow E \overset{\pi}{\to} X$ be a fibration and denote $\F=C_*(F)$ with the DG-module structure over $C_*(\Omega X)$  induced by a transitive lifting function $\Phi : E \ftimes{\pi}{\ev_0} \mathcal{P} X \to E$ associated to the fibration.  Let $m : E\ _{\pi}\!\times_{\pi} E \to E$ be a morphism of fibrations over $X$ (see definition \ref{defi : Ai morphism of fibration}). This morphism induces a degree $-n$ product
$$\CS : H_*(X, \F)^{\otimes 2} \to H_*(X, \F).$$

The following properties hold :
    \begin{itemize}
    \item \textbf{Associativity (\ref{prop : associativité CSDG}) :} If $m_*$ is associative in homology, then so is $\CS$.
    \item \textbf{Commutativity (\ref{prop : commutativité CSDG}) :} If $m_*$ is commutative in homology, then $\CS$ is commutative up to sign in homology, ie $\CS(\gamma \otimes \tau) = (-1)^{(n-|\gamma|)(n-|\tau|)} \CS(\tau \otimes \gamma).$
    \item \textbf{Neutral element (\ref{prop : Neutral element}) :}  If $\pi$ admits a section $s : X \to E$ such that $m(s(\pi(e)), e) = m(e, s(\pi(e))) = e$ for all $e \in E$, then $\CS$ admits a neutral element.
    \item \textbf{Functoriality (\ref{prop : Functoriality property}) :}\\ $\bullet$ For any pointed, oriented, closed and connected manifold $Y^k$, any continuous map $g : Y \to X$ induces a degree $-k$ product for the fibration $F \hookrightarrow g^*E \to Y$

    $$\CS^Y : H_i(Y,g^*\F)\otimes H_j(Y, g^*\F) \to H_{i+j-k}(Y,g^*\F)$$ such that $g_! : H_*(X,\F) \to H_{*+n-k}(Y, g^*\F)$ is a morphism of rings up to sign.
    
     $\bullet$ If $g$ is an orientation-preserving homotopy equivalence then,
     $g_!$ and $g_* : H_*(Y,g^*\F) \to H_{*}(X, \F)$  are isomorphisms of rings inverses of each other.
     
     \item \textbf{Spectral sequence (\ref{prop : spectral sequence CSDG}) :} The canonical filtration
    $$F_p(C_*(X, \F)) = \bigoplus_{\substack{i +j = k\\ i \leq p}} \F_j \otimes \Z\Crit_i(f)$$
   induces a spectral sequence $E^r_{p,q}$ that is endowed with an algebra structure
    $$E^r_{p,q} \otimes E^r_{l,m} \to E^r_{p+l-n, q+m}$$
   induced by a chain description $\CS : C_*(X,\Xi,\F)^{\otimes 2} \to C_*(X,\Xi,\F)$ and converges towards $H_*(X, \F)$ as algebras. For $s,t \geq 0$, $E^2_{s,t} = H_s(X,H_t(\F))$ and the algebra structure is given, up to sign, by the intersection product on $X$ with coefficients in $H_t(\F)$.
    \end{itemize}

Moreover, this product corresponds in homology, via the Fibration Theorem, to the product $\mu_* : H_i(E) \otimes H_j(E) \to H_{i+j-n}(E)$ defined in \cite{GS07} (see Proposition \ref{prop : CS < CSDG}). In particular, if the fibration is the loop-loop fibration $\Omega X \hookrightarrow \mathcal{L}X \to X$, then $\CS$ corresponds to the Chas-Sullivan product.
\end{theorem}

This (re)proves that the product $\mu_*$, and in particular the Chas-Sullivan product, has the properties of Associativity, Commutativity, Neutral element, Functoriality and Spectral sequence.\\
 
We will define $\CS : H_*(X, C_*(F))^{\otimes 2} \to H_*(X, C_*(F))$ as the composition
$$
\begin{array}{ccccc}
H_i(X, C_*(F)) \otimes H_j(X, C_*(F)) & \overset{K^{top}}{\longrightarrow} &  H_{i+j}(X \times X , C_*(F \times F)) &\overset{\Delta_!}{\longrightarrow} &
H_{i+j-n}(X, \Delta^*C_*(F \times F))  \\
& & & \overset{\tilde{m}}{\longrightarrow} & H_{i+j-n}(X, C_*(F))
\end{array}
$$
with the Dold sign, ie $\CS(\gamma \otimes \tau) = (-1)^{n(n-|\tau|)} \tilde{m} \circ \Delta_! \circ K(\gamma \otimes \tau)$.

The Dold sign has been first introduced in \cite[Chapter VIII, \textsection 13.3]{Dold72} in order to make the intersection product in homology Poincaré dual to the unsigned cup-product, and used in \cite{Lau11} in order to make the Chas-Sullivan product graded commutative. It also serves such purposes in our setting.

The map $K^{top} : C_*(X, C_*(F))^{\otimes 2} \to C_*(X^2, C_*(F\times F))$ is a version of the cross product in the enriched Morse homology setting. The construction of this map is the object of Section \ref{section : DG Künneth formula}.

The map $\tilde{m}: C_{*}(X, \Delta^*C_*(F\times F)) \to C_{*}(X, C_*(F))$ is the morphism of complexes induced by the morphism of fibrations $m$. This will be constructed and studied in Section \ref{section : Functoriality on the coefficient}. In particular, we will prove that this map is a good notion of \emph{multiplication on the fiber} in our model.\\

\emph{Occurence of $\Ai$-modules and morphisms.}

Let us give an intuition and a motivation behind the use of $\Ai$-modules and morphisms in the case of the \emph{loop-loop fibration} $\Omega X \hookrightarrow \ls{X} \overset{\ev}{\to} X$. It is quite clear from \eqref{eq : Maurer-Cartan} that if $\varphi : \F \to \G$ is a morphism of DG-modules, then $\Tilde{\varphi} : C_*(X, \F) \to C_*(X, \G)$ defined by $\Tilde{\varphi}(\alpha \otimes x) = \varphi(\alpha) \otimes x$ is a morphism of complexes. However, the map $m_* : C_*(\Omega X \times \Omega X)^{ad} \to C_*(\Omega X)^{ad}$ induced by the concatenation $m:\Omega X \times \Omega X \to \Omega X$ is not a morphism of DG-modules over $C_*(\Omega X)$. The notation "\emph{ad}" indicates that the module structures are induced by the conjugation at the topological level. Indeed,

$$\forall \alpha, \beta \in \Omega X, \forall \gamma \in \Omega X, \ m((\alpha,\beta)\cdot\gamma) = m(\gamma^{-1} \alpha \gamma, \gamma^{-1} \beta \gamma) = \gamma^{-1} \alpha \gamma \gamma^{-1} \beta \gamma \neq \gamma^{-1} \alpha \beta \gamma = m(\alpha,\beta)\cdot\gamma.$$

However, we will prove that there exists a sequence of maps $$\{m_n : [0,1]^{n-1} \times (\Omega X\times \Omega X) \times \Omega X^{n-1} \to \Omega X\}_{n \geq 1}$$ that successively compensate (as more precisely defined in Definition \ref{defi : Ai morphism of topological module}) for the fact that the concatenation $m_1 := m$ commutes with the topological module structure induced by the fibrations only up to homotopy.
We will refer to such a sequence $(m_n)$ as an \textbf{$\Ai$-morphism of topological modules}. This yields a morphism of $\Ai$-modules $\{m_{n,*}\} :C_*(\Omega X \times \Omega X)^{ad} \to C_*(\Omega X)^{ad}$ over $C_*(\Omega X)$ (see Proposition \ref{prop : Ai topo induit Ai}). The concatenation being a core operation for the Chas-Sullivan product, we needed to prove that one is still able to extract a morphism of complexes $C_*(X,C_*(\Omega X \times \Omega X)^{ad}) \to C_*(X, C_*(\Omega X)^{ad})$ out of it. 
That is why we will prove \\

\textbf{Proposition B (\ref{prop : Morphisme de Ai module en morphisme de complexe})} : \emph{Let $\mathcal{A},\mathcal{B}$ be DG-modules over $C_*(\Omega X)$. Any morphism of $\Ai$-modules $\boldsymbol{\nu} : \mathcal{A} \to \mathcal{B}$ induces a morphism of complexes denoted $\tilde{\nu} : C_*(X,\mathcal{A}) \to C_*(X,\mathcal{B})$.}\\

In the context of a fibration, we will prove the following succession of result in Section \ref{section : Functoriality on the coefficient}. \begin{align*}
    &\varphi : E_1 \to E_2 \textup{ is a morphism of fibrations.}\\
     &\overset{\textup{\ref{lemme : morphism of fibrations induces coherent homotopy}}}{\Longrightarrow} \textup{There exists a \textbf{coherent homotopy} } \{\varphi_{n+1} : I^n \times  F_1 \times \Omega X^{n-1} \times \mathcal{P}_{\star \to X}X \to E_2\}\\ 
     & \textup{ that induces an } \Ai\textup{-morphism of topological modules } \{\varphi_{n+1}^{F_1} : I^n  \times F_1 \times \Omega X^{n} \to F_2\}.\\
    & \overset{\textup{\ref{prop : Ai topo induit Ai}}}{\Longrightarrow} \textup{The map } \varphi \textup{ induces a morphism } C_*(F_1) \to C_*(F_2) \textup{ of }\Ai\textup{-modules over } C_*(\Omega X). \\
    & \overset{\ref{prop : Morphisme de Ai module en morphisme de complexe}}{\Longrightarrow}\textup{The map } \varphi \textup{ induces a morphism } \tilde{\varphi} : C_*(X,C_*(F_1)) \to C_*(X,C_*(F_2)).\\
    & \overset{\ref{thm : morphisme induit commute avec iso}}{and} \ \textup{A coherent homotopy induces a chain homotopy between } \Psi_2 \circ \tilde{\varphi} \textup{ and } \varphi_* \circ \Psi_1.
\end{align*} 

The notion of coherent homotopy is defined in Definition \ref{defi : coherent homotopy}. We denoted $\Psi_1 : C_*(X,C_*(F_1)) \to C_*(E_1)$ and $\Psi_2 : C_*(X,C_*(F_2)) \to C_*(E_2)$ the quasi-isomorphisms given by the Fibration Theorem.
This constitutes the outline of a proof for the following theorem:

\begin{theorem}[Theorem \ref{thm : morphisme induit commute avec iso}]
Let $F_1 \hookrightarrow E_1 \to X$ and $F_2 \hookrightarrow E_2 \to X$ be two fibrations. If $\varphi: E_1 \to E_2$ is a morphism of fibrations, then there exists a morphism of complexes $\tilde{\varphi} : C_*(X, C_*(F_1)) \to C_*(X, C_*(F_2))$ such that the following diagram commutes up to chain homotopy

    \[
    \xymatrix{
    C_*(X, C_*(F_1)) \ar[r]^-{\Tilde{\varphi}} \ar[d]^{\Psi_1} & C_*(X, C_*(F_2)) \ar[d]_{\Psi_2} \\
    C_*(E_1) \ar[r]_{\varphi_*} & C_*(E_2),
    }
    \] where the morphisms $\Psi_1$ and $\Psi_2$ are the quasi-isomorphisms given by the Fibration Theorem.
\end{theorem}

This compatibility between $\Ai$-structures and DG Morse theory led us to explore the relation between these theories. We define the enriched Morse complex with coefficients in an $\Ai$-module over $C_*(\Omega X)$ (Section \ref{section : Morse homology with coefficients in a Ai-module}) and prove invariance properties using the canonical spectral sequences associated to such a complex (section \ref{section : Invariance Ai Morse}). We will moreover prove Proposition B for $\mathcal{A}$ and $\mathcal{B}$ any $\Ai$-modules over $C_*(\Omega X)$. \\

\emph{Künneth formula.}

Let $(X,\star_X)$ and $(Y, \star_Y)$ be pointed, oriented, closed, connected and smooth manifolds with Morse-Smale pairs $(f,\xi_X)$ and $(g,\xi_Y)$. Let $\F$ and $\G$ be DG modules over $C_*(\Omega X)$ and  $C_*(\Omega Y)$. We build in Section \ref{subsection : Morse data on a Cartesian product} what we will refer to as the \textbf{Künneth twisting cocycle} $\{m^K_{z,w} \in C_{|z|-|w|-1}(\Omega X \times \Omega Y), \ z,w \in \Crit(f\oplus g)\}$ adapted to the Cartesian product (see Definition \ref{defi : m^K}) such that there exists a homotopy equivalence

$$C_*(X \times Y, \mathcal{H}) \simeq C_*(X\times Y, m^K_{z,w}, \mathcal{H})$$ for any DG-module $\mathcal{H}$ over $C_*(\Omega X\times \Omega Y)$ (see Proposition \ref{prop : dg Kunneth m^0 and m^1 quasi-iso}).
\\

From there, let us present two versions of cross products on enriched Morse complexes :

$\bullet$ We can use the Serre diagonal $\Delta :  C_*(\Omega X \times \Omega Y)  \to C_*(\Omega X) \otimes C_*(\Omega Y)$ to define a $C_*(\Omega(X \times Y))$-module structure on the tensor product $\F \otimes_{\Z} \G$ such that $$ \deffct{K^{alg}}{C_*(X, \F) \otimes C_*(Y, \G)}{C_*(X \times Y, m^K_{z,w}, \F \otimes \G)}{(\alpha \otimes x) \otimes (\beta \otimes y)}{(-1)^{|\beta||x|}(\alpha \otimes \beta) \otimes (x,y)}$$ is an isomorphism of complexes.

$\bullet$ If $F \hookrightarrow E \to X$ and $G \hookrightarrow E' \to Y$ are two fibrations with respective transitive lifting functions $\Phi_X$ and $\Phi_Y$, then the fibration $F \times G \hookrightarrow E \times E' \to X \times Y$ is naturally endowed with the transitive lifting function $(\Phi_X, \Phi_Y)$. Therefore, to study fibrations, if $\F = C_*(F)$ and $\G = C_*(G)$, it is natural to consider the complex $C_*(X\times Y, m^K_{z,w}, C_*(F \times G))$ where the module structure is defined by $(\Phi_X, \Phi_Y)$. 

That is why if $F$ and $G$ are topological spaces each endowed with a topological module structure over respectively $\Omega X$ and $\Omega Y$, we define

$$\deffct{K^{top}}{C_*(X, C_*(F)) \otimes C_*(Y, C_*(G))}{C_*(X \times Y, m^K_{z,w}, C_*(F \times G))}{(\alpha \otimes x) \otimes (\beta \otimes y)}{(-1)^{|\beta||x|}(\alpha , \beta) \otimes (x,y)}.$$

To summarize, we have the following theorem:

\begin{theorem}[Theorem \ref{thm : Künneth formula}]
    Let $\F$ and $\G$ be DG modules over $C_*(\Omega X)$ and  $C_*(\Omega Y)$. There exists a Künneth twisting cocycle $\{m^K_{z,w} \in C_{|z|-|w|-1}(\Omega X \times \Omega Y), \ z,w \in \Crit(f\oplus g)\}$ which computes the same homology as the Barraud-Cornea cocycle and such that :

$$K^{alg} : C_*(X, \F) \otimes C_*(Y,\G)  \to C_*(X \times Y, m^K_{z,w}, \F \otimes \G)$$ is an isomorphism of complexes.

If $\F = C_*(F)$ and $\G = C_*(G)$ are complexes of cubical chains of topological spaces equipped with topological module structures respectively over $\Omega X$ and $\Omega Y$, then $$K^{top}: C_*(X, C_*(F)) \otimes C_*(Y,  C_*(G))\to C_*(X \times Y, m^K_{z,w}, C_*(F \times G))$$ is a quasi-isomorphism of complexes.
\end{theorem}

In particular, this will prove that $H_*(X \times Y, C_*(F \times G))$ is isomorphic to $H_*(X \times Y, C_*(F) \otimes C_*(G)) $. In this paper, we will mainly use the cross product $K := K^{top}$ since we will heavily rely on considerations on fibrations to define the product $\CS.$ \\

This result is not a consequence of the Künneth formula from the "classical" Morse theory.
Consider a Morse-Smale pair $(f,\xi_X)$ on $X$ and a Morse-Smale pair $(g,\xi_Y)$ on $Y$. Then $F = f \oplus g : X \times Y \to \R$ is a Morse function and $\xi = (\xi_X \oplus \xi_Y)$ is an adapted pseudo-gradient. The identification between the moduli spaces of Morse trajectories $$\traji{F}{(x,y),(x',y')} \simeq \traji{f}{x,x'} \times \traji{g}{y,y'}$$ is only true if either the moduli spaces are empty, $x=x'$ or $y=y'$. In "classical" Morse theory, this is all we have to consider since the differential only counts trajectories between critical points of consecutive indexes and, if $|x|+|y| = |x'|+|y'|+1$, then one of the conditions above is true. However, in DG Morse theory, we have to consider all pairs of critical points in order to construct the twisting cocycle $m_{(x,y),(x',y')}$ and this identification no longer holds.\\

\emph{Thom isomorphism in DG Morse theory.}

Recall that $\ev : \ls{X} \to X$ denotes the evaluation at the basepoint.
Since the product $\CS$ is meant to generalize the Chas-Sullivan product, we constructed it so that the diagram 

$$  \xymatrix
    @C=20pt
    @R=1cm
    {
H_i(X,\F)^{\otimes 2} \ar[r]^-{K^{top}} \ar[d]^-{\sim} & H_{i+j}(X \times X, \F^2)  \ar[d]^-{\sim} \ar[r]^-{\Delta_!}&    H_{i+j-n}(X, \F^2) \ar[r]^-{\tilde{m}} \ar[d]^-{\sim} & H_{i+j-n}(X, \F) \ar[d]^-{\sim}\\
H_i(\mathcal{L}X)^{\otimes 2} \ar[r]^-{\EZ} & H_{i+j}(\mathcal{L}X \times \ls{X}) \ar[r]^-{\Delta_!}& H_{i+j-n}(\mathcal{L} X \ftimes{\ev}{\ev} \mathcal{L}X) \ar[r]^-{m_*} & H_{i+j-n}(\mathcal{L}X) 
}
$$   
commutes, where we denoted $\F = C_*(\Omega X)$ and $\F^2 = C_*(\Omega X \times \Omega X)$, the vertical arrows are defined by the Fibration Theorem and $\EZ : H_i(\mathcal{L}X)^{\otimes 2}  \to H_{i+j}(\mathcal{L}X \times \ls{X})$ is the Eilenberg-Zilber map. The map $\Delta_! : H_{i+j}(\mathcal{L}X \times \ls{X}) \to H_{i+j-n}(\mathcal{L} X \ftimes{\ev}{\ev} \mathcal{L}X)$ is the "Gysin" or "Umkehr" map induced by the diagonal $\Delta : X \to X\times X.$

It is a simple proof (Lemma \ref{lemme : Compatibility K and Fibration Theorem}) that the first square commutes and Theorem B will conclude that, since the concatenation $m : \ls{X} \ftimes{\ev}{\ev} \ls{X}\to \ls{X}$ is a morphism of fibrations, the last square also commutes. In order to complete this correspondence, we prove a DG version of the Pontryagin-Thom construction stating that $\Delta_!$ corresponds to the intersection product :

\begin{theorem}[Theorem \ref{thm : DG Thom iso}]
Let $Y^{n-k} \overset{\varphi}{\hookrightarrow} Z^n$ be an embedding of closed manifolds with tubular neighborhood $U$. Let $F \hookrightarrow E \overset{\pi}{\to} Z$ be a fibration and $F \hookrightarrow E_Y \to Y$ be the pullback fibration by $\varphi$. Define a Gysin map $\varphi_! : H_*(E) \to H_{*-k}(E_Y)$ as the composition of a Pontryagin-Thom collapse map $\tau_* : H_*(E) \to H_*(E_U,E_{\partial U})$ on $E_U = \pi^{-1}(U)$ with the Thom isomorphism $u_* :H_*(E_U, E_{\partial U}) \to H_{*-k}(E_Y)$.

    The following diagram commutes up to the sign $(-1)^{ki +k}$.

    \[
    \xymatrix{
    H_i(Z,C_*(F)) \ar[r]^-{\varphi_!} \ar[d]_{\Psi_E}^{\sim}& H_{i-k}(Y,\varphi^*C_*(F)) \ar[d]^{\Psi_{E_Y}}_{\sim} \\
    H_i(E) \ar[r]_-{\varphi_!} & H_{i-k}(E_Y).
    }
    \]  
\end{theorem}

To the best of the author's knowledge, the setting in which the Gysin map $\varphi_! : H_*(E_1) \to H_{*-k}(E_2)$ is classically considered is when $E_2 \hookrightarrow E_1$  are Hilbert manifolds with $E_2$ of finite codimension in $E_1$ so that there always exists a tubular neighborhood $V$ of $E_2$ in $E_1$, unique up to isotopy, in which to consider the Thom isomorphism (see \cite[Section 3]{ChOa15} and \cite{Lang95}). In a similar fashion as \cite{hingston2022product} did for the particular case where the fibration is $\ls{Y} \times \ls{Y} \to Y \times Y = Z$ and $\varphi : Y \to Z$ is the diagonal map, we here avoid infinite dimensional arguments by using the fibration $E \overset{\pi}{\to} Z$ to pull back a tubular neighborhood $U$ of $Y$ in $Z$ which is finite dimensional and prove that there exists a Thom isomorphism $H_*(E_U, E_{\partial U}) \to H_{*-k}(E_Y)$ in this case, where $E_U = \pi^{-1}(U)$.\\

\emph{Path-product.} 

Theorem B, Theorem C and Theorem D have their own interest independently of Theorem A and may be used to describe other operations involving the idea of combining operations on the base and on the fiber. In this spirit, we also give a Morse-theoretic description of the path product on $H_*(\mathcal{P}_{X \to X} Y)$, the homology of the space of paths in $Y$ starting and ending in $X$, of degree $-n$ where $(X^n,\star)$ is a pointed, oriented, closed and connected manifold included in a topological space $Y$. This space and this product have been studied in \cite{Max23} when $Y$ is a closed manifold. The construction also relies on "intersecting" paths to make them concatenable and then concatenating.

Consider the fibration $\Omega Y \hookrightarrow \mathcal{P}_{X \to X} Y \overset{(\ev_0,\ev_1)}{\to} X^2$ on the space of paths in $Y$ that start and end on $X$. We state here a theorem in preparation that will be proved in a future paper :

\begin{theorem}
The concatenation $m : \mathcal{P}_{X \to X} Y \ftimes{\ev_1}{\ev_0} \mathcal{P}_{X \to X} Y \to \mathcal{P}_{X \to X} Y$ is a morphism of fibrations and defines a product $$\textup{PP}_{DG} : H_*(X^2,C_*(\Omega Y))^{\otimes 2}) \to H_*(X^2,C_*(\Omega Y))$$ of degree $-n$ that consists in intersecting endpoints of a chain of such paths with the startpoints of another chain, concatenating them and then forgetting the concatenation point. This can be written as the composition 
\begin{align*}
    H_*(X^2,C_*(\Omega Y))^{\otimes 2} \overset{K}{\to} H_*(X^4, C_*(\Omega Y^2)) \overset{D_!}{\to}H_{*-n}(X^3,D^*C_*(\Omega Y^2)) & \overset{\Tilde{m}}{\to} H_{*-n}(X^3, p^*C_*(\Omega Y))\\
    & \overset{p_{*}}{\to} H_{*-n}(X^2, C_*(\Omega Y))
\end{align*}

with the Dold sign. Here $D : X^3 \to X^4$, $D(a,b,c) = (a,b,b,c)$ and $p : X^3 \to X^2$, $p(a,b,c) = (a,c)$,

This product is associative up to the sign $(-1)^n$ and admits a neutral element.
\end{theorem}

\textbf{Structure of this paper :} Since we use quite often in this paper tools and techniques that have been developed and introduced in \cite{BDHO23}, we present in Section \ref{section : Morse homology with DG coefficients} a non-exhaustive (and sometimes heuristic) summary of definitions and constructions that we will need in the rest of the paper.
We prove Theorem D in Section \ref{section : DG Thom iso}.
In Section \ref{section : Morse homology with coefficients in a Ai-module}, we define the Morse complex with coefficients in an $\Ai$-module over $C_*(\Omega X)$ and prove that the homology of this complex is independent of the choices made to define it. 
We then study in Section \ref{section : Functoriality on the coefficient} the maps induced by an $\Ai$-morphism of modules between the coefficients and prove Theorem B.
 
We prove Theorem C in Section \ref{section : DG Künneth formula}.
We then use the results of the previous sections to define $\CS$ and prove Theorem A in Section \ref{section : Morse description and generalization of the Chas-Sullivan product}.

We then discuss in Section \ref{section : Further direction} further directions in which we could expand this paper. \\ 

\emph{Acknowledgments.} I would like to express my gratitude to my advisors Mihai Damian and Alexandru Oancea for their many advices and continuous support through the writing of this article. I am also very grateful to Colin Fourel for a lot of very useful discussions.

\setlength{\parindent}{0pt}

\section{Recollections on Morse homology with DG coefficients}\label{section : Morse homology with DG coefficients}

Throughout this paper, $X$ will denote an oriented, closed and connected manifold. Let $\star \in X$ be a preferred point in $X$ and $(f,\xi)$ be a Morse-Smale pair on $X$.
This section is a brief and non-exhaustive summary of \cite{BDHO23} aimed at presenting the main tools used in the present paper.

\subsection{Morse homology with DG coefficients}\label{subsection : Enriched Morse complex}

For any topological space $V$, we can consider its \textbf{cubical complex} $C_*(V)$, the $\Z$-module spanned by continuous maps $\sigma : [0,1]^k \to V$ endowed with the differential \begin{equation*}\label{def diff cubique}
    \partial \sigma = \sum_{i=1}^{k} (-1)^{i-1} (\left. \sigma \right \lvert_{x_i=1} - \left. \sigma \right \lvert_{x_i=0}),
\end{equation*}
modded out by the $\Z$-module spanned by the maps $\sigma : [0,1]^k \to V$ that factor through a face of the cube $[0,1]^k$. We call these chains \textbf{degenerate}.

Consider $\Omega X = \{ \gamma : [0,a] \to X \ \textup{continuous}, \ a \in [0, + \infty) , \ \gamma(0)=\gamma(a)\}$ the \textbf{space of Moore loops} based at $\star$. The concatenation is strictly associative on this space. Define the \textbf{Pontryagin product} on $C_*(\Omega X)$ by $$\sigma \otimes \tau  \mapsto \left(\sigma.\tau : (x,y) \mapsto \sigma(x)\#\tau(y) \right)$$ for every continuous maps $\sigma : [0,1]^k \to \Omega X$ and $\tau : [0,1]^l \to \Omega X$, where $\#$ denotes the concatenation in $\Omega X$. It extends to a product $\mu : C_i(\Omega X) \otimes C_j(\Omega X)  \to  C_{i+j}(\Omega X)$ which defines a DGA structure on $C_*(\Omega X)$.

\begin{defi}

A \textbf{twisting cocycle} is a family $\{ m_{x,y} \in C_{|x|-|y|-1}(\Omega X),  \ x,y \in \Crit(f)\}$ that satisfies the Maurer-Cartan equation \begin{equation}
\partial m_{x,y} = \sum_{|x|>|z|>|y|} (-1)^{|x| - |z|} m_{x,z}m_{z,y}.
\label{MC}
\end{equation}
\end{defi}

\begin{defi}
    We call \textbf{DG system} a right differential graded module $(\F_*, \partial_{\F})$ on $C_*(\Omega X)$.
\end{defi}
\begin{defi}

Let  $\F$ be a DG system and let $(m_{x,y})$ be a twisting cocycle.
The Morse complex of $X$ with coefficients in $\F$ is defined by 

$$C_*(X,m_{x,y},\F) :=  \F_* \otimes \Z \Crit(f).$$

with the differential \[\partial(\alpha \otimes x) = \partial \alpha \otimes x + (-1)^{|\alpha|} \left ( \sum_{|y|<|x|} \alpha\cdot m_{x,y} \right ) \otimes y.\]

The resulting homology depends on the choice of twisting cocycle. We denote by $H_*(X,m_{x,y},\F)$ the homology of this complex.
We will refer to a complex with coefficient in a DG system as an \textbf{enriched Morse complex} and its homology as \textbf{enriched Morse homology}.

\end{defi}

In the next section, we describe how the authors define a particular twisting cocycle called \textbf{Barraud-Cornea twisting cocycle} as the construction follows the seminal paper \cite{BC07}. It carries homological information on the moduli space of Morse trajectories $\trajb{x,y}$. This construction depend on choices but the resulting homology will not depend on them. We therefore denote $H_*(X,\F)$ the homology of an enriched Morse complex $C_*(X,m_{x,y},\F)$ constructed with a Barraud-Cornea twisting cocycle.

\subsection{Twisting cocycles defined by DG Morse data}\label{subsection : Construction of the twisting cocyle and representing systems}

Choose a Morse-Smale pair $(f,\xi)$ on $X$ and $o$, an orientation of the unstable manifolds $W^u(x)$ for each $ x \in \Crit(f)$.

\begin{defi}
    For any pair of critical points $x,y \in \textup{Crit}(f)$, define the \textbf{moduli space of broken Morse trajectories} 

$$\trajb{x,y} := \traj{x,y} \cup \bigcup_{\substack{k \geq 1 \\z_1, \dots, z_k \in \Crit(f)}} \traj{x,z_1} \times \traj{z_1,z_2} \times \dots \times \traj{z_k,y}.$$
    
\end{defi}

We refer to \cite[Section 3.2]{AD14} to prove that with its natural topology, $\trajb{x,y}$ is an orientable compact manifold with boundary and corners of dimension $|x|-|y|-1$ whose interior is $\traj{x,y}$.

Intuitively, the twisting cocycle $m_{x,y} \in C_{|x|-|y|-1}(\Omega X)$ is constructed by evaluating in the based loop space $\Omega X$ a well-suited representative of the fundamental class of $\trajb{x,y}$ for any $x,y \in \Crit(f)$.

\begin{defi}\label{defi : representing chain system traj}
    A \textbf{representing chain system} of the moduli spaces of trajectories is a collection $\{s_{x,y} \in C_{|x|-|y|-1}(\trajb{x,y}) , \ x,y \in \textup{Crit}(f)\}$ such that 

    \begin{enumerate}
        \item each $s_{x,y}$ is a cycle relative to the boundary and represents the fundamental class $[\trajb{x,y}]$.\\
        \item each $s_{x,y}$ satisfies 

        \begin{equation}\label{eq : relation sys repr}
            \partial s_{x,y} = \sum_z (-1)^{|x|-|z|} s_{x,z} \times s_{z,y}.
        \end{equation}
    \end{enumerate}
\end{defi}

Such a representing chain system always exists and is inductively constructed. It is nonetheless not unique (although all representing chain system are homologous in a sense defined in \cite[Proposition 5.2.8]{BDHO23}). \\

\begin{defi}\label{defi : Evaluation maps}
    We say that a family of maps $q_{x,y} : \trajb{x,y} \to \Omega X$ for $x,y \in \Crit(f)$ is a family of \textbf{evaluation maps} if it satisfies 

\begin{enumerate}
        \item \underline{The concatenation relation} : For any $(\lambda, \lambda') \in \trajb{x,z} \times \trajb{z, y}\subset \partial \trajb{x,y}$, 
        
        \begin{equation}\label{eq : Concatenation relation evaluation map}
            q_{x,y}(\lambda, \lambda') = q_{x,z}(\lambda) \# q_{z, y}(\lambda').
        \end{equation}
        
    \item \underline{Compatibility with lifting} : Choose for each critical points $x$ a preferred lift $\tilde{x}$ in $\tilde{X}$, the universal cover of $X$. For any $|x|=|y|+1$ and $\lambda \in \trajb{x,y}$, denote $g=[q_{x,y}(\lambda)] \in \pi_1(X)$ the homotopy class of the loop $q_{x,y}(\lambda)$. We say that $q_{x,y}$ is compatible with lifting if the lift $\tilde{\lambda}$ of $\lambda$ starting from $\tilde{x}$ ends in $g\tilde{y}$.
    
    \end{enumerate}
\end{defi}

Here is the construction of a family of evaluation maps given in \cite[Lemma 5.9]{BDHO23} that we will use as a blueprint in the rest of the paper when it will come to construct families of evaluation maps:

First, choose 

\begin{enumerate}
    \item A tree $\mathcal{Y}$ in $X$ whose vertices are the critical points of $f$ and whose root is the basepoint $\star$.
    \item $\theta : \Xq \to X$, a homotopy inverse of the canonical projection $p : X \to \Xq$ such that $\theta([\mathcal{Y}]) = \star$.
\end{enumerate}

Define a parametrization map $\Gamma_{x,y}$ parametrizing a trajectory $\lambda$ into a path. Here, we parametrize maps using the values of the Morse function $f$.
$$\deffct{\Gamma_{x,y}}{\trajb{x,y}}{C^0([0, f(x)-f(y)], X)}{\lambda}{t \mapsto (\lambda \cap f^{-1}(f(x)-t))}.$$
This map is well-defined since $f$ is strictly decreasing along the (broken) trajectory $\lambda$. This parametrization map is clearly compatible with concatenations, ie if $(\lambda, \lambda') \in \trajb{x,z} \times \trajb{z,y} \subset \partial \trajb{x,y}$, then $$\Gamma_{x,y}(\lambda,\lambda') = \Gamma_{x,z}(\lambda)\#\Gamma_{z,y}(\lambda'). $$

The evaluation maps $q_{x,y} : \trajb{x,y} \to \Omega X$ are then defined by $$q_{x,y} = \theta \circ p \circ \Gamma_{x,y}.$$ Here, modding out by the tree $\mathcal{Y}$ turns the path $\Gamma_{x,y}(\lambda)$ into a loop.
The twisting cocycle is defined by $m_{x,y} = q_{x,y,*}(s_{x,y}) \in C_{|x|-|y|-1}(\Omega X).$\\

To construct the Barraud-Cornea cocycle, we needed six choices : 

\begin{itemize}
    \item[$\bullet$] A Morse function $f : X \to \R$.\\
    \item[$\bullet$] A pseudo-gradient $\xi$ adapted to $f$. \\
    \item[$\bullet$] An orientation $o$ of the unstable manifolds of $(f,\xi)$.\\
    \item[$\bullet$] A representing chain system $\{s_{x,y}\}$\\
    \item[$\bullet$] A tree $\mathcal{Y}$\\
    \item[$\bullet$] A homotopy inverse $\theta$ of the projection $p : X \to \Xq$.
\end{itemize}

\begin{figure}[!ht]
    \centering
    \includegraphics[width=9cm, height=3cm]{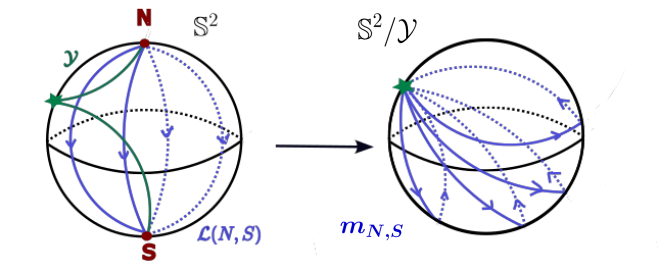}\\
    \caption{Example of a Barraud-Cornea twisting cocycle on $\mathbb{S}^2$.}
    \label{Im : twisting cocycle on S^2}
\end{figure}

\begin{defi}
    Denote $\Xi = (f, \xi, \mathcal{Y}, o, s_{x,y}, \theta)$  the choices that have been made in order to construct such a twisting cocycle. We will refer to such a set $\Xi$ as \textbf{DG Morse data} on $(X, \star_X)$. We will omit to mention the basepoint $\star_X$ if it is not necessary.
\end{defi}

Given a set of DG Morse data  $\Xi$, we will denote $C_*(X, \Xi, \F) := C_*(X, m_{x,y}, \F) = \F_* \otimes \Z \Crit(f)$ the enriched Morse complex with the Barraud-Cornea cocycle constructed with the data $\Xi$. \cite{BDHO23} proves that the homology of this complex does not depend on the set of data $\Xi$. More precisely, given two sets of data $\Xi_0$ and $\Xi_1$, \cite{BDHO23} builds a \textbf{continuation data} $\Xi$, a \textbf{continuation map} $$\Psi^{\Xi} : C_*(X,\Xi_0,\F) \to C_*(X,\Xi_1,\F)$$ and proves the following theorem.

\begin{thm}\label{thm : continuation morphism Psi_01}\cite[Theorem 6.3.1]{BDHO23}

    1) Given two sets $\Xi_0$ and $\Xi_1$ of DG Morse data on $X$, the continuation map $$\Psi^{\Xi} : C_*(X,\Xi_0,\F) \to C_*(X,\Xi_1,\F)$$ is a homotopy equivalence and its chain homotopy type does only depend on $\Xi_0$ and $\Xi_1$. The map $\Psi^{\Xi}$ is in particular a quasi-isomorphism. \\
    
    2) Given another set of data $\Xi_2$ on  $X$ and denoting $\Psi_{ij}$ the continuation map between the data $\Xi_i$ and $\Xi_j$, then $\Psi_{00}$ is homotopic to the identity and $\Psi_{02}$ is homotopic to $\Psi_{12}\circ \Psi_{01}$. In particular, in homology

    $$\Psi_{00} = \textup{Id} \ \textup{and} \ \Psi_{12}\circ \Psi_{01} = \Psi_{02}.$$
\end{thm}

\subsection{Fibration Theorem}\label{subsection : Fibration Theorem}

In this section, we will introduce one of the main result of \cite{BDHO23}, the Fibration Theorem, that will connect our construction on the enriched Morse complex of the base $X$ of a fibration with previous constructions that use the singular complex of the total space $E$ of this fibration. 

\subsubsection{Lifting functions}

The notion of lifting function is a tool that we will use all along this paper to study fibrations and enriched Morse complexes. 

\begin{notation}\quad

\begin{itemize}
    \item[$\bullet$] For any topological space $A$ and $x,y \in A$, we will denote $$\mathcal{P}_{x \to y} A := \left\{ \gamma \in C^0([0,a], A), \ \gamma(0) = x , \ \gamma(a) = y \right\}$$
    the space of Moore paths that start in $x$ and end in $y$. We will respectively denote $\ev_0$ and $\ev_1$ the evaluation at the startpoint and the endpoint of a Moore path.\\
    \item[$\bullet$] If $B,C \subset A$, $b\in B$ and $c \in C$ we use the notations $$\mathcal{P}_{B \to C} A := \displaystyle \bigcup_{b \in B, c \in C} \mathcal{P}_{b \to c} A, \quad \mathcal{P}_{b \to C} A := \mathcal{P}_{\{b\} \to C} A \ \textup{and} \ \mathcal{P}_{B \to c} A := \mathcal{P}_{B \to \{c\}} A.$$
\end{itemize}
\end{notation}

A \textbf{Hurewicz fibration} $ E \overset{\pi}{\to} X$ is a continuous and surjective map that has the \textbf{homotopy lifting property} with respect to all spaces : For any space $B$ and any maps $H : B \times [0,1] \to X$ and $f : B \to E$ such that $\pi \circ f = H(\cdot, 0)$, there exists a lift $\tilde{H}$

$$
\xymatrix{
B \ar[r]^f \ar[d] & E \ar[d]^{\pi} \\
B \times [0,1] \ar[r]^H \ar@{.>}[ur]^{\tilde{H}} & X.
}
$$

This property is equivalent to the existence of a map called \textbf{lifting function} $\phi : E \ _{\pi}\!\times_{\ev_0} \mathcal{P} X \to \mathcal{P}E$ such that $\pi \circ \phi = pr_2$ and $\ev_0 \circ \phi = pr_1.$

A \textbf{transitive lifting function} is a map $$\Phi : E \ftimes{\pi}{\ev_0} \mathcal{P} X \to E$$ such that $\pi \circ \Phi = \ev_1 \circ pr_2$, such that for any $e\in E$,  $\Phi(e,\pi(e)) = e$ where $\pi(e)$ is the constant path in $\pi(e) \in X$ and such that for every $\gamma, \delta \in \mathcal{P} X$ such that $\ev_1\gamma= \ev_0\delta$,
$$\Phi(\Phi(e,\gamma),\delta) = \Phi( e, \gamma \# \delta).$$

It is proven in \cite[Proposition 5.5]{DK69}  that every fibration is fiber homotopy equivalent to a fibration that admits a transitive lifting function and we will therefore assume that any fibration is equipped with a transitive lifting function.\\

Denote $F = \pi^{-1}(\star)$\textbf{ the fiber} of the fibration.
A transitive lifting function induces a right $C_*(\Omega X)$-module structure on $C_*(F)$ given by $\alpha \cdot \sigma = \Phi_*(\alpha \otimes \sigma)$.

\subsubsection{Statement}

One of the main results of \cite{BDHO23} is the following theorem 

\begin{thm}(\cite[Theorem 7.2.1]{BDHO23})
    Let $F \hookrightarrow E \to X$ a (Hurewicz) fibration, $\Phi : E \ftimes{\pi}{\ev_0} \mathcal{P}X \to E$ a transitive lifting function and $\Xi$ a set of DG Morse data on $X$. Then there exists a quasi-isomorphism $$\Psi_E : C_*(X,\Xi,C_*(F)) \to C_*(E),$$
    where $C_*(F)$ is endowed with the $C_*(\Omega X)$-module structure induced by $\Phi$.
\end{thm}

\begin{notation}
    For any fibration $F \hookrightarrow E \to X$ over a pointed oriented, closed and connected manifold $(X, \star)$, unless otherwise specified, we will denote $\Psi_E : H_*(X, C_*(F)) \to H_*(E)$ the isomorphism given by the Fibration Theorem.
\end{notation}

\begin{rem}(\cite[Remark 7.3.5]{BDHO23})
    This theorem has a version for manifolds with boundary. Let $\partial X = \partial_+ X \cup \partial_-X$, where $\partial_+ X$ is the subset of $\partial X$ where the pseudo-gradient points outwards and $\partial_- X$ the subset where it points inwards. 
    Then, if we denote $E_+ = \pi^{-1}(\partial_+ X)$,
    there exists a quasi-isomorphism 

    $$C_*(X, \partial_+ X, C_*(F)) \to C_*(E,E_+).$$
\end{rem}

\subsubsection{Tools for the proof}

We will need in some proofs of this paper to understand how this quasi-isomorphism is constructed. Therefore, let us lay out the tools necessary for the proof of this theorem.

We first consider $F' \hookrightarrow E' \to X/\mathcal{Y}$ the pullback fibration by $\theta : X/\mathcal{Y} \to X$, the chosen homotopy inverse of the projection $p : X \to X/\mathcal{Y}$.
The family of evaluation maps $\left\{q'_{x,y} = p \circ \Gamma_{x,y} : \trajb{x,y} \to \Omega (\Xq) \right\}$ define a twisting cocycle $m'_{x,y} \in C_*(\Omega (\Xq))$ such that $\theta_*(m'_{x,y}) = m_{x,y}$ and the $C_*(\Omega (\Xq))$-module structure on $C_*(F')$ is given by $\alpha \cdot \sigma := \alpha \cdot \theta_*(\sigma)$.

The map $\Psi_E$ is then defined as a composition of three maps 

$$C_*(X, C_*(F)) \to C_*(X,C_*(F')) \overset{\Psi_{E'}}{\to} C_*(E') \to C_*(E).$$

The only real consideration is on the middle map since the first one is just an identification $\alpha \cdot m'_{x,y} = \alpha \cdot \theta_*(m'_{x,y}) = \alpha \cdot m_{x,y} \in C_*(F')=C_*(F)$ and the last one is a homotopy equivalence given by $\theta.$

The map $\Psi_{E'} : C_*(X,C_*(F')) \to C_*(E')$ is essentially constructed by using the transitive lifting function to lift the cellular decomposition of $X$ given by its unstable manifolds. Let us be more precise.

\begin{defi}\label{defi : Latour cells}

Define for every $x\in \textup{Crit}(f)$, its \textbf{Latour cell} $$\wb{u}{x} = W^u(x) \cup \displaystyle \bigcup_{|y|<|x|} \trajb{x,y} \times W^u(y).$$ It has been proven in \cite[Proposition 2.11]{Lat94}, in \cite[Theorem 3]{qin10} and  \cite[section 4.9.c]{AD14} that with its natural topology, $\wb{u}{x}$ is a compact manifold with boundary and corners of dimension $|x|$ whose interior is $W^u(x)$ and which is homeomorphic to the closed disk $\overline{D}^{|x|}$. Latour cells induce a cellular decomposition of $X$.
\end{defi}

\begin{defi}\label{defi : for the fibration thm}

We define the morphism $\Psi_{E'}$ in several steps :\\

$\bullet$ There exists a family of evaluation maps $q_x : \wb{u}{x} \to \mathcal{P}_{\star \to X/\mathcal{Y}} X/\mathcal{Y} $ constructed by parametrizing the gradient lines and projecting them into $X/\mathcal{Y}$ (see \cite[Lemma 7.3.3]{BDHO23}).

$\bullet$ Let $(s_{x,y})$ be a representing chain system for the Morse moduli spaces in $B$. A \textbf{compatible representing chain system of the Latour cells} in $B$ as defined in \cite{BDHO23} is a collection $\{s_x \in C_{|x|}(\wb{u}{x}) \ \lvert \  x \in \textup{Crit}(f) \}$ of chains such that :

    \begin{enumerate}
        \item each $s_x$ is a cycle relative to the boundary and represent the fundamental class $[\wb{u}{x}, \partial \wb{u}{x}]$.\\
        \item each $s_x$ satisfies $\partial s_x = \displaystyle \sum_y s_{x,y} \times s_y$ with the product of chains defined via the inclusions $\trajb{x,y} \times \overline{W^u}(y) \subset \partial \overline{W^u}(x) \subset \overline{W^u}(x).$ 
    \end{enumerate}

Such a system always exists and is constructed inductively.

$\bullet$ Define $m_x = q_{x,*}(s_x) \in C_{|x|}(\mathcal{P}_{\star \to B/\mathcal{Y}} B/\mathcal{Y})$ for each $x \in \Crit(f)$. This family satisfies $$\partial m_x = \sum_y m'_{x,y}m_y.$$

$\bullet$ The map $\Psi_{E'} : C_*(X,C_*(F')) \to C_*(E')$ is then defined by $\Psi_{E'}(\alpha \otimes x) = \Phi_{*}(\alpha \otimes m_x).$

\end{defi}

To conclude the proof, one remarks that $\Psi$ induces a morphism of spectral sequences between the spectral sequence defined earlier that converges towards $H_*(X, C_*(F'))$ and the Leray-Serre spectral sequence that converges towards $H_*(E')$, and proves that this is an isomorphism between the second pages.

\subsection{Direct and shriek maps}\label{subsection : Direct and shriek maps}

Let $\varphi : (X^n,\star_X) \to (Y^m,\star_Y)$ be a continuous map between pointed, oriented, closed and connected manifolds. We will assume $\varphi(\star_X) = \star_Y.$

Let $\F$ be a DG right-module over $C_*(\Omega Y)$ and $\varphi^* \F$ be its pullback by $\varphi$, ie $\F$ endowed with the DG right-module over $C_*(\Omega X)$, 

$$\forall \alpha \in \F, \ \forall \gamma \in C_*(\Omega X), \ \alpha \cdot \gamma := \alpha \cdot \underbrace{\varphi_*(\gamma)}_{\in C_*(\Omega Y)}$$

\begin{thm}
\cite[Theorem 8.1.1]{BDHO23}
There exist two maps respectively called \textbf{direct map} and \textbf{shriek map}, 

$$\varphi_* : H_*(X, \varphi^*\F) \to H_*(Y, \F) \ \textup{and} \ \varphi_! : H_*(Y, \F) \to H_{*+n-m}(X, \varphi^* \F)$$ 

that satisfy the following properties :

\begin{enumerate}
    \item \textbf{Identity:} $\textup{Id}_* = \textup{Id}_!= \textup{Id}$ in homology.
    \item \textbf{Composition:} Let $\psi : Z^k \to X $ be a continuous map and $\F$ be a DG right-module over $C_*(\Omega Y)$ . Then $$(\varphi \psi)_* = \varphi_* \psi_* : H_*(Z, \varphi^*\psi^*\F) \to H_*(X, \varphi^*\F) \to H_*(Y, \F)  $$ and $$(\varphi \psi)_! = \psi_! \varphi_! : H_*(Y, \F) \to H_{*+n-m}(X, \varphi^*\F) \to H_{*+k-m}(Z, \varphi^*\psi^*\F).$$
    \item \textbf{Homotopy:} Two homotopic maps induce the same direct and shriek maps in homology.
    \item \textbf{Spectral sequence :} The morphisms $\varphi_*$ and $\varphi_!$ in homology are limits of morphisms between spectral sequences associated to the corresponding enriched complexes, given at the second page by 
    $$\varphi_{p,*} : H_p(X, \varphi^* H_q(\F)) \to H_p(Y, H_q(\F)) \ \textup{and} \ \varphi_{p,!} : H_p(Y,H_q(\F)) \to H_{p+n-m}(X, \varphi^* H_q(\F)). $$
\end{enumerate}
\end{thm}

There are two equivalent definitions for these maps (see \cite[Sections 9 and 10]{BDHO23}).

\section{Thom isomorphism in DG Morse theory. Proof of Theorem D}\label{section : DG Thom iso}

Let $(Z^n,\star)$ be a pointed, oriented, closed, connected manifold and $F \hookrightarrow E \overset{\pi}{\to} Z$ a fibration endowed with a transitive lifting function $\Phi : E \ftimes{\pi}{\ev_0} \mathcal{P}Z \to E$.

Let $\varphi : Y^{n-k} \hookrightarrow Z^n$ be an embedding of a closed manifold and $i_U : U \to Z$ be an embedding of the total space of the normal bundle $p : U \to Y$ of $\varphi$. Denote $[\tau] \in H^k(U,\partial U)$ the Thom class of this disk bundle and choose $\tau \in C^k(U,\partial U)$ a representative. We will see $\varphi$ as an inclusion since the diffeomorphism $Y \to \varphi(Y)$ will not play a role in this section and we will write $Y$ instead of $\varphi(Y)$ in the arguments that follow. We will therefore see $U$ as a tubular neighborhood of $Y$ in $Z$ and denote $E_Y = \pi^{-1}(Y)$, $E_U = \pi^{-1}(U)$ and $E_U^+ = \pi^{-1}(\partial U)$.

The next definition is a fibered version of the classical Thom isomorphism.

\begin{defi}
    The cap product with the pullback $\pi^*[\tau] \in H^k(E_U, E_U^+)$ defines an isomorphism $$u_* : H_*(E_U, E_U^+) \to H_{*-k}(E_Y)$$ in the following way.
    Let $v \in E_U$ and $w = \pi(v) \in U$. Let $\gamma_w : [0,1] \to p^{-1}(p(w))$ such that $\gamma_w(0) = w$ and $\gamma_w(1) = p(w) \in Y$ be the path given by the gradient line passing through $w$. Let $p_{\pi} : E_U \to E_Y$ be the map defined by $p_{\pi}(v) = \Phi(v,\gamma_{\pi(v)}).$  We then define $$u_* = p_{\pi,*}(\cdot \cap \pi^*[\tau]).$$
\end{defi}

It will be a direct consequence of Proposition \ref{prop : Morse identification DG is the Thom iso} that $u_*$ is an isomorphism.\\

 One can construct a Gysin map $\varphi_! : H_*(E) \to H_{*-k}(E_Y)$ given by the composition of a Pontryagin-Thom collapse map $\tau_{\varphi,E,*} : H_*(E) \to H_*(E, E-E_U) \simeq H_*(E_U, E_U^+)$ with the Thom isomorphism $u_*$ (see \cite{GS07}).

We now prove that the shriek map $\varphi_! : H_*(Z,C_*(F)) \to H_{*-k}(Y,\varphi^*C_*(F))$ corresponds via the Fibration Theorem to the Gysin map $\varphi_! : H_*(E) \to H_{*-k}(E_Y)$.

\begin{thm}\label{thm : DG Thom iso}
    Let $Y^{n-k} \overset{\varphi}{\hookrightarrow} Z^n$ be an embedding of manifolds, $F \hookrightarrow E \overset{\pi}{\to} Z$ be a fibration endowed with a transitive lifting function $\Phi : E \ftimes{\pi}{\ev_0} \mathcal{P}Z \to E$ and $F \hookrightarrow E_Y \overset{\pi}{\to} Y$ be the pullback fibration by $\varphi$. The shriek map $\varphi_! : H_*(Z,C_*(F)) \to H_{*-k}(Y,\varphi^*C_*(F))$ is compatible, up to sign, with the Gysin map $\varphi_! : H_*(E) \to H_{*-k}(E_Y)$ via the Fibration Theorem \cite[Theorem 7.2.1]{BDHO23}. More precisely, for $i \geq 0$, the following diagram commutes up to the sign $(-1)^{ki -k}$.

    \[
    \xymatrix{
    H_i(Z,C_*(F)) \ar[r]^-{\varphi_!} \ar[d]_{\Psi_E}^{\sim}& H_{i-k}(Y,\varphi^*C_*(F)) \ar[d]^{\Psi_{E_Y}}_{\sim} \\
    H_i(E) \ar[r]_-{\varphi_!} & H_{i-k}(E_Y).
    }
    \]  
    
\end{thm}

\begin{proof}
Denote $\F = C_*(F)$. Recall $i_U : U \to Z$ is an embedding of the normal bundle $p : U \to Y$ of $\varphi$ and that we see $\varphi : Y \hookrightarrow Z$ as an inclusion.\\

Take a Morse-Smale pair $(f_Y,\xi_Y)$ on $Y$ and extend $f_Y$  on $U$ by $f_U = \chi\cdot f_Y\circ p$, where $\chi$ is some cutoff function equals $1$ near $Y$ and $0$ near $\partial U$. The pseudo-gradient $\xi_U$ extends $\xi_Y$ and points outwards of $U$. Extend it again to a Morse-Smale pair $(f,\xi)$ on $Z$  such that $\Crit_j(f) \cap U = \Crit_{j-k}(f_Y)$. This is the Morse identification $C^{\textup{Morse}}_*(U,\partial U) = C^{\textup{Morse}}_{*-k}(Y)$.\\

We use the definition of the shriek map in the case of an embedding given in \cite[Section 9]{BDHO23} :

    $$i_{U,!} : H_*(Z,\F) \to H_*(U,\partial U, i_U^*\F) = H_{*-k}(Y, \varphi^*\F), \ i_{U,!}(\alpha \otimes x) = \left\{ \begin{array}{cc}
        \alpha \otimes x & \textup{ if } x \in U  \\
         0 & \textup{ otherwise } 
    \end{array} \right..$$

We prove the theorem by showing that $i_{U,!} : H_*(Z,\F) \to H_*(U,\partial U, i_U^*\F)$ corresponds to the Pontryagin-Thom collapse $\tau_{\varphi,E,*} : H_*(E) \to H_*(E_U, E_U^+)$ and that the equality $H_*(U,\partial U, i_U^*\F) = H_{*-k}(Y, \varphi^*\F)$ corresponds to the Thom isomorphism via the Fibration Theorem.

\begin{prop}
The following diagram commutes :
    \[
    \xymatrix{
    H_*(Z,\F) \ar[r]^-{\Psi_{E}} \ar[d]^{i_{U,!}} & H_*(E) \ar[d]^{\tau_{\varphi,E,*}}\\
    H_*(U,\partial U,i_U^*\F)  \ar[r]^-{\Psi_{E_{U}}} & H_*(E_{U}, E_{U}^+).
    }
    \]

\end{prop}

\begin{proof}
Let $\mathcal{Y}$ be a tree rooted at $\star$ whose vertices are the critical points of $f$ and denote $\theta : Z/\mathcal{Y} \to Z$ a homotopy inverse of the canonical projection $X \to X/\mathcal{Y}.$
Consider $F_{\mathcal{Y}} \hookrightarrow E_{\mathcal{Y}} \overset{\pi_{\mathcal{Y}}}{\to} Z/\mathcal{Y}$ the pullback fibration of $F \hookrightarrow E \to Z$ by $\theta : Z/\mathcal{Y} \to Z$. Denote $E_{U_{\mathcal{Y}}} = \pi_{\mathcal{Y}}^{-1}(U_{\mathcal{Y}})$, where $U_{\mathcal{Y}} :=U/\mathcal{Y}$, $E^+_{U_\mathcal{Y}} = \pi_{\mathcal{Y}}^{-1}(\partial U_{\mathcal{Y}}) $ and $\F_{\mathcal{Y}} = C_*(F_{\mathcal{Y}})$.

By the definition of the quasi-isomorphisms $\Psi_E$ and $\Psi_{E_U}$ (see Section \ref{subsection : Fibration Theorem} or \cite[Section 7]{BDHO23} for more details), the diagram is decomposable into
    \[
    \xymatrix{
    H_*(Z,\F) \ar[r]^-{\sim}\ar[d]^{i_{U,!}} & H_*(Z,\F_{\mathcal{Y}}) \ar[r]^-{\Psi_{E_{\mathcal{Y}}}} \ar[d]^{i_{U,!}} & H_*(E_{\mathcal{Y}}) \ar[r]^-{\sim} \ar[d]^{\tau_{\varphi,E_{\mathcal{Y}},*}} & H_*(E) \ar[d]^{\tau_{\varphi,E,*}}\\
    H_*(U,\partial U,i_U^*\F) \ar[r]^-{\sim} & H_*(U,\partial U,i_{U}^*\F_{\mathcal{Y}}) \ar[r]^-{\Psi_{E_{U,\mathcal{Y}}}} & H_*(E_{U_\mathcal{Y}}, E^{+}_{U_\mathcal{Y}}) \ar[r]^-{\sim} & H_*(E_{U}, E_{U}^+)
    }
    \]
     The leftmost square commutes because the left horizontal maps are identifications of DG modules. The horizontal maps on the right are the isomorphism induced by the homotopy equivalence $\tilde{\theta} : E_\mathcal{Y} \to E$, $\tilde{\theta}(y,e) = (\theta(y),e)$. We define $\tau_{\varphi,E_{\mathcal{Y}},*} = \tilde{\theta}_* \circ \tau_{\varphi,E,*} \circ \tilde{\theta}_*^{-1}$  so that the rightmost square commutes.\\
    
    Let us now prove that the middle square commutes. Let $x \in \Crit(f)$ and $\alpha \in \F_{\mathcal{Y}}$. If $x \notin U$, then none of the gradient lines coming from $x$ will cross $U$ since the pseudo-gradient points outwards $U$. Therefore, 

    $\Psi_{E_{\mathcal{Y}}}(\alpha \otimes x) \in C_*(E_{\mathcal{Y}}\setminus E_{U,\mathcal{Y}})$ and

    $$\tau_{\varphi,E_{\mathcal{Y}},*}\circ \Psi_{E_{\mathcal{Y}}}(\alpha \otimes x) = 0.$$

     Moreover, by definition of $i_{U,!}$, $$\Psi_{E_{U, \mathcal{Y}}} \circ i_{U,!}(\alpha \otimes x)= 0.$$ 
        
     If $x \in U$, we decompose the middle square into 
    
    $$\xymatrix{
    H_*(Z,\F_{\mathcal{Y}}) \ar[d]^{p_{U,*}} \ar[r]^-{\Psi_{E_{\mathcal{Y}}}} & H_*(E_{\mathcal{Y}}) \ar[d]^{p_{E_U,*}} \\
    H_*(Z, Z \setminus U, \F_{\mathcal{Y}}) \ar@{=}[d] & H_*(E_{\mathcal{Y}}, E_{\mathcal{Y}} \setminus E_{U_ \mathcal{Y}}) \\
    H_{*}(U,\partial U, i_{U}^*\F_{\mathcal{Y}}) \ar[r]^-{\Psi_{E_{U_\mathcal{Y}}}} & H_*(E_{U_\mathcal{Y}}, E^+_{U_ \mathcal{Y}}) \ar[u]^{\sim}_{j_{U,*}}
    }$$

    The map $j_{U,*} : C_*(E_{U_\mathcal{Y}},E^+_{U_\mathcal{Y}}) \to  C_*(E_{\mathcal{Y}}, E_{\mathcal{Y}} / E_{U_\mathcal{Y}})$ is the homotopy equivalence given by excision and $p_{E_U,*} : C_*(E_{\mathcal{Y}}) \to C_*(E_{\mathcal{Y}}, E_{\mathcal{Y}} / E_{U_\mathcal{Y}})$, $p_{U,*} : C_*(Z,\F_{\mathcal{Y}}) \to C_*(Z, Z \setminus U, \F_{\mathcal{Y}}) $ are the canonical projections.\\
    
    Let $\{s_x \in C_{|x|}(\wb{u}{x}), \ x \in \Crit(f)\}$ be a representing chain system for the Latour cells of the pair $(f,\xi)$. 
    We can define evaluation maps (see \cite[Lemma 7.3.3]{BDHO23}) $$q_{x} : \wb{u}{x} \to \mathcal{P}_{\star \to Z/\mathcal{Y}} Z/\mathcal{Y}$$ into the space of based paths in $Z/\mathcal{Y}$ or $$q^U_{x} :  \wb{u}{x} \to \mathcal{P}_{\star \to U_{\mathcal{Y}}} U_{\mathcal{Y}}$$ into $U_{\mathcal{Y}}$. These two maps respectively evaluate $s_x$ to obtain $$m_x \in C_{|x|}(\mathcal{P}_{\star \to Z/\mathcal{Y}} Z/\mathcal{Y}),$$ or to obtain $$m^{U}_{x} \in C_{|x|}(\mathcal{P}_{\star \to U_{\mathcal{Y}}} U_{\mathcal{Y}}, \mathcal{P}_{\star \to \partial U_{\mathcal{Y}}} U_{\mathcal{Y}}).$$ 

   It is then clear that, up to homotopy, for all $\alpha \in \F_{\mathcal{Y}}$, 
   $$j_{U,*} (\Psi_{E_{U_\mathcal{Y}}}(\alpha \otimes x)) = j_{U,*}(\Phi( \alpha \otimes m_x^{U})) = p_{E_U,*}(\Phi(\alpha \otimes m_x)) = p_{E_U,*}(\Psi_{E_{\mathcal{Y}}}(\alpha \otimes x)) \in C_*(E_{\mathcal{Y}}, E_{\mathcal{Y}}\setminus E_{U_\mathcal{Y}}).$$

\end{proof}
    
    In order to finish the proof of Theorem \ref{thm : DG Thom iso}, we prove the following proposition, which is a result that has been referred to in \cite[Remark 9.2.5]{BDHO23}.

\begin{prop}\label{prop : Morse identification DG is the Thom iso}
    The Morse identification $ H_*(U, \partial U, i_{U}^*\F) = H_{*-k}(Y, i^*\F)$ corresponds to the Thom isomorphism in the sense that the diagram

    $$\xymatrix{
     & H_i(U, \partial U, i_U^*\F) = H_{i-k}(Y, i^*\F) \ar[dr]_{\Psi_{E_Y}} \ar[dl]^{\Psi_{E_U}} &\\
     H_i(E_U , E_{U}^+) \ar[rr]^{u_*}& & H_{i-k}(E_Y)
    }$$
    commutes up to the sign $(-1)^{ki-k}$.
\end{prop}

\begin{proof}
    It is enough to consider the pullback fibration by $\theta$ and prove that the diagram 

$$\xymatrix{
     & H_*(U, \partial U, i_U^*\F_{\mathcal{Y}}) = H_{*-k}(Y, i^*\F_{\mathcal{Y}}) \ar[dr]_{\Psi_{E_{Y,\mathcal{Y}}}} \ar[dl]^{\Psi_{E_{U,\mathcal{Y}}}} &\\
     H_*(E_{U, \mathcal{Y}} , E^{+}_{U, \mathcal{Y}}) \ar[rr]^{u_*}& & H_{*-k}(E_{Y,\mathcal{Y}}).
    }$$

    \noindent commutes.

    Let $\alpha \in \F$ and $x \in \Crit(f_Y).$ Denote $|x|_U = |x| +k$ the index of $x$ as a critical point of $f_U$. Let $\{s_{x,y} \in C_{|x|-|y|-1}(\trajbi{Y}{x,y}), \ x,y \in \Crit(f_Y)\}$ be a representing chain system for the Morse moduli spaces in $Y$ and $\left\{s_x^Y \in C_{|x|_Y}(\wbi{u}{Y}{x}), \  x \in \Crit(f_Y)\right\}$ a compatible representing chain system of the Latour cells in $Y$ as defined in Definition \ref{defi : for the fibration thm}. Recall that it is a family of chains that satisfies 

    \begin{enumerate}
        \item each $s_x^Y$ is a cycle relative to the boundary and represents the fundamental class $[\wbi{u}{Y}{x}, \partial \wbi{u}{Y}{x}]$.\\
        \item each $s_x$ satisfies $\partial s_x = \displaystyle \sum_y s_{x,y} \times s_y$ with the product of chains defined via the inclusions $\trajbi{Y}{x,y} \times \wbi{u}{Y}{y} \subset \partial \wbi{u}{Y}{x} \subset \wbi{u}{Y}{x}.$ 
    \end{enumerate}

    We want to build a family $\left\{s_x^U \in C_{|x|_U}\left(\wbi{u}{U}{x}, \partial U \cap \wbi{u}{U}{x}\right), x \in \Crit(f_U)\right\}$ such that 

    \begin{enumerate}
    \item $s_x^U$ is a cycle relative to the boundary and represents the fundamental class $[\wbi{u}{U}{x}, \partial \wbi{u}{U}{x}]$.
        \item $s_x^Y = s_x^U\lvert_{[0,1]^{|x|} \times \{0\}}  \in C_{|x|}(\wbi{u}{Y}{x})$ is a representing chain system of the unstable manifold of $x$ seen as a critical point in $Y$.
        \item $a = s_x^U\lvert_{\{0\} \times [0,1]^k}$  is a representative of the fundamental class $[\mathbb{D}^k, \partial \mathbb{D}^k].$
        
    \end{enumerate}

Once such a representing chain system is built, we can define the associated twisting cocycle $m_x^Y$ and $m_x^U$ by evaluating respectively $s_x^Y$ and $s_x^U$ in the space of based paths of $Y/\mathcal{Y}$ and $U/\mathcal{Y}$. Recall that the isomorphisms $\Psi_{E_{U,\mathcal{Y}}}$ and $\Psi_{E_{Y,\mathcal{Y}}}$ are then defined by

    $$\Psi_{E_{U,\mathcal{Y}}}(\alpha \otimes x) = \Phi_{*}(\alpha \otimes m_x^U) \  \textup{and} \ \Psi_{E_{Y,\mathcal{Y}}}(\alpha \otimes x) = \Phi_{*}(\alpha \otimes m_x^Y).$$
    
We then can just evaluate at the chain level using a representative $\tau$ of the Thom class $[\tau] \in H^k(E_{U,\mathcal{Y}}, E^+_{U,\mathcal{Y}})$:
\begin{align*}
   u_*\left(\Psi_{E_{U,\mathcal{Y}}}(\alpha \otimes x)\right) &= u_*(\Phi_{*}(\alpha \otimes m_x^U)) \\
   &= (-1)^{k(|\alpha| + |x|_Y)} \tau(a) \Phi_{*}(\alpha \otimes m_x^Y) \\
   &= (-1)^{k(|\alpha| + |x|_U - k)} \Phi_{*}(\alpha \otimes m_x^Y)\\
   &= (-1)^{k(|\alpha| + |x|_U) - k} \Psi_{E_{Y,\mathcal{Y}}}(\alpha \otimes x).
\end{align*}

It remains to construct such a representing chain system $\{s_x^U, \ x \in \Crit(f_U)\}$. We do it by induction on $|x|_Y.$ Let $\{s_x^Y, \ x \in \Crit(f_Y)\}$ be a representing chain system of the Latour cells $\left\{\overline{W}^u_Y(x)\right\}$.

If $|x|_Y = 0$, then $\overline{W}^u_U(x) = \mathbb{D}^k$ and $\partial U \cap \overline{W}^u_U(x) = \partial \overline{W}^u_U(x) = \partial \mathbb{D}^k$. Define $s_x^U = a$ to be a representative of the fundamental class $[\mathbb{D}^k, \partial \mathbb{D}^k] = [\overline{W}^u_U(x), \partial \overline{W}^u_U(x)]$.

Suppose now that $s_y^U$ has been constructed for every $|y|_Y \leq l$ for some $l \in \N$. Let $x\in Crit_{k+1}(f_Y)$. 

Define $s_x^U = s_x^Y \times a$ and remark that, since $a$ is a cycle, $$\partial s_x^U = \sum_y s_{x,y} s_y^U$$ and it is a representative of $[\partial \wbi{u}{U}{x}]$ by the induction hypothesis. Indeed, $$\partial \wbi{u}{U}{x} =  \bigcup_{|y|<|x|} \trajbi{U}{x,y} \times \wbi{u}{U}{y}$$ and the product orientation is equal to the boundary orientation.  Therefore, if $s'^{U}_x$ is a representative of $[\wbi{u}{U}{x}]$, there exists a chain $p_{x} \in C_{|x|}(\partial\wbi{u}{U}{x})$ such that $\partial(s^U_x - s'^{U}_x) = \partial p_x.$

It follows that $s'^{U}_x + p_x - s^U_x\in C_{|x|_U}(\wbi{u}{U}{x})$ is a cycle. However, since $\wbi{u}{U}{x}$ is a $|x|_U$-dimensional, closed and connected manifold with boundary and corners, we have $H_{|x|_U}(\wbi{u}{U}{x}) = 0$ and therefore there exists $b \in C_{|x|_U+1}(\wbi{u}{U}{x})$ such that

$$s'^{U}_x + p_x - s^U_x = \partial b$$

and therefore 

$$s^U_x = s'^{U}_x = [\wbi{u}{U}{x}, \partial \wbi{u}{U}{x}] \in H_{|x|_U}(\wbi{u}{U}{x}, \partial \wbi{u}{U}{x}).$$

\end{proof}

To summarize, we have proved that the diagram 

$$\xymatrix@C=4cm{
H_*(Z,\F) \ar@/^2pc/[rr]^{\varphi_!} \ar[r]_{i_{U,!}} \ar[d]^{\Psi_E} & H_*(U, \partial U , i_U^*\F) \ar@{=}[r] \ar[d]^{\Psi_{E_U}} & H_{*-k}(Y, \varphi^*\F) \ar[d]_{\Psi_{E_Y}} \\
H_*(E)\ar@/_2pc/[rr]_{\varphi_!} \ar[r]^-{\tau_{\varphi,E,*}} & H_*(E_U, E_u^+) \ar[r]^-{u_*} & H_{*-k}(E_{Y})
}$$
commutes and therefore Theorem \ref{thm : DG Thom iso} is proved.

\end{proof}

\section{Morse homology with coefficients in an \texorpdfstring{$\Ai$}{Ai}-module}\label{section : Morse homology with coefficients in a Ai-module}

In this section, we extend the construction of Morse homology with coefficients in a DG module to the more general case of a $\Ai$-module structure over $C_*(\Omega X)$. The main motivation to consider this setting is that a morphism of $\Ai$-modules $\boldsymbol{\varphi} : \mathcal{A} \to  \mathcal{A}'$ induces a morphism of enriched complexes $\Tilde{\varphi} : C_*(X,\mathcal{A}) \to C_*(X, \mathcal{A}')$ (Proposition \ref{prop : Morphisme de Ai module en morphisme de complexe}). This property plays a key role in the definition of the Chas-Sullivan product $\CS$.

\subsection{\texorpdfstring{$\Ai$}{Ai}-modules over \texorpdfstring{$C_*(\Omega X)$}{C(OmegaX)}}

The complex $C_*(\Omega X)$ has an associative algebra structure, which we view as an $\Ai$-algebra structure $(\mu_i)_{i\geq 0}$ given by the differential $\mu_1 : C_*(\Omega X) \to C_{*-1}(\Omega X)$, the Pontryagin product $\mu_2 : C_*(\Omega X)^{\otimes 2} \to C_*(\Omega X)$, and $\mu_i = 0$ for $i \geq 3$.

\begin{defi}
    A differential graded $\Z$-module $(\mathcal{A},\nu)$ is an $\boldsymbol{\mathcal{A_{\infty}}}$\textbf{-module} over $C_*(\Omega X)$ if there exists a sequence of operations $\nu_n^A : \mathcal{A} \otimes C_*(\Omega X)^{\otimes n-1} \to \mathcal{A}$ of degrees $n-2$ for all $n \geq 1$ such that $\nu_1^A = \nu$ and for all $N\geq 1$,
\begin{equation*}
    \displaystyle \sum_{\substack{s+t =N \\ s\geq 1}} (-1)^{st} \nu^A_{t+1}( \nu^A_s \otimes 1^{\otimes t}) + \sum_{\substack{r+s+t=N \\ r,s \geq 1}} (-1)^{r+st} \nu^A_{r+t+1}(1^{\otimes r} \otimes \mu_s \otimes 1^{\otimes t}) = 0.
\end{equation*}

Using the fact that $\mu_i = 0$ for $i\geq 3$, this leads to the equation

\begin{equation}\label{eq : nu A infinite module}
    \displaystyle \sum_{\substack{s+t =N \\ s\geq 1}} (-1)^{st} \nu^A_{t+1}( \nu^A_s \otimes 1^{\otimes t}) + \sum_{\substack{r+t=N-1 \\ r \geq 1}} (-1)^{N-1} \nu^A_{N}(1^{\otimes r} \otimes \mu_1 \otimes 1^{\otimes t}) + \sum_{\substack{r+t=N-2 \\ r \geq 1}} (-1)^r \nu^A_{N-1}(1^{\otimes r} \otimes \mu_2 \otimes 1^{\otimes t}) = 0.
\end{equation}

\end{defi}
Note that these functional expressions are subject to the Koszul sign rule when applied to elements: if $f : V_* \to W_*$ and $g : V'_* \to W'_*$ are homogeneous graded linear maps of respective degrees $|f|$ and $|g|$, then for any $x \otimes y \in V \otimes V'$, $(f \otimes g)(x\otimes y) = (-1)^{|g||x|} f(x) \otimes g(y)$.

\begin{defi}
    A \textbf{morphism }$\boldsymbol{\varphi}: (\mathcal{A}, \boldsymbol{\nu^A}) \to (\mathcal{B}, \boldsymbol{\nu^B})$ \textbf{of} $\boldsymbol{\mathcal{A}_{\infty}}$\textbf{-modules} over $C_*(\Omega X)$ is the data, for all $n \geq 1$, of a chain map $\varphi_n : \mathcal{A} \otimes C_*(\Omega X)^{\otimes n-1} \to \mathcal{B}$ of degree $n-1$ such that for all $N \in \N^*$,

\begin{equation*}
    \sum_{s+t = N} (-1)^{st} \varphi_{t+1}(\nu^A_s \otimes 1^{\otimes t}) +  \sum_{\substack{r+s+t = N \\ r \geq 1}} (-1)^{r +st} \varphi_{r+t+1}( 1^{\otimes r} \otimes \mu_s \otimes 1^{\otimes t}) = \sum_{s+t = N} (-1)^{(s+1)t} \nu^B_{t+1}(\varphi_s \otimes 1^{\otimes t}).
\end{equation*}
    
Using that $C_*(\Omega X)$ is an associative algebra, we get
\begin{equation}\label{eq : relation morphisme de Ai module}
\begin{split}
    \sum_{n+k = N} (-1)^{n(k+1)} \varphi_{n+1}(\nu_{k+1}^A \otimes & 1^{\otimes n}) + \sum_{r=1}^{N} (-1)^N \varphi_{N+1}(1^{\otimes r} \otimes \mu_1 \otimes 1^{\otimes N-r}) \\
    &+ \sum_{r=1}^{N-1} (-1)^r \varphi_{N}(1^{\otimes r} \otimes \mu_2 \otimes 1^{\otimes N-1-r}) = \sum_{n+k = N} (-1)^{kn} \nu^B_{n+1}(\varphi_{k+1} \otimes 1^{\otimes n}).
\end{split}
\end{equation}
\end{defi} 

\subsection{Complex with coefficients in an \texorpdfstring{$\Ai$}{Ai}-module}\label{subsection : Complex with coefficients in an Ai module}

Defining a complex with this larger kind of coefficients is not very different from the case of DG modules. However, in order to prove that there exist $\Ai$ analogues for the definitions and tools of \cite{BDHO23}, raw computations can be very complicated. That is why, in this paper, we deliberately use as much functional expressions as possible in order to avoid Kozsul signs that quickly make  computations unreadable. To this effect, we use the following notation:

\begin{notation}\label{not : morphism tilde m}
Let $\{m_{x,y} \in C_{|x|-|y|-1}(\Omega X), \ x,y \in \Crit(f)\}$ be a twisting cocycle. We denote $\mathbf{m} \in \Hom_{-1}(\Z\Crit(f), C_*(\Omega X) \otimes \Z\Crit(f))$ the morphism given by $$\mathbf{m}(x) = \sum_y m_{x,y} \otimes y.$$ 

For a more compact writing, if $\mathcal{A}$ is an $\Ai$-module over $C_*(\Omega X)$, we extend $\mathbf{m}$ to $\mathbf{\Tilde{m}}\in \End_{-1}(\mathcal{A} \otimes TC_*(\Omega X) \otimes \Z\textup{Crit}(f))$ defined by

$$\mathbf{\Tilde{m}}(\alpha \otimes \gamma_1 \otimes \dots \otimes \gamma_k \otimes x) = (1^{\otimes k+1} \otimes \mathbf{m})(\alpha \otimes \gamma_1 \otimes \dots \otimes \gamma_k \otimes x),$$

where $TC_*(\Omega X) = \displaystyle \bigoplus_{i \geq 0} C_*(\Omega X)^{\otimes i}$ is the tensor algebra of $C_*(\Omega X)$. 

We also extend $\mu_1$ and $\mu_2$ to $\mathcal{A} \otimes TC_*(\Omega X) \otimes \Z\Crit(f)$ by 
\begin{align*}
    &\Tilde{\mu}_1(\alpha \otimes \gamma_1 \otimes \dots \otimes \gamma_k \otimes x) = (1^{\otimes k} \otimes \mu_1 \otimes 1)( \alpha \otimes \gamma_1 \otimes \dots \otimes \gamma_k \otimes x), \\
    &\Tilde{\mu}_2(\alpha \otimes \gamma_1 \otimes \dots \otimes \gamma_{k+1} \otimes x) = (1^{\otimes k} \otimes \mu_2 \otimes 1)( \alpha \otimes \gamma_1 \otimes \dots \otimes  \gamma_{k+1} \otimes x). 
\end{align*}
\end{notation}

\begin{lemme}
    The Maurer-Cartan equation can be written as

\begin{equation}\label{eq : MC Ai module}
    \Tilde{\mu}_1 \mathbf{\Tilde{m}} + \Tilde{\mu}_2 \mathbf{\Tilde{m}}^2 = 0.
\end{equation}
\end{lemme}

\begin{proof}
    Using the Kozsul sign rule, we check  

    \begin{align*}
        \Tilde{\mu}_1 \mathbf{\Tilde{m}}(\alpha \otimes \gamma_1 \otimes \dots \otimes \gamma_k \otimes x) &= (-1)^{|\alpha| + |\gamma_1| + \dots + |\gamma_k|} \sum_y \Tilde{\mu}_1(\alpha \otimes \gamma_1 \otimes \dots \otimes \gamma_k \otimes m_{x,y} \otimes y)\\
        &= \sum_{y,z} (-1)^{|x| - |z|} \alpha \otimes \gamma_1 \otimes \dots \otimes \gamma_k \otimes \mu_2( m_{x,z} \otimes m_{z,y}) \otimes y\\
        &= \sum_{y,z} (-1)^{|x| - |z|} \Tilde{\mu}_2(\alpha \otimes \gamma_1 \otimes \dots \otimes \gamma_k \otimes m_{x,z} \otimes m_{z,y} \otimes y)\\
        &= - \Tilde{\mu_2} \mathbf{\Tilde{m}}^2(\alpha \otimes \gamma_1 \otimes \dots \otimes \gamma_k \otimes x).
    \end{align*}
\end{proof}

Using this equation, we prove a technical lemma that will be useful in several proofs :

\begin{lemme}\label{lemme : elimination using MC}
    For all $r \in \{1, \dots, N\}$, $$(-1)^N (1^{\otimes r} \otimes \mu_1 \otimes 1^{\otimes N-r}) \m^N = (-1)^{r+1} (1^{\otimes r} \otimes \mu_2 \otimes 1^{\otimes N-r})\mathbf{\Tilde{m}}^{N+1}.$$
\end{lemme}
\begin{proof}
    \begin{align*}
    (-1)^N (1^{\otimes r} \otimes \mu_1 \otimes 1^{\otimes N-r})\mathbf{\Tilde{m}}^N &=   (-1)^N (-1)^{N-r} \mathbf{\Tilde{m}}^{N-r} \Tilde{\mu_1}\mathbf{\Tilde{m}}^r\\
    &= - (-1)^{r} \mathbf{\Tilde{m}}^{N-r} \Tilde{\mu_2}\mathbf{\Tilde{m}}^{r+1}\\
    &= -(-1)^{r} (1^{\otimes r} \otimes \mu_2 \otimes 1^{\otimes N-r})\mathbf{\Tilde{m}}^{N+1}.
\end{align*}
\end{proof}

\begin{defi}
    Let $f : X \to \R$ be a Morse function on $X$, $\{m_{x,y} \in C_{|x|-|y|-1}(\Omega X), \ x,y \in \Crit(f)\}$ be a twisting cocycle and $(\mathcal{A}, \boldsymbol{\nu} = \{\nu_{n}\}_n)$ be a $\Ai$-module over $C_*(\Omega X)$. We define the complex $$C_*(X,m_{x,y},\mathcal{A}) = C_*(X, \mathcal{A}) := \mathcal{A}_* \otimes \Z\Crit(f),$$

endowed with the differential 

$$\partial = \sum_{n\geq 0} (\nu_{n+1} \otimes 1) \mathbf{\Tilde{m}}^n.$$
\end{defi}

We spell out the three first terms of this differential when applied to a generator $\alpha \otimes x \in C_*(X,\mathcal{A}).$

\begin{align*}
    \partial(\alpha \otimes x) &= \nu_1(\alpha) \otimes x + (\nu_2\otimes 1)\m(\alpha \otimes x) + (\nu_3\otimes 1) \m^2(\alpha \otimes x) + \dots \\
    &= \partial \alpha \otimes x + (-1)^{|\alpha|} \sum_{y}(\nu_2 \otimes 1)(\alpha \otimes  m_{x,y} \otimes y) + \sum_{y} (-1)^{|\alpha|}(\nu_3 \otimes 1)\m(\alpha \otimes  m_{x,y} \otimes y) + \dots \\
    &= \partial \alpha \otimes x + (-1)^{|\alpha|} \sum_{y}\nu_2(\alpha \otimes  m_{x,y}) \otimes y + \sum_{y,z} (-1)^{|\alpha|}(-1)^{|\alpha| + |x|-|y|-1} \nu_3(\alpha \otimes m_{x,y} \otimes m_{y,z}) \otimes z + \dots
\end{align*}

\begin{rem}
    This sum is finite because $\m$ strictly decreases the index of the critical point. We can also remark that if $\mathcal{A}$ is a DG-module (ie $\nu_{n+1} = 0$ for all $n \geq 2$), then this differential corresponds to the known differential for the DG Morse complex.
\end{rem}

\begin{prop}
    The map $\partial$ is a differential on $C_*(X,m_{x,y},\mathcal{A})$.
\end{prop}

\begin{proof}
The map $\partial$ has degree $-1$ because $\mathbf{\Tilde{m}}^n$ has degree $-n$ and $\nu_{n+1}$ has degree $n-1$.

It remains to show that $\partial^2 =0$.

\begin{align*}
    \partial^2 &= \sum_{n,k}(\nu_{n+1} \otimes 1)\mathbf{\Tilde{m}}^n (\nu_{k+1} \otimes 1) \mathbf{\Tilde{m}}^k\\
    &= \sum_{n,k} (-1)^{n(k-1)} (\nu_{n+1} (\nu_{k+1} \otimes 1^{\otimes n}) \otimes 1) \mathbf{\Tilde{m}}^{n+k}\\
    &= - \sum_N \left[ \sum_{r=1}^{N} (-1)^N \nu_{N+1}(1^{\otimes r} \otimes \mu_1 \otimes 1^{\otimes N-r}) + \sum_{r=1}^{N-1} (-1)^r \nu_{N}(1^{\otimes r} \otimes \mu_2 \otimes 1^{\otimes N-1-r}) \right] \mathbf{\Tilde{m}}^N. 
\end{align*}
We used equation \eqref{eq : nu A infinite module} for the last equality.
It then follows, using Lemma \ref{lemme : elimination using MC}, that
\begin{align*}
    \partial^2 &= - \sum_N \left[ \sum_{r=1}^{N} (-1)^N \nu_{N+1}(1^{\otimes r} \otimes \mu_1 \otimes 1^{\otimes N-r}) + \sum_{r=1}^{N-1} (-1)^r \nu_{N}(1^{\otimes r} \otimes \mu_2 \otimes 1^{\otimes N-1-r}) \right] \mathbf{\Tilde{m}}^N\\
    &= \sum_N \left[ \sum_{r=1}^{N} (-1)^r \nu_{N+1}(1^{\otimes r} \otimes \mu_2 \otimes 1^{\otimes N-r}) \mathbf{\Tilde{m}}^{N+1} - \sum_{r=1}^{N-1} (-1)^r \nu_{N}(1^{\otimes r} \otimes \mu_2 \otimes 1^{\otimes N-1-r})\mathbf{\Tilde{m}}^N \right]\\
    &=0
\end{align*}
\end{proof}

We will also denote the associated homology groups by  $H_*(X, m_{x,y}, \mathcal{A})$. If $\{m_{x,y}\}$ is a Barraud-Cornea cocycle, we will denote $H_*(X, \mathcal{A})$ the associated homology groups ; we will prove in Theorem \ref{thm : Continuation morphism Ai invariance} that this homology does not depend on the Morse data set $\Xi$ used to define the $\{m_{x,y}\}$.

\subsection{\texorpdfstring{$\Ai$}{Ai}-Morse toolset}\label{subsection : Continuation and homotopy cocycles} 

The following propositions are $\Ai$ analogues to \cite{BDHO23} Proposition 2.3.3 and Proposition 2.3.4. The goal is to prove that a cocycle defined by a homotopy between two Morse-Smale pairs induces a chain map between enriched complexes defined using those pairs. If those cocycles are themselves homotopic (in a sense that will be defined in equation \eqref{eq : mu_1 de h}), then it induces an homotopy between the chain maps. However, we will state these propositions in a more general setting, so they can be applied even in situations where the twisting cocycles considered are not Barraud-Cornea cocycles. 
Let $\F$ be an $\Ai$-module over $C_*(\Omega X)$.

\begin{prop}\label{Prop : continuation morphisms}

Let $f_0, f_1 : X \to \R$ be two Morse functions on $X$. Let $\{m^0_{x,y} \in C_{|x|-|y|-1}(\Omega X), \ x,y \in \Crit(f_0)\}$ and $\{m^1_{x',y'} \in C_{|x'|-|y'|-1}(\Omega X), \ x',y' \in \Crit(f_1)\}$ be twisting cocycles.

Let $\{\tau_{x,y'}\in C_{|x|-|y'|}(\Omega X), \ x \in \Crit(f_0), \ y' \in \Crit(f_1)\}$ be a cocycle satisfying the equation

\begin{equation}\label{eq : mu_1 de tau}
    \partial \tau_{x,y'} = \sum_{z \in \Crit(f_0)} m^0_{x,z}\cdot\tau_{z,y'} - \sum_{w' \in \Crit(f_1)} (-1)^{|x|-|w'|} \tau_{x,w'}\cdot m^1_{w',y'}.
\end{equation}

Then, the map $\Psi : C_*(X, m^0_{x,y},\F) \to C_*(X, m^1_{x',y'}, \F)$ defined by

\begin{align*}
    \Psi = \sum_{n \geq 1} \sum_{u=1}^{n} (-1)^{u-1} (\nu_{n+1} \otimes 1) \mathbf{\Tilde{m}}_{(1)}^{n-u}\boldsymbol{\Tilde{\tau}} \mathbf{\Tilde{m}}_{(0)}^{u-1}
\end{align*}

is a morphism of complexes where $\boldsymbol{\Tilde{\tau}} \in \Hom_0(\F \otimes TC_*(\Omega X) \otimes \Z\Crit(f_0), \F \otimes TC_*(\Omega X) \otimes \Z\Crit(f_1))$ is defined below.
\end{prop}

\begin{rem}
    We used notations similar to those in the previous section : Define the morphism $\boldsymbol{\tau} \in \Hom_{0}(\Z\Crit(f_0), C_*(\Omega X) \otimes \Z\Crit(f_1))$ by

    $$\boldsymbol{\tau}(x) = \sum_{y'} \tau_{x,y'} \otimes y'$$

    and $\boldsymbol{\Tilde{\tau}} \in \Hom_0(\F \otimes TC_*(\Omega X) \otimes \Z\Crit(f_0), \F \otimes TC_*(\Omega X) \otimes \Z\Crit(f_1))$ by

    $$\boldsymbol{\Tilde{\tau}}(\alpha \otimes \gamma_1 \otimes \dots \otimes \gamma_k \otimes x) = (1^{\otimes (k+1)} \otimes \boldsymbol{\tau})(\alpha \otimes \gamma_1 \otimes \dots \otimes \gamma_k \otimes x).$$

    We can then write equation \eqref{eq : mu_1 de tau} as

    \begin{equation}\label{eq: mu_1 of tau functional version}
        \Tilde{\mu}_1 \boldsymbol{\Tilde{\tau}} = \Tilde{\mu}_2 \boldsymbol{\Tilde{\tau}}\mathbf{\Tilde{m}}_{(0)} - \Tilde{\mu}_2 \mathbf{\Tilde{m}}_{(1)} \boldsymbol{\Tilde{\tau}}.
    \end{equation}
\end{rem}

\begin{myproof}{of Proposition}{\ref{Prop : continuation morphisms}} 
    We compute

    \begin{align*}
        \partial \Psi &= \sum_{\substack{n \geq 1\\ k \geq 0}} \sum_{u=1}^n (-1)^{u-1} (\nu_{k+1} \otimes 1) \mathbf{\Tilde{m}}^k_{(1)} (\nu_{n+1} \otimes 1) \mathbf{\Tilde{m}}^{n-u}_{(1)} \boldsymbol{\Tilde{\tau}} \mathbf{\Tilde{m}}^{u-1}_{(0)}\\
        &= \sum_{\substack{n \geq 1\\ k \geq 0}} \sum_{u=1}^n (-1)^{u-1} (-1)^{(n+1)k} (\nu_{k+1} (\nu_{n+1} \otimes 1^{\otimes k}) \otimes 1) \mathbf{\Tilde{m}}^{n+k-u}_{(1)} \boldsymbol{\Tilde{\tau}} \mathbf{\Tilde{m}}^{u-1}_{(0)}
    \end{align*}

    and

    \begin{align*}
        \Psi \partial &= \sum_{\substack{n \geq 0 \\ k \geq 1}} \sum_{u=1}^k (-1)^{u-1} (\nu_{k+1} \otimes 1)\mathbf{\Tilde{m}}^{k-u}_{(1)} \boldsymbol{\Tilde{\tau}} \mathbf{\Tilde{m}}^{u-1}_{(0)} (\nu_{n+1} \otimes 1)\mathbf{\Tilde{m}}^{n}_{(0)}\\
        &=\sum_{\substack{n \geq 0\\ k \geq 1}} \sum_{u=1}^k (-1)^{u-1} (-1)^{(k-1)(n+1)} (\nu_{k+1} (\nu_{n+1} \otimes 1^{\otimes k}) \otimes 1) \mathbf{\Tilde{m}}^{k-u}_{(1)} \boldsymbol{\Tilde{\tau}} \mathbf{\Tilde{m}}^{n+u-1}_{(0)}\\
        &=\sum_{\substack{n \geq 0\\ k \geq 1}} \sum_{u=n+1}^{n+k} (-1)^{u-n-1} (-1)^{(n +1)k + n + 1} (\nu_{k+1} (\nu_{n+1} \otimes 1^{\otimes k}) \otimes 1) \mathbf{\Tilde{m}}^{n+k-u}_{(1)} \boldsymbol{\Tilde{\tau}} \mathbf{\Tilde{m}}^{u-1}_{(0)}\\
        &=-\sum_{\substack{n \geq 0\\ k \geq 1}} \sum_{u=n+1}^{n+k} (-1)^{u-1} (-1)^{(n +1)k } (\nu_{k+1} (\nu_{n+1} \otimes 1^{\otimes k}) \otimes 1) \mathbf{\Tilde{m}}^{n+k-u}_{(1)} \boldsymbol{\Tilde{\tau}} \mathbf{\Tilde{m}}^{u-1}_{(0)}.
    \end{align*}

    Therefore, taking $N= n+k$, equation \eqref{eq : nu A infinite module} gives 

    \begin{align*}
         \Psi \partial - \partial \Psi &= \sum_{N \geq 1} \sum_{u=1}^N (-1)^{u-1} \left[ \sum_r (-1)^N \nu_{N+1}(1^{\otimes r} \otimes \mu_1 \otimes 1^{\otimes N-r}) \right. \\
         &+ \left. \sum_r (-1)^r \nu_{N}(1^{\otimes r} \otimes \mu_2 \otimes 1^{\otimes N-1-r})  \right] \mathbf{\Tilde{m}}^{N-u}_{(1)} \boldsymbol{\Tilde{\tau}} \mathbf{\Tilde{m}}^{u-1}_{(0)}.
    \end{align*}

    We decompose the computation of $(1^{\otimes r} \otimes \mu_1 \otimes 1^{\otimes N-r})\mathbf{\Tilde{m}}^{N-u}_{(1)} \boldsymbol{\Tilde{\tau}} \mathbf{\Tilde{m}}^{u-1}_{(0)}$ into three cases:

\begin{enumerate}
    \item If $r \leq u-1$, Lemma \ref{lemme : elimination using MC} gives 
    \begin{align*}
        (1^{\otimes r} \otimes \mu_1 \otimes 1^{\otimes N-r})\mathbf{\Tilde{m}}^{N-u}_{(1)} \boldsymbol{\Tilde{\tau}} \mathbf{\Tilde{m}}^{u-1}_{(0)} = (-1)^{N-r} (1^{\otimes r} \otimes \mu_2 \otimes 1^{\otimes N-r})\mathbf{\Tilde{m}}^{N-u}_{(1)} \boldsymbol{\Tilde{\tau}} \mathbf{\Tilde{m}}^{u}_{(0)}.
    \end{align*}\\
    \item If $r > u$, the same lemma gives
    \begin{align*}
        (1^{\otimes r} \otimes \mu_1 \otimes 1^{\otimes N-r})\mathbf{\Tilde{m}}^{N-u}_{(1)} \boldsymbol{\Tilde{\tau}} \mathbf{\Tilde{m}}^{u-1}_{(0)} = (-1)^{N-r+1} (1^{\otimes r} \otimes \mu_2 \otimes 1^{\otimes N-r})\mathbf{\Tilde{m}}^{N+1-u}_{(1)} \boldsymbol{\Tilde{\tau}} \mathbf{\Tilde{m}}^{u-1}_{(0)}.
    \end{align*}\\
    \item If $r = u$, we use equation \eqref{eq: mu_1 of tau functional version} to obtain the equality:
    \begin{align*}
         &(1^{\otimes u} \otimes \mu_1 \otimes 1^{\otimes N-u})\mathbf{\Tilde{m}}^{N-u}_{(1)} \boldsymbol{\Tilde{\tau}}\mathbf{\Tilde{m}}^{u-1}_{(0)} \\
         &= (-1)^{N-u} \mathbf{\Tilde{m}}^{N-u}_{(1)}(\Tilde{\mu}_1\boldsymbol{\Tilde{\tau}}) \mathbf{\Tilde{m}}^{u-1}_{(0)} \\
         &= (-1)^{N-u} \mathbf{\Tilde{m}}^{N-u}_{(1)} \left[ \Tilde{\mu}_2 \boldsymbol{\Tilde{\tau}} \mathbf{\Tilde{m}}_{(0)} - \Tilde{\mu}_2 \mathbf{\Tilde{m}}_{(1)} \boldsymbol{\Tilde{\tau}}\right] \mathbf{\Tilde{m}}^{u-1}_{(0)} \\
         &= (-1)^{N-u} (1^{\otimes u} \otimes \mu_2 \otimes 1^{\otimes N-r}) \left[ \mathbf{\Tilde{m}}^{N-u}_{(1)} \boldsymbol{\Tilde{\tau}} \mathbf{\Tilde{m}}^{u}_{(0)} - \mathbf{\Tilde{m}}^{N+1-u}_{(1)}\boldsymbol{\Tilde{\tau}} \mathbf{\Tilde{m}}^{u-1}_{(0)} \right].
    \end{align*}
\end{enumerate}

Therefore, \begin{align*}
    &\sum_r (1^{\otimes r} \otimes \mu_1 \otimes 1^{\otimes N-r})\mathbf{\Tilde{m}}^{N-u}_{(1)} \boldsymbol{\Tilde{\tau}} \mathbf{\Tilde{m}}^{u-1}_{(0)} \\
    &= \sum_{r=1}^u (-1)^{N-r} (1^{\otimes r} \otimes \mu_2 \otimes 1^{\otimes N-r}) \mathbf{\Tilde{m}}^{N-u}_{(1)} \boldsymbol{\Tilde{\tau}} \mathbf{\Tilde{m}}^{u}_{(0)} \\
    &+ \sum_{r=u}^{N} (-1)^{N-r+1} (1^{\otimes r} \otimes \mu_2 \otimes 1^{\otimes N-r}) \mathbf{\Tilde{m}}^{N+1-u}_{(1)} \boldsymbol{\Tilde{\tau}} \mathbf{\Tilde{m}}^{u-1}_{(0)},
\end{align*}

and \begin{align*}
    &\Psi \partial - \partial \Psi = \sum_N \sum_{u=1}^N \left[ \sum_{r=1}^u (-1)^{u+1+r} \nu_{N+1}( 1^{\otimes r} \otimes \mu_2 \otimes 1^{\otimes N-r}) \mathbf{\Tilde{m}}^{N+1 -(u+1)}_{(1)} \boldsymbol{\Tilde{\tau}} \mathbf{\Tilde{m}}^{(u+1)-1}_{(0)}\right.\\
    & + \sum_{r=u}^N (-1)^{u+r} \nu_{N+1}( 1^{\otimes r} \otimes \mu_2 \otimes 1^{\otimes N-r}) \mathbf{\Tilde{m}}^{N+1-u}_{(1)} \boldsymbol{\Tilde{\tau}} \mathbf{\Tilde{m}}^{u-1}_{(0)}\\
    &- \left. \sum_{r=1}^{N-1} (-1)^{u+r} \nu_{N}( 1^{\otimes r} \otimes \mu_2 \otimes 1^{\otimes N-1-r}) \mathbf{\Tilde{m}}^{N-u}_{(1)} \boldsymbol{\Tilde{\tau}} \mathbf{\Tilde{m}}^{u-1}_{(0)} \right] \\
    &= 0.
\end{align*}

Indeed, for every triple $(N,u,r)$ with $N\geq 3$, $N \geq u$, and $N \geq r$, the term $$(-1)^{u+r}\nu_{N}(1^{\otimes r} \otimes \mu_2 \otimes 1^{\otimes N-r-1})\mathbf{\Tilde{m}}^{N-u}_{(1)} \boldsymbol{\Tilde{\tau}} \mathbf{\Tilde{m}}^{u-1}_{(0)}$$ appears only once with a positive sign (either in the first sum if $r \leq u-1$ or in the second sum) and always once negatively in the third sum.

\end{myproof}

The next proposition is an $\Ai$ analogue to Proposition 2.3.4 of \cite{BDHO23}. It aims to provide homotopies between the maps defined by the previous proposition. We will state this proposition to the same degree of generality as the previous one.

\begin{prop}\label{Prop : homotopy Criterion Ai}

    Let $f_0,f_1 : X \to \R$ be Morse functions.
    Let $\{m^0_{x,y} \in C_{|x|-|y|-1}(\Omega X), \ x,y \in \Crit(f_0)\}$ and $\{m^1_{x',y'} \in C_{|x'|-|y'|-1}(\Omega X), \ x',y' \in \Crit(f_1)\}$ be twisting cocycles.
    Let $\{\tau_{x,y'}\in C_{|x|-|y'|}(\Omega X), \ x \in \Crit(f_0), \ y' \in \Crit(f_1)\}$ and $\{\tau'_{x,y'}\in C_{|x|-|y'|}(\Omega X), \ x \in \Crit(f_0), \ y' \in \Crit(f_1)\}$ be cocycles that satisfy \eqref{eq : mu_1 de tau}, and let $\Psi$ and $\Psi'$ be the morphisms associated with $\{\tau_{x,y'}\}$ and $\{\tau'_{x,y'}\}$ respectively (see Proposition \ref{Prop : continuation morphisms}).

     Suppose that there exists a cocycle $\{h_{x,y'} \in C_{|x|-|y|+1}(\Omega X)\}$ such that 

\begin{equation}\label{eq : mu_1 de h}
    \mu_1 h_{x,y'} = \tau_{x,y'} - \tau'_{x,y'} + \sum_{z \in \textup{Crit}(f_0)} (-1)^{|x|-|z|} m^{0}_{x,z} \cdot h_{z,y'} + \sum_{w' \in \textup{Crit}(f_1)} (-1)^{|x|-|w|} h_{x,w'} \cdot m^{1}_{w',y'}.
\end{equation}

Then the map $H : C_*(X, m^0_{x,y}, \F) \to C_{*+1}(X, m^1_{x,y}, \F)$ defined by
\begin{align*}
    H = \sum_{n \geq 1} \sum_{u=1}^{n} (\nu_{n+1} \otimes 1) \mathbf{\Tilde{m}}_{(1)}^{n-u}\boldsymbol{\Tilde{h}} \mathbf{\Tilde{m}}_{(0)}^{u-1}
\end{align*}

is a chain homotopy between $\Psi$ and $\Psi'$.
\end{prop}

\begin{rem}

    We will refer to such a cocycle $\{h_{x,y}\}$ as \textbf{a homotopy cocycle}.
    
   As for any cocycle, we used the notations
   $$\boldsymbol{\Tilde{h}} \in \Hom_1(\F \otimes TC_*(\Omega X) \otimes \Z\Crit(f_0), \F \otimes TC_*(\Omega X) \otimes \Z\Crit(f_1))$$ for the morphism associated with a homotopy cocycle $\{h_{x,y}\}$. We can rewrite equation \eqref{eq : mu_1 de h} using this morphism :

   \begin{equation}\label{eq: mu_1 of h functional version}
       \Tilde{\mu}_1 \boldsymbol{\Tilde{h}} = \boldsymbol{\Tilde{\tau}} - \boldsymbol{\Tilde{\tau}'} - \Tilde{\mu}_2 \boldsymbol{\Tilde{h}}\m_{(0)} - \Tilde{\mu}_2 \m_{(1)}\boldsymbol{\Tilde{h}}.
   \end{equation}
\end{rem}

\begin{myproof}{of Proposition}{\ref{Prop : homotopy Criterion Ai}}
The proof is very similar to the proof of Proposition \ref{Prop : continuation morphisms}. We compute 

\begin{align*}
        \partial H &= \sum_{\substack{n \geq 1\\ k \geq 0}} \sum_{u=1}^n (\nu_{k+1} \otimes 1) \mathbf{\Tilde{m}}^k_{(1)} (\nu_{n+1} \otimes 1) \mathbf{\Tilde{m}}^{n-u}_{(1)} \boldsymbol{\Tilde{h}} \mathbf{\Tilde{m}}^{u-1}_{(0)}\\
        &= \sum_{\substack{n \geq 1\\ k \geq 0}} \sum_{u=1}^n (-1)^{(n+1)k} (\nu_{k+1} \nu_{n+1} \otimes 1) \mathbf{\Tilde{m}}^{n+k-u}_{(1)} \boldsymbol{\Tilde{h}} \mathbf{\Tilde{m}}^{u-1}_{(0)}
    \end{align*}

    and 

    \begin{align*}
        H \partial &= \sum_{\substack{n \geq 0 \\ k \geq 1}} \sum_{u=1}^k  (\nu_{k+1} \otimes 1)\mathbf{\Tilde{m}}^{k-u}_{(1)} \boldsymbol{\Tilde{h}} \mathbf{\Tilde{m}}^{u-1}_{(0)} (\nu_{n+1} \otimes 1)\mathbf{\Tilde{m}}^{n}_{(0)}\\
        &=\sum_{\substack{n \geq 0\\ k \geq 1}} \sum_{u=1}^k  (-1)^{k(n+1)} (\nu_{k+1} \nu_{n+1} \otimes 1) \mathbf{\Tilde{m}}^{k-u}_{(1)} \boldsymbol{\Tilde{h}} \mathbf{\Tilde{m}}^{n+u-1}_{(0)}\\
        &=\sum_{\substack{n \geq 0\\ k \geq 1}} \sum_{u=n+1}^{n+k} (-1)^{(n +1)k} (\nu_{k+1} \nu_{n+1} \otimes 1) \mathbf{\Tilde{m}}^{n+k-u}_{(1)} \boldsymbol{\Tilde{h}} \mathbf{\Tilde{m}}^{u-1}_{(0)}.
    \end{align*}

    Therefore, by taking $N= n+k$,

    \begin{align*}
         H \partial + \partial H &= - \sum_{N \geq 1} \sum_{u=1}^N \left[ \sum_r (-1)^N \nu_{N+1}(1^{\otimes r} \otimes \mu_1 \otimes 1^{\otimes N-r}) \right. \\
         &+ \left. \sum_r (-1)^r \nu_{N}(1^{\otimes r} \otimes \mu_2 \otimes 1^{\otimes N-1-r})  \right] \mathbf{\Tilde{m}}^{N-u}_{(1)} \boldsymbol{\Tilde{h}} \mathbf{\Tilde{m}}^{u-1}_{(0)}.
    \end{align*}

    We decompose the computation of $(1^{\otimes r} \otimes \mu_1 \otimes 1^{\otimes N-r})\mathbf{\Tilde{m}}^{N-u}_{(1)} \boldsymbol{\Tilde{h}}\mathbf{\Tilde{m}}^{u-1}_{(0)}$ into three cases:

\begin{enumerate}
    \item If $r \leq u-1$, Lemma \ref{lemme : elimination using MC} gives 
    \begin{align*}
        (1^{\otimes r} \otimes \mu_1 \otimes 1^{\otimes N-r})\mathbf{\Tilde{m}}^{N-u}_{(1)} \boldsymbol{\Tilde{h}} \mathbf{\Tilde{m}}^{u-1}_{(0)} = (-1)^{N-r+1} (1^{\otimes r} \otimes \mu_2 \otimes 1^{\otimes N-r})\mathbf{\Tilde{m}}^{N-u}_{(1)} \boldsymbol{\Tilde{h}} \mathbf{\Tilde{m}}^{u}_{(0)}.
    \end{align*}\\
    \item If $r > u$, the same lemma gives
    \begin{align*}
        (1^{\otimes r} \otimes \mu_1 \otimes 1^{\otimes N-r})\mathbf{\Tilde{m}}^{N-u}_{(1)} \boldsymbol{\Tilde{h}} \mathbf{\Tilde{m}}^{u-1}_{(0)} = (-1)^{N-r+1} (1^{\otimes r} \otimes \mu_2 \otimes 1^{\otimes N-r})\mathbf{\Tilde{m}}^{N+1-u}_{(1)} \boldsymbol{\Tilde{h}} \mathbf{\Tilde{m}}^{u-1}_{(0)}.
    \end{align*}\\
    \item If $r = u$, we use equation \eqref{eq: mu_1 of h functional version} to obtain the equality:
    \begin{align*}
         &(1^{\otimes u} \otimes \mu_1 \otimes 1^{\otimes N-u})\mathbf{\Tilde{m}}^{N-u}_{(1)} \boldsymbol{\Tilde{h}} \mathbf{\Tilde{m}}^{u-1}_{(0)} \\
         &= (-1)^{N-u} \mathbf{\Tilde{m}}^{N-u}_{(1)}(\Tilde{\mu}_1\boldsymbol{\Tilde{h}}) \mathbf{\Tilde{m}}^{u-1}_{(0)} \\
         &= (-1)^{N-u} \mathbf{\Tilde{m}}^{N-u}_{(1)}[ \boldsymbol{\Tilde{\tau}} - \boldsymbol{\Tilde{\tau}'} - \Tilde{\mu}_2\boldsymbol{\Tilde{h}}\m_{(0)} - \Tilde{\mu}_2 \m_{(1)}\boldsymbol{\Tilde{h}}]\mathbf{\Tilde{m}}^{u-1}_{(0)}.
    \end{align*}
\end{enumerate}

Therefore, \begin{align*}
    &\sum_r (1^{\otimes r} \otimes \mu_1 \otimes 1^{\otimes N-r})\mathbf{\Tilde{m}}^{N-u}_{(1)} \boldsymbol{\Tilde{h}} \mathbf{\Tilde{m}}^{u-1}_{(0)} \\
    &= \sum_{r=1}^u (-1)^{N-r} (1^{\otimes r} \otimes \mu_2 \otimes 1^{\otimes N-r+1}) \mathbf{\Tilde{m}}^{N-u}_{(1)} \boldsymbol{\Tilde{h}} \mathbf{\Tilde{m}}^{u}_{(0)} \\
    &+ \sum_{r=u}^{N} (-1)^{N-r+1} (1^{\otimes r} \otimes \mu_2 \otimes 1^{\otimes N-r}) \mathbf{\Tilde{m}}^{N+1-u}_{(1)} \boldsymbol{\Tilde{h}} \mathbf{\Tilde{m}}^{u-1}_{(0)}\\
    &+ (-1)^{N-u+1} (1^{\otimes r} \otimes \mu_2 \otimes 1^{\otimes N-r}) \mathbf{\Tilde{m}}^{N-u}_{(1)} \boldsymbol{\Tilde{\tau}} \mathbf{\Tilde{m}}^{u-1}_{(0)} - (-1)^{N-u+1} (1^{\otimes r} \otimes \mu_2 \otimes 1^{\otimes N-r}) \mathbf{\Tilde{m}}^{N-u}_{(1)} \boldsymbol{\Tilde{\tau}'} \mathbf{\Tilde{m}}^{u-1}_{(0)}
\end{align*}

and \begin{align*}
    H \partial + \partial H &= \sum_N \sum_{u=1}^N \left[ \sum_{r=1}^u (-1)^{r} \nu_{N+1}( 1^{\otimes r} \otimes \mu_2 \otimes 1^{\otimes N-r}) \mathbf{\Tilde{m}}^{N+1 -(u+1)}_{(1)} \boldsymbol{\Tilde{h}} \mathbf{\Tilde{m}}^{(u+1)-1}_{(0)}\right.\\
    & + \sum_{r=u}^N (-1)^{r} \nu_{N+1}( 1^{\otimes r} \otimes \mu_2 \otimes 1^{\otimes N-r}) \mathbf{\Tilde{m}}^{N+1-u}_{(1)} \boldsymbol{\Tilde{h}} \mathbf{\Tilde{m}}^{u-1}_{(0)}\\
    &- \left. \sum_{r=1}^{N-1} (-1)^{r} \nu_{N}( 1^{\otimes r} \otimes \mu_2 \otimes 1^{\otimes N-1-r}) \mathbf{\Tilde{m}}^{N-u}_{(1)} \boldsymbol{\Tilde{h}} \mathbf{\Tilde{m}}^{u-1}_{(0)} \right] + \Psi - \Psi'\\
    &= \Psi - \Psi'.
\end{align*}

\end{myproof}

\subsection{Filtration and spectral sequence}\label{subsection : Filtration and spectral sequence}

If $f : X \to \R$ is a Morse function on $X$, $\{m_{x,y} \in C_{|x|-|y|-1}(\Omega X) , \ x,y \in \Crit(f)\}$ is a twisting cocycle and $\F$ is an $\Ai$-module over $C_*(\Omega X)$, then the associated enriched Morse complex has a natural filtration

$$F_p(C_k(X, m_{x,y}, \F)) = \bigoplus_{\substack{i+j = k \\ i \leq p}} \F_j \otimes \Z\Crit_i(f).$$

The spectral sequence $E^r_{p,q}$ associated to this filtration converges to $H_{p+q}(X, \F)$ and its first page is 

$$E^1_{p,q} = H_q(\F) \otimes \Z \Crit_p(f).$$

The structure of $\Ai$ $\Z[\pi_1(X)]$-module on $H_q(\F)$ is associative. Therefore, we can define a differential on this first page in the same way as in the case of a DG-module :

$$d^1(\hat{\alpha} \otimes x) = (-1)^q \sum_{|y| = |x| -1} \hat{\alpha}\cdot\hat{m}_{x,y} \otimes y,$$
where $\hat{m}_{x,y} \in H_0(\Omega X) = \Z[\pi_1(X)]$ is the projection of the twisting cocycle $m_{x,y} \in C_0(\Omega X)$.\\

Denote $\Tilde{C}_p(f) = \Z[\pi_1(X)] \otimes_{\Z} \Z \Crit_p(f)$ the \textbf{lifted Morse complex} of $(f,\xi)$ (see \cite[section 5.1.1]{BDHO23}, \cite{Lat94} and \cite{Dam12}). It is the $\Z[\pi_1(X)]$-complex spanned by preferred lifts $\tilde{x}$ of critical points $x \in \Crit(f)$ in the universal cover $\tilde{X}$ of $X$ and whose differential is given by counting trajectories $\lambda \in \traj{x,y}$ which lift in $\tilde{X}$ to a trajectory starting at $\tilde{x}$ and ending at $g\tilde{y}$ for $g \in \pi_1(X)$ for each pair $x,y \in \Crit(f)$ such that $|x|=|y|+1$.\\

From the \emph{compatibility with lifting} property of the evaluation maps (see Definition \ref{defi : Evaluation maps}), it follows that the differential of the lifted Morse complex can be written
$$d(x) = \sum_{|y| = |x|-1} \hat{m}_{x,y} \otimes  y.$$

Therefore $d^1$ is, up to sign, the differential of the complex $H_q(\F) \otimes \tilde{C}_*(f)$ and the second page is given by $$E^2_{p,q} = H_p(\Tilde{C}_p(f); H_q(\F)).$$

\subsection{Invariance and \texorpdfstring{$\Ai$}{Ai} continuation map}\label{section : Invariance Ai Morse}
Let  $\Xi_0 = (f_0, \xi_0, o_0, s^{(0)}_{x,y}, \mathcal{Y}_0, \theta_0) $ and $\Xi_1 = (f_1, \xi_1, o_1, s^{(1)}_{x,y}, \mathcal{Y}_1, \theta_1)$ be enriched Morse data on $X$ and let $\F$ be a $\Ai$-module over $C_*(\Omega X)$.
We prove that there exists a continuation map $\Psi_{01} : C_*(X,\Xi_0,\F) \to C_*(X,\Xi_1, \F)$ which is a quasi-isomorphism. As a consequence, the homology of the enriched complex with coefficients in a $\Ai$-module does not depend on any of the choices that have been made to define it.

 Following \cite[Lemma 6.2.1]{BDHO23}, one builds a DG Morse set of \textbf{continuation data} $\Xi$ on $X \times [0,1]$ by considering a representing chain system $s_F$ on the moduli spaces of trajectories in $X\times [0,1]$ such that:

    \begin{itemize}
        \item $s^F_{x_0,y_0} = (-1)^{|x|-|y|} s^{(0)}_{x,y}$ for all $x,y \in \Crit(f_0)$.
        \item $s^F_{x_1,y_1} = s^{(1)}_{x,y}$ for all $x,y \in \Crit(f_1)$.
        \item $\sigma_{x,y} := s^F_{x_0,y_1}$ for all $x \in \Crit(f_0)$ and $y \in \Crit(f_1)$.
    \end{itemize}

The family $\{\sigma_{x,y}\}$ can then be evaluated into $\Omega X$
 to build a continuation cocycle $\{\tau_{x,y}\in C_{|x|-|y|}(\Omega X), \ x \in \Crit(f_0), \ y \in \Crit(f_1)\}$ associated to the continuation data $\Xi$. For $i \in \{0,1\}$, denote $m^i_{x,y}$ the Barraud-Cornea twisting cocycle defined by $\Xi_i$. The cocycle $\{\tau_{x,y}\}$ satisfies equation \eqref{eq : mu_1 de tau}.

\begin{defi}\label{defi : Ai continuation map}
    We call $\boldsymbol{\Ai}$ \textbf{continuation map}
    the map induced by $\{\tau_{x,y}\in C_{|x|-|y|}(\Omega X), \ x \in \Crit(f_0), \ y \in \Crit(f_1)\}$
    by Proposition \ref{Prop : continuation morphisms} :
     $$\Psi^{\Xi} : C_*(X,\Xi_0,\mathcal{F}) \to C_*(X,\Xi_1,\mathcal{F})$$ defined by

    $$\Psi^{\Xi} = \sum_{n\geq 1}\sum_{u=1}^n (-1)^{u-1} (\nu_{n+1} \otimes 1)
     \m_{(1)}^{n-u} \Tilde{\boldsymbol{\tau}} \m_{(0)}^{u-1}.$$
\end{defi}

\begin{prop}\label{Prop : QIso Criterion Ai}
Let $\Tilde{\Psi} : \Tilde{C}(f_0, \xi_0) \to \Tilde{C}(f_1, \xi_1)$ be the morphism  between lifted Morse complexes defined by $\Tilde{\Psi}(x) = \sum_{|x|=|y|} n_{x,y} y$, where $n_{x,y}$ is the projection of $\tau_{x,y}$ in $H_0(\Omega X) = \Z[\pi_1(X)]$.

If $\Tilde{\Psi}$ is a quasi-isomorphism, then so is $\Psi.$    
\end{prop}

\begin{proof}
    As in \cite[Proposition 4.4.1]{BDHO23},
    the proof comes from the fact that $\Psi$ preserves the canonical filtrations and therefore induces a morphism of spectral sequences $(\Psi^{(r)})_r$ such that $\Psi^{(1)}= \textup{Id} \otimes \Tilde{\Psi}.$ It then follows from the existence of a spectral sequence for the change of coefficients that, if $\Tilde{\Psi}$ is a quasi-isomorphism, then so is $\Psi^{(1)}$ and therefore $\Psi.$
\end{proof}

The $\Ai$ continuation map $\Psi^{\Xi}$ is therefore a quasi-isomorphism since it is shown in \cite[Proposition 6.2.2]{BDHO23} that the cocycle $\{\tau_{x,y}\}$ satisfies the condition of Proposition \ref{Prop : QIso Criterion Ai}. Hence, the homology  $H_*(X, \F)$ is well-defined if $\F$ is an $\Ai$-module over $C_*(\Omega X)$.\\

In fact more is true : there is an $\Ai$ analogue of the invariance Theorem \ref{thm : continuation morphism Psi_01}.

\begin{thm}\label{thm : Continuation morphism Ai invariance}
Let $\F$ be an $\Ai$-module over $C_*(\Omega X)$.

    1) Given two sets $\Xi_0$ and $\Xi_1$ of enriched Morse data on $X$ and continuation data $\Xi$ on $X \times [0,1]$ , the $\Ai$ continuation map $\Psi^{\Xi} : C_*(X,\Xi_0,\F) \to C_*(X,\Xi_1,\F)$ is a homotopy equivalence and its chain homotopy type only depends on $\Xi_0$ and $\Xi_1$. The map $\Psi^{\Xi}$ is in particular a quasi-isomorphism. \\
    
    2) Given another set of data $\Xi_2$ on  $X$ and denoting $\Psi_{ij}$ the $\Ai$ continuation map between the data $\Xi_i$ and $\Xi_j$, then $\Psi_{00}$ is homotopic to the identity and $\Psi_{02}$ is homotopic to $\Psi_{12}\circ \Psi_{01}$. In particular, in homology

    $$\Psi_{00} = \textup{Id} \ \textup{and} \ \Psi_{12}\circ \Psi_{01} = \Psi_{02}.$$
\end{thm}

Theorem \ref{thm : continuation morphism Psi_01} is proven in \cite[Section 6.2.2]{BDHO23}. Using Proposition \ref{Prop : homotopy Criterion Ai} to construct homotopies between continuation maps, their proof carries over in our setting.

\section{Morphisms of fibrations and \texorpdfstring{$\Ai$}{Ai}-morphism of modules. Proof of Theorem B}\label{section : Functoriality on the coefficient}

The main goal of this section is to prove that a morphism of fibrations $\varphi : E_1 \to E_2$ over $X$ induces a morphism of complexes $\tilde{\varphi} : C_*(X,C_*(F_1)) \to C_*(X, C_*(F_2))$ that is compatible with the Fibration theorem, with direct and shriek maps. The main difficulty is that $\varphi$ does not necessarily respect the transitive lifting functions associated to the fibrations $F_1 \hookrightarrow E_1 \to X$ and $F_2 \hookrightarrow E_2 \to X$ and therefore does not induce a DG morphism of modules over $C_*(\Omega X)$ from $C_*(F_1)$ to $C_*(F_2)$ in general. Hence, the map $\tilde{\varphi} : C_*(X,C_*(F_1)) \to C_*(X, C_*(F_2))$, $\tilde{\varphi}(\alpha \otimes x) = \varphi_*(\alpha) \otimes x$ is \emph{a priori} not a morphism of complexes.
We will prove that if there exists a morphism $\boldsymbol{\varphi} : \F \to \G$ of $\Ai$-modules over $C_*(\Omega X)$, then there exists a morphism of complexes $\tilde{\varphi} : C_*(X,\F) \to C_*(X,\G)$ compatible with direct and shriek maps.

In the context of a fibration, we will prove Theorem B (\ref{thm : morphisme induit commute avec iso}) that states that a morphism of fibrations $\varphi : E_1 \to E_2$ induces a morphism $\boldsymbol{\varphi}: C_*(F_1) \to C_*(F_2)$ of $\Ai$-modules over $C_*(\Omega X)$ and that the map $\tilde{\varphi} : C_*(X,\F) \to C_*(X,\G)$ corresponds to the map $\varphi_* : C_*(E_1) \to C_*(E_2)$ via the Fibration Theorem. Therefore, in this context, $\tilde{\varphi}$ inherits the good functorial behavior of singular complexes in our finite dimensional model of enriched Morse complexes.

\subsection{Functoriality with respect to the coefficients}
 Let $\Xi$ be a set of DG Morse data on $X$ with Morse function $f : \R \to X$.
\begin{prop}\label{prop : Morphisme de Ai module en morphisme de complexe}
    Let $\boldsymbol{\varphi} : (\mathcal{A}, \nu_i^{\mathcal{A}}) \to (\mathcal{B}, \nu_i^{\mathcal{B}})$ be a morphism of $\Ai$-modules over $C_*(\Omega X)$.\\
    Then, the map $\Tilde{\varphi} : C_*(X,\Xi, \mathcal{A}) \to C_*(X,\Xi, \mathcal{B})$ defined by $$\Tilde{\varphi} = \sum_{n \geq 0} (\varphi_{n+1} \otimes 1) \mathbf{\Tilde{m}}^n$$

is a morphism of complexes.
\end{prop}

\begin{proof}

The definition of $\m\in \End_{-1}(\mathcal{A} \otimes TC_*(\Omega X) \otimes \Z\textup{Crit}(f))$ and  is given in \ref{not : morphism tilde m}.

We use equation \eqref{eq : relation morphisme de Ai module} to compute
\begin{align*}
    \partial \Tilde{\varphi} &= \sum_{n,k}(\nu^B_{n+1} \otimes 1)\mathbf{\Tilde{m}}^n (\varphi_{k+1} \otimes 1) \mathbf{\Tilde{m}}^k\\
    &= \sum_{n,k} (-1)^{nk} (\nu^B_{n+1} (\varphi_{k+1} \otimes 1^{\otimes n}) \otimes 1) \mathbf{\Tilde{m}}^{n+k}\\
    &=\sum_{n,k} (-1)^{n(k+1)} (\varphi_{n+1}(\nu^A_{k+1} \otimes 1^{\otimes n}) \otimes 1) \mathbf{\Tilde{m}}^{n+k}\\
    &+ \sum_N \left[ \sum_{r=1}^{N} (-1)^N \varphi_{N+1}(1^{\otimes r} \otimes \mu_1 \otimes 1^{\otimes N-r}) + \sum_{r=1}^{N-1} (-1)^r \varphi_{N}(1^{\otimes r} \otimes \mu_2 \otimes 1^{\otimes N-1-r}) \right] \mathbf{\Tilde{m}}^N
\end{align*}

and 
\begin{align*}
\Tilde{\varphi}\partial &= \sum_{n,k}(\varphi_{n+1} \otimes 1)\mathbf{\Tilde{m}}^n (\nu^A_{k+1} \otimes 1) \mathbf{\Tilde{m}}^k\\
    & =\sum_{n,k} (-1)^{n(k+1)} (\varphi_{n+1}(\nu^A_{k+1} \otimes 1^{\otimes n}) \otimes 1) \mathbf{\Tilde{m}}^{n+k}.
\end{align*}

It then follows from Lemma \ref{lemme : elimination using MC} that $\partial \Tilde{\varphi} - \Tilde{\varphi}\partial = 0$, and $\Tilde{\varphi}$ is indeed a morphism of complexes.
\end{proof}

\begin{notation}
    For any $\boldsymbol{\varphi} : (\mathcal{A}, \nu_i^{\mathcal{A}}) \to (\mathcal{B}, \nu_i^{\mathcal{B}})$ morphism of $\Ai$-modules over $C_*(\Omega X)$, we will denote $$\tilde{\varphi} : C_*(X, \Xi, \mathcal{A}) \to C_*(X, \Xi, \mathcal{B})$$ the associated morphism of complexes. 
\end{notation}

\subsection{Topological modules and morphisms of fibrations. Proof of Theorem B}\label{subsection : Topological modules and morphism of fibrations}

Since many naturally occuring morphisms of $\Ai$-modules are of a topological nature, we define the notion of \textbf{$\mathbf{\Ai}$-morphism of topological modules} over $\Omega X$.

We will denote $\cdot$ the concatenation.

\begin{defi}\label{defi : Ai morphism of topological module}
    Let $(F,\nu_F)$ and $(G,\nu_G)$ be two topological spaces endowed with strictly associative topological module structures over $\Omega X$. An \textbf{$\Ai$-morphism of topological modules} (over $\Omega X$) is the data of a sequence of maps 
    $$\varphi_{n+1} : I^{n} \times F \times \Omega X^{n} \to G, \ n \geq 0$$ such that for all $t_1, \dots, t_n \in I^{n}$, $\alpha \in F$ and $\gamma_1, \dots, \gamma_n \in \Omega X$,

    \begin{equation}\label{eq : Ai morphism of topo modules}
    \begin{split}
    & \varphi_{n+1}(t_1,\dots, t_n, \alpha, \gamma_1, \dots, \gamma_n) = \\
    & \left\{\begin{array}{rl}
      \varphi_{n}(t_2,\dots, t_n, \nu_F(\alpha,\gamma_1),\gamma_2, \dots, \gamma_n)   &  \textup{if} \ t_1=1, \\
       \varphi_{n}(\hat{t}_j , \alpha, \gamma_1, \dots,\gamma_{j-1} \cdot \gamma_j,\dots, \gamma_n)  & \textup{if} \ t_j=1, \ j \geq 2,\\
       \nu_G\left(\varphi_{j}(t_1,\dots, t_{j-1}, \alpha, \gamma_1, \dots, \gamma_{j-1}),\gamma_{j} \cdot \ \dots\ \cdot \gamma_n)\right) & \textup{if} \ t_j=0.
    \end{array}\right.
    \end{split}
    \end{equation}
For $j \geq 2$, we denoted $\hat{t}_j = (t_1, \dots, t_{j-1}, t_{j+1}, \dots, t_n)$.    
\end{defi}

Let $\varphi_{n+1} : I^{n} \times F \times \Omega X^{n} \to G, \ n \geq 0$ be an $\Ai$-morphism of topological modules between two topological spaces $(F,\nu_F)$ and $(G,\nu_G)$  endowed with strictly associative topological module structures over $\Omega X$.\\

Then, the map $\varphi_2 : I \times F \times \Omega X \to G$ satisfies \eqref{eq : Ai morphism of topo modules} for $n=1$, ie $\varphi_2(1,\alpha, \gamma) = \varphi_1 (\nu_F(\alpha,\gamma))$ and $\varphi_2(0, \alpha, \gamma) = \nu_G(\varphi_1(\alpha), \gamma)$ for all $\alpha \in F$ and $\gamma \in \Omega X$. It follows that $\varphi_2$ is a homotopy between $\varphi_1 \circ \nu_F$ and $\nu_G(\varphi_1, \cdot)$ and that $\varphi_1$ has to be a \emph{morphism of topological modules} up to homotopy.\\

For $n=2$, \eqref{eq : Ai morphism of topo modules} reads $$\begin{array}{ccl}
    \varphi_3(1,t,\alpha,\gamma_1,\gamma_2)& = & \varphi_2(t,\nu_F(\alpha, \gamma_1), \gamma_2) \\
    \varphi_3(t,1,\alpha,\gamma_1,\gamma_2)& = & \varphi_2(t,\alpha, \gamma_1 \cdot \gamma_2) \\
    \varphi_3(0,t,\alpha,\gamma_1,\gamma_2)& = & \nu_G(\varphi_1(\alpha), \gamma_1 \cdot \gamma_2) \\
    \varphi_3(t,0,\alpha,\gamma_1,\gamma_2)& = & \nu_G(\varphi_2(t,\alpha, \gamma_1),\gamma_2).
\end{array}$$

\begin{rem}
    The reader accustomed to $\Ai$-theory may expect the maps $\varphi_{n+1}$ to be defined on the multiplihedra $J_n$ introduced in \cite{Sta70}. Since we will mainly work on fibrations, where we can always assume that there exists a transitive lifting function, we will only work on the case where $F$ and $G$ have strictly associative topological module structures on $\Omega X$ itself endowed with a strictly associative multiplication. The next proposition is a proof that cubes are enough to encode coherent homotopies for a map between two strictly associative (topological) modules over a strictly associative (topological) algebra.
\end{rem}

\begin{prop}\label{prop : Ai topo induit Ai}
    An $\Ai$-morphism of topological modules $\{\varphi_{n+1} : I^{n} \times F \times \Omega X^{n} \to G\}$ induces a morphism $\boldsymbol{\phi} :(C_*(F),\partial, \nu_{F,*}) \to (C_*(G), \partial, \nu_{G,*})$ of $\Ai$-modules over $C_*(\Omega X)$, where $\phi_{n+1} : C_*(F) \otimes C_*(\Omega X)^{\otimes n} \to C_*(G)$, the map of degree $n$ induced by $\varphi_{n+1}$, is defined by 

    $$\phi_{n+1}(\alpha \otimes \gamma_1 \otimes \dots \otimes \gamma_n) = \varphi_{n+1,*}(\Id_{I^n} \otimes \alpha \otimes \gamma_1 \otimes \dots \otimes \gamma_n)$$
    
\end{prop}

\begin{proof}
    Let $\mu_* : C_*(\Omega X) \otimes C_*(\Omega X) \to C_*(\Omega X)$ denote the Pontryagin-product.
    We use equation \eqref{eq : Ai morphism of topo modules} in order to compute $\partial \phi_{n+1}$ for all $n \in \N$ :

    \begin{align*}
        &\partial \phi_{n+1}(\alpha \otimes \gamma_1 \otimes \dots \otimes \gamma_n)\\
        &= \left[ \underbrace{\phi_{n}(\nu_{F,*} \otimes 1^{\otimes n-1})}_{t_1=1} + \sum_{r=1}^{n-1} (-1)^r \underbrace{\phi_n(1^{\otimes r} \otimes \mu_* \otimes 1^{\otimes n-1-r})}_{t_{r+1}=1} - (-1)^{n-1} \underbrace{\nu_{G,*} (\phi_{n} \otimes 1)}_{t_n = 0} + \underbrace{0}_{t_r=0, \ r<n} \right. \\
        &+ \left. \sum_{r=0}^{n-1} (-1)^n \phi_{n+1}(1^{\otimes r} \otimes \mu_1 \otimes 1^{\otimes n-r})\right](\alpha \otimes \gamma_1 \otimes \dots \otimes \gamma_n).
    \end{align*}

    Indeed, for any $\delta \in F, \ \tau_1, \dots, \tau_n \in \Omega X$, if there exists $r<n$ such that $t_r = 0$, then $\varphi_{n+1}(t_1, \dots, t_n, \delta, \tau_1, \dots, \tau_n)$ does not depend on $t_{r+1}$. Therefore, $\phi_{n+1}(\alpha \otimes \gamma_1 \otimes \dots \otimes \gamma_n) \lvert_{t_r=0}$ is a degenerate chain.

    This shows that $\boldsymbol{\phi} :(C_*(F),\partial, \nu_{F,*}) \to (C_*(G), \partial, \nu_{G,*}) $ satisfies equation \eqref{eq : relation morphisme de Ai module} and finishes the proof.
    
\end{proof}

Proposition \ref{prop : Morphisme de Ai module en morphisme de complexe} gives a direct Corollary.

\begin{cor}\label{cor : Ai morphism of toplogical modules induces morphism of comlpexes}
    If $\{\varphi_{n+1} : I^{n} \times F \times \Omega X^{n} \to G\}$ is an $\Ai$-morphism of topological modules, then it induces a morphism of complexes $$\tilde{\varphi} : C_*(X, C_*(F)) \to C_*(X, C_*(G)).$$
\end{cor}
\begin{flushright}$\blacksquare$\end{flushright}

Assume that $F$ and $G$ are fibers of fibrations $F \hookrightarrow E_1 \to X$ and $G \hookrightarrow E_2 \to X$ endowed with the $\Omega X$ topological module structures induced by transitive lifting functions.  For the morphism of complexes $\tilde{\varphi} : C_*(X, C_*(F)) \to C_*(X, C_*(G))$ to be compatible with the Fibration Theorem, we will require that $\{\varphi_{n+1} : I^{n} \times F \times \Omega X^{n}  \to G\}$ is induced by a \textbf{morphism of fibrations} $\varphi_1 : E_1 \to E_2$.

\begin{defi}\label{defi : Ai morphism of fibration}
    Let $F \hookrightarrow E_1 \overset{\pi_1}{\to} X$ and $F_2 \hookrightarrow E_2 \overset{\pi_2}{\to} X$ be fibrations over $X$. A \textbf{morphism of fibrations} (over X) is a continuous map $\varphi : E_1 \to E_2$ such that $\pi_2 \circ \varphi = \pi_1$.

\end{defi}

A morphism of fibrations does not in general preserve lifting functions. Indeed, consider $\Omega X \hookrightarrow \ls{X} \overset{\ev}{\to} X$ endowed with the transitive lifting function

$$\Phi : \ls{X} \ftimes{\ev}{\ev_0} \mathcal{P} X \to \ls{X}, \quad \Phi(\alpha,\gamma) = \gamma^{-1} \alpha \gamma$$
and the fibration of the figure-eight space $\Omega X^2 \hookrightarrow \ls{X} \ftimes{\ev}{\ev} \ls{X} \overset{\ev}{\to} X$ endowed with the transitive lifting function
$$ \Phi^2 : (\ls{X} \ftimes{\ev}{\ev} \ls{X}) \ftimes{\ev}{\ev_0} \mathcal{P} X, \quad    \Phi^2((\alpha,\beta), \gamma) = (\gamma^{-1} \alpha \gamma, \gamma^{-1} \beta \gamma).$$

The concatenation $m : (\ls{X} \ftimes{\ev}{\ev} \ls{X}) \to \ls{X}$ does not preserve those transitive lifting functions :

$$m(\Phi((\alpha,\beta),\gamma)) = \gamma^{-1}\alpha \gamma \gamma^{-1} \beta\gamma \neq \gamma^{-1}\alpha\beta\gamma = \Phi(m(\alpha,\beta),\gamma).$$

This makes it difficult to define a morphism $$m : \underbrace{H_*(X, C_*(\Omega X^2))}_{\simeq H_*(\ls{X} \ftimes{\ev}{\ev} \ls{X})} \to \underbrace{H_*(X,C_*(\Omega X))}_{\simeq H_*(\ls{X})}$$ at the chain level, since it would have to be compatible with the module structures, which are themselves defined by the respective transitive lifting functions.

However, Theorem B states that every morphism of fibrations yields a morphism between enriched Morse complexes that corresponds to the morphism between the singular complexes of the total spaces. We will now restate Theorem B and prove it.

\begin{thm}\label{thm : morphisme induit commute avec iso}
   
    Let $F_1 \hookrightarrow E_1 \to X$ and $F_2 \hookrightarrow E_2 \to X$ be two fibrations. Let $\varphi: E_1 \to E_2$ be a morphism of fibrations and $\Xi$ be a set of DG Morse data on $X$.\\
    
    i) There exists a sequence of maps $\left\{\varphi_{n+1} : I^n \times F_1 \times \Omega X^{n-1} \times \mathcal{P}_{\star \to X} X \to E_2\right\}$ called a \textbf{coherent homotopy} for $\varphi$ that induces a morphism of complexes $\tilde{\varphi} : C_*(X,\Xi, C_*(F_1)) \to C_*(X,\Xi, C_*(F_2))$.\\

    ii) A coherent homotopy induces a chain homotopy $v : C_*(X,\Xi,C_*(F_1)) \to C_{*+1}(E_2)$ between $\Psi_2 \circ \tilde{\varphi}$ and $\varphi_* \circ \Psi_1$ for any set of DG Morse data $\Xi$ on $X$  where the morphisms $\Psi_1$ and $\Psi_2$ are the quasi-isomorphisms given by the Fibration Theorem. In other words, the following diagram commutes up to chain homotopy

    \[
    \xymatrix{
    C_*(X,\Xi, C_*(F_1)) \ar[r]^-{\Tilde{\varphi}} \ar[d]^{\Psi_1} & C_*(X,\Xi, C_*(F_2)) \ar[d]_{\Psi_2} \\
    C_*(E_1) \ar[r]_{\varphi_*} & C_*(E_2).
    }
    \]
\end{thm}

\begin{myproof}{of}{i)}
 We will prove that even if \emph{a priori} 

$$\begin{array}{ccc}
F_1 \times \mathcal{P}_{\star \to X}X & \to & E_2\\
(\alpha,\gamma) & \mapsto & \Phi_2(\varphi(\alpha),\gamma)
\end{array} \neq 
\begin{array}{ccc}  F_1 \times \mathcal{P}_{\star \to X}X & \to & E_2\\
(\alpha,\gamma) & \mapsto & \varphi(\Phi_1(\alpha,\gamma)) 
\end{array},$$
these maps are always homotopic and there exists a family of higher homotopies as defined in the following lemma.

\begin{lemme}\label{lemme : morphism of fibrations induces coherent homotopy}
    Let $F_1 \hookrightarrow E_1 \overset{\pi_1}{\to} X $ and $F_2 \hookrightarrow E_2 \overset{\pi_2}{\to} X$ be two fibrations over $X$ endowed with transitive lifting functions $\Phi_1 : E_1 \times \mathcal{P}_{\star \to X} X \to E_1$ and $\Phi_2 : E_2 \times \mathcal{P}_{\star \to X} X \to E_2$.
    
    For any morphism of fibrations $\varphi : E_1 \to E_2$, there exists a sequence of maps $$\varphi_{n+1} : I^n \times F_1 \times \Omega X^{n-1} \times \mathcal{P}_{\star \to X} X \to E_2$$ for $n\geq 1$ such that $\varphi = \varphi_1$, $\pi_2 \circ \varphi_{n+1}(t_1,\dots, t_n, \alpha, \gamma_1, \dots, \gamma_{n}) = \ev_1(\gamma_n)$ and
    
    \begin{equation}\label{eq : Ai morphism of fibration}
    \begin{split}
    & \varphi_{n+1}(t_1,\dots, t_n, \alpha, \gamma_1, \dots, \gamma_n) = \\
    & \left\{\begin{array}{rl}
      \varphi_{n}(t_2,\dots, t_n, \Phi_1(\alpha,\gamma_1),\gamma_2, \dots, \gamma_{n})   &  \textup{if} \ t_1=1, \\
       \varphi_{n}(\hat{t}_j , \alpha, \gamma_1, \dots,\gamma_{j-1} \cdot \gamma_j,\dots, \gamma_{n})  & \textup{if} \ t_j=1, \ j \geq 2,\\
       \Phi_2\left(\varphi_{j}(t_1,\dots, t_{j-1}, \alpha, \gamma_1, \dots, \gamma_{j-1}),\gamma_{j} \cdot \ \dots \ \cdot \gamma_{n})\right) & \textup{if} \ t_j=0.
    \end{array}\right.
    \end{split}
    \end{equation} 
 
\end{lemme}

\begin{defi}\label{defi : coherent homotopy}
    We say that a family $$\left\{\varphi_{n+1} : I^n \times F_1 \times \Omega X^{n-1} \times \mathcal{P}_{\star \to X} X \to E_2\right\}$$ satisfying equation \eqref{eq : Ai morphism of fibration} is a \textbf{coherent homotopy} for the morphism of fibrations $\varphi = \varphi_1$.
\end{defi}

Before proving this lemma, we will see how it concludes the proof of i).

If $\varphi: E_1 \to E_2$ is a morphism of fibrations, Lemma \ref{lemme : morphism of fibrations induces coherent homotopy} builds a coherent homotopy $\{\varphi_{n+1} : I^n \times F_1 \times \Omega X^{n-1} \times \mathcal{P}_{\star \to X}X \to E_2, \ n\geq 1\}$. It suffices to check that a coherent homotopy induces an $\Ai$-morphism of topological modules $\{\varphi^{F_1}_{n+1} : I^n \times F_1 \times \Omega X^{n-1} \times \Omega X \to F_2, n \geq 0\}$ where $\varphi^{F_1}_1 = \varphi\lvert_{F_1} : F_1 \to F_2$ and $\varphi^{F_1}_{n+1} = \varphi_{n+1}\lvert_{I^n \times F_1 \times \Omega X^{n}}$.
Indeed, it would prove that $\varphi \lvert_{F_1} : F_1 \to F_2 $ induces a
morphism $$\tilde{\varphi} : C_*(X,\Xi, C_*(F_1)) \to C_*(X,\Xi, C_*(F_2))$$ by Corollary \ref{cor : Ai morphism of toplogical modules induces morphism of comlpexes}.

     First remark that since $\pi_2\circ \varphi = \pi_1$, $\varphi(F_1) \subset F_2$ and therefore $\varphi\lvert_{F_1} : F_1 \to F_2$. 
     Moreover, since for any $t_1,\dots,t_n, \alpha, \gamma_1, \dots, \gamma_{n-1}, \tau \in I^n \times F_1 \times \Omega X^{n-1} \times \mathcal{P}_{\star \to X}X$, $\pi_2 \circ \varphi_{n+1}(t_1,\dots, t_n, \alpha, \gamma_1, \dots, \gamma_{n}) = \ev_1(\gamma_n)$, if $\gamma_n \in \Omega X$, then $\varphi_{n+1}(t_1,\dots, t_n, \alpha, \gamma_1, \dots, \gamma_{n}) \in F_2.$

     The module structure on $C_*(F_1)$ and $C_*(F_2)$ are defined by the lifting functions $\Phi_1$ and $\Phi_2$. Therefore since $\{\varphi_{n+1} : I^n \times F_1 \times \Omega X^{n-1} \times \mathcal{P}_{\star \to X}X \to E_2\}$ satisfies equation \eqref{eq : Ai morphism of fibration}, the family $\{\varphi^{F_1}_{n+1} : I^n \times F_1 \times \Omega X^{n} \to F_2\}$ satisfies equation \eqref{eq : Ai morphism of topo modules} and indeed defines an $\Ai$-morphism of topological modules over $\Omega X$.\\

\begin{myproof}{of Lemma}{\ref{lemme : morphism of fibrations induces coherent homotopy}}

Before defining the whole coherent homotopy for the morphism of fibrations $\varphi : E_1 \to E_2$, let us define $\varphi_2 : I \times F_1 \times \mathcal{P}_{\star \to X} X \to E_2$. This map is a homotopy between $\varphi_2(0, \cdot) : (\alpha,\gamma) \mapsto \Phi_2(\varphi(\alpha),\gamma)$ and $\varphi_2(1, \cdot) : (\alpha,\gamma) \mapsto \varphi(\Phi_1(\alpha,\gamma))$. 

Let $\alpha \in F_1$ and $\gamma : [0,s] \to X \in \mathcal{P}_{\star \to X}X.$
This homotopy is defined by 
$$\varphi_2(t,\alpha,\gamma) = \Phi_2(\varphi(\Phi_1(\alpha,\gamma\lvert_{[0,ts]})),\gamma_{[ts,s]}).$$

It lifts a part of $\gamma$ in $E_1$ starting from $\alpha$, sends $\Phi_1(\alpha,\gamma\lvert_{[0,ts]}) $ into $E_2$ by $\varphi$ and lifts the rest of $\gamma$ starting from $\varphi(\Phi_1(\alpha,\gamma\lvert_{[0,ts]})).$ See Figure \hyperref[fig : 2]{2}.

We now define $\varphi_{n+1} : I^n \times F_1 \times \Omega X^{n-1} \times \mathcal{P}_{\star \to X}X \to E_2$ for all $n\geq 1$.
Let us first fix some notations.

$\bullet$ Let $J_1 : I \times [0,+\infty) \to  [0,+\infty)$

$$J_1(t,s) = ts.$$

For any $n \geq 2$, define $J_{n} : I^n \times [0, + \infty)^{n} \to [0, +\infty)$ by

\begin{align*}
    J_n(t_1, \dots,t_n, s_1, \dots s_n) &= J_{n-1}(t_1, \dots, t_{n-1},s_1, \dots, s_{n-1} + t_{n}s_n)\\
    &= t_1( s_1 + t_2(s_2 + t_3(s_3 + \dots + t_{n-1}( s_{n-1} +t_{n}s_{n}) \dots))).
\end{align*}

By convention, we set $J_0 = 0.$
For any $n \in \N$, denote $$I_n^0(t_1, \dots,t_n, s_1, \dots s_n) = [0, J_n(t_1, \dots,t_n, s_1, \dots s_n)]$$ and $$I_n^1(t_1, \dots,t_n, s_1, \dots s_n) = [J_n(t_1, \dots,t_n, s_1, \dots s_n), s_1 + \dots +s_n].$$

$\bullet$  In the following arguments, for any $\gamma_i \in \mathcal{P}_{\star \to X} X$ we will always denote $a_i\geq 0$ such that $\gamma_i : [0,a_i] \to X$.

Let $r_1 : I \times E_1 \times \mathcal{P}_{\star \to X} X \to E_2$, $$r_1(t,\alpha,\gamma_1) = \varphi\left(\Phi_1(\alpha, \gamma_1 \lvert_{[0,ta_1]})\right).$$

\begin{center}
\begin{overpic}[scale = 0.5]{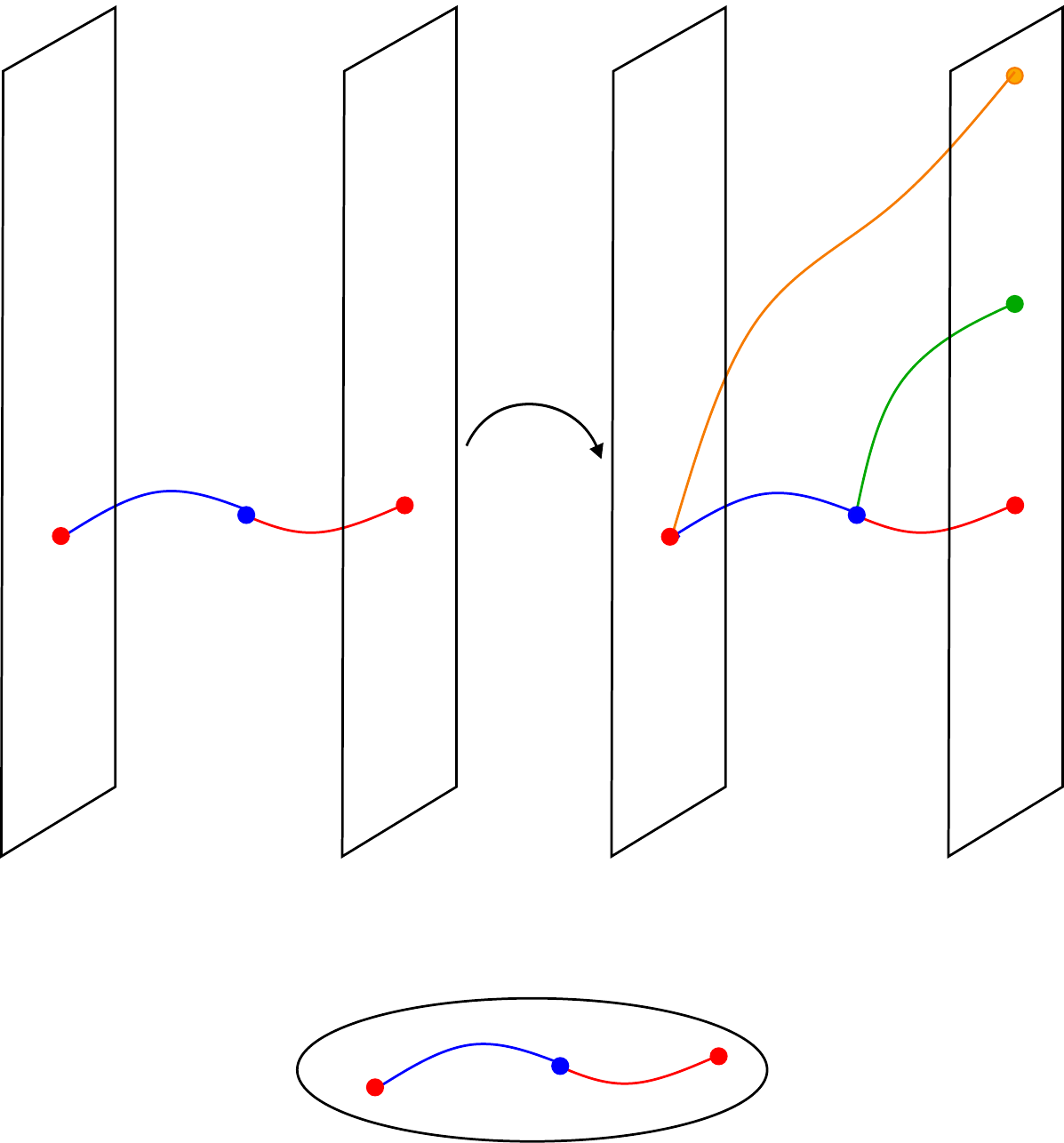}\label{fig : 2}
    \put(-8,53){$E_1$}
    \put(100,53){$E_2$}
    \put(70,5){$X$}
    \put(46,66){$\varphi$}
    \put(3,48){\textcolor{red}{$\alpha$}}
    \put(31,50){\red{$\Phi_1(\alpha,\gamma)$}}
    \put(55,48){\red{$\varphi(\alpha)$}}
    \put(50.5,8.5){$\gamma$}
    \put(40,2.5){\blue{$\gamma \lvert_{[0,ts]}$}}
    \put(53,2.5){\red{$\gamma\lvert_{[ts,s]}$}}
    \put(94,62){\red{$\varphi(\Phi_1(\alpha,\gamma))$}}
    \put(94,58){\red{$= \varphi_2(1,\alpha,\gamma)$}}
    \put(63,50){\blue{$\varphi(\Phi_1(\alpha,\gamma\lvert_{[0,ts]}))$}}
    \put(11,60){\blue{$\Phi_1(\alpha,\gamma\lvert_{[0,ts]})$}}
    \put(78.2,75){\textcolor{darkgreen}{$\Phi_2(\varphi(\Phi_1(\alpha,\gamma\lvert_{[0,ts]})),\gamma\lvert_{[ts,s]})$}}
    \put(95,71){\textcolor{darkgreen}{$= \varphi_2(t,\alpha,\gamma)$}}
    \put(68,100){\textcolor{orange}{$\Phi_2(\varphi(\alpha),\gamma)$}}
    \put(68,96){\textcolor{orange}{$=\varphi_2(0,\alpha,\gamma) $}}
\end{overpic}\\
\paragraph{}
Figure 2 - Homotopy $\varphi_2$
\end{center}

For any $n \geq 2$, define $r_{n} : I^n \times F_1 \times \Omega X^{n-1} \times \mathcal{P}_{\star \to X} X \to E_2$ by

\begin{align*}
    r_n(t_1,\dots,t_n,\alpha,\gamma_1,\dots,\gamma_n) &= r_{n-1}\left(t_1,\dots,t_{n-1}, \alpha, \gamma_1, \dots, \gamma_{n-1}\cdot \left(\gamma_n\lvert_{[0,t_na_n]}\right)\right)\\
    &=\varphi\left(\Phi_1\left(\alpha, (\gamma_1\cdot \ \dots \ \cdot \gamma_n)\lvert_{I^0_n(t_1,\dots,t_n,a_1,\dots,a_n)}\right)\right).
\end{align*}

By convention, we set $r_0 =\varphi : E_1 \to E_2$.

$\bullet$ Define the coherent homotopy in the following way :

Let $n \geq 1$. Define $\varphi_{n+1} : I^n \times F_1 \times \Omega X^{n-1} \times \mathcal{P}_{\star \to X} \to E_2$ by 
$$\varphi_{n+1}(t_1,\dots,t_n,\alpha, \gamma_1, \dots, \gamma_n) = \Phi_2\left( r_n(t_1, \dots,t_n,\alpha, \gamma_1, \dots, \gamma_n), (\gamma_1\cdot \ \dots \ \cdot \gamma_n)\lvert_{I^1_n(t_1, \dots, t_n, a_1, \dots, a_n)}\right).$$

We now prove that this satisfies the three equations of \eqref{eq : Ai morphism of fibration}.

First, we remark that for any $n \geq 1$, $j\geq 2$, $k\geq 1$, $t_1,t_2, \dots, t_n \in I^n$ and $s_1,\dots,s_n \in [0,+\infty)^{n},$ 

\begin{equation}\label{eq : Jn(1,...)}
    J_n(1,t_2,\dots, t_n,s_1,\dots,s_n) = s_1 + J_{n-1}(t_2,\dots, t_{n},s_2,\dots,s_n),
\end{equation}

\begin{equation}\label{eq : Jn(tj=1,...)}
    J_n(t_j=1, s_1, \dots,s_n) = J_{n-1}(\hat{t}_j, s_1,\dots, s_{j-1}+s_j, \dots, s_n),
\end{equation}

\begin{equation}\label{eq : Jn(tk=0,...)}
    J_n(t_k=0, s_1, \dots, s_n) = J_{k-1}(t_1, \dots, t_{k-1},s_1, \dots, s_{k-1}),
\end{equation}

where $\hat{t}_j = (t_1, \dots, t_{j-1},t_{j+1}, \dots, t_n).$\\

\textbf{\underline{First equation :}} For any $n \geq 1$, $t_1,t_2, \dots, t_n \in I^n$, $\alpha \in F_1$ and $\gamma_1,\dots,\gamma_n \in \Omega X^{n} \times \mathcal{P}_{\star \to X} X,$
using \eqref{eq : Jn(1,...)}, 

$$(\gamma_1\cdot \ \dots \ \cdot \gamma_n)\lvert_{I^0_n(1,t_2, \dots, t_n, a_1, \dots, a_n)} = \gamma_1 \cdot (\gamma_2 \cdot \ \dots \ \cdot \gamma_n)\lvert_{I^0_{n-1}(t_2, \dots, t_n, a_2, \dots,a_n)}$$
and
$$(\gamma_1\cdot \ \dots \ \cdot \gamma_n)\lvert_{I^1_n(1,t_2, \dots, t_n, a_1, \dots, a_n)} = \left\{ \begin{array}{ll}
    (\gamma_2 \cdot \ \dots \ \cdot \gamma_n)\lvert_{I^1_{n-1}(t_2, \dots, t_n, a_2, \dots, a_n)} & \textup{if } n \geq 2  \\
    \ev_1(\gamma_1) & \textup{if } n=1. 
\end{array} \right..$$

Therefore,

\begin{align*}
    r_n(1,t_2, \dots, t_n, \alpha , \gamma_1, \dots, \gamma_n) &= \varphi\left( \Phi_1\left(\alpha, \gamma_1 \cdot (\gamma_2 \cdot \ \dots \ \cdot \gamma_n)\lvert_{I^0_{n-1}(t_2, \dots, t_n, a_2, \dots,a_n)}\right) \right)\\
    &= \varphi\left( \Phi_1\left( \Phi_1(\alpha,\gamma_1), (\gamma_2 \cdot \ \dots \ \cdot \gamma_n)\lvert_{I^0_{n-1}(t_2, \dots, t_n, a_2, \dots,a_n)}\right) \right)\\
    &= r_{n-1}(t_2, \dots, t_n, \alpha \cdot \gamma_1 , \gamma_2, \dots, \gamma_n)
\end{align*}
and

$$\varphi_{n+1}(1, t_2, \dots, t_n, \alpha, \gamma_1, \dots, \gamma_n) = \varphi_{n}(t_2, \dots, t_n, \alpha \cdot \gamma_1, \gamma_2, \dots, \gamma_n).$$

\textbf{\underline{Second equation :}} 
For any $n \geq 2$, $j\geq 2$, $t_1,t_2, \dots, t_n \in I^n$, $\alpha \in F_1$ and $\gamma_1,\dots,\gamma_n \in \Omega X^{n} \times \mathcal{P}_{\star \to X} X$, using \eqref{eq : Jn(tj=1,...)},

$$(\gamma_1\cdot \ \dots \ \cdot \gamma_n)\lvert_{I^i_n(t_j=1, a_1, \dots, a_n)} = (\gamma_1 \cdot \ \dots \ \cdot(\gamma_{j-1} \cdot \gamma_j) \cdot \ \dots \ \cdot \gamma_n)\lvert_{I^i_{n-1}(\hat{t}_j, a_2, \dots, a_{j-1}+a_j, \dots a_n)}$$
for $i \in \{0,1\}$.
Hence, $$\varphi_{n+1}(t_j =1, \alpha, \gamma_1, \dots, \gamma_n) = \varphi_{n}(\hat{t}_j, \alpha, \gamma_1, \dots, \gamma_{j-1} \cdot \gamma_j, \dots, \gamma_n).$$

\textbf{\underline{Third equation :}} For any $n \geq 1$, $k\geq 1$, $t_1,t_2, \dots, t_n \in I^n$, $\alpha \in F_1$ and $\gamma_1,\dots,\gamma_n \in \Omega X^{n} \times \mathcal{P}_{\star \to X} X,$ using \eqref{eq : Jn(tk=0,...)},

$$(\gamma_1\cdot \ \dots \ \cdot \gamma_n)\lvert_{I^0_n(t_k=0, a_1, \dots, a_n)} = (\gamma_1 \cdot \ \dots \ \cdot\gamma_{k-1})\lvert_{I^0_{k-1}(t_1, \dots, t_{k-1}, a_1, \dots, a_{k-1})}$$

and $$(\gamma_1\cdot \ \dots \ \cdot \gamma_n)\lvert_{I^1_n(t_k=0, a_1, \dots, a_n)} = (\gamma_1 \cdot \ \dots \ \cdot \gamma_{k-1})\lvert_{I^1_{k-1}(t_1, \dots, t_{k-1},a_1, \dots, a_{k-1})} \cdot \gamma_k \cdot \ \dots \ \cdot\gamma_{n}.$$
Therefore, 

\begin{align*}
    &\varphi_{n+1}(t_k = 0, \alpha, \gamma_1, \dots, \gamma_n)\\
    &= \Phi_2\left( r_{k-1}(t_1, \dots, t_{k-1},\alpha,\gamma_1, \dots, \gamma_{k-1}), (\gamma_1 \cdot \ \dots \ \cdot \gamma_{k-1})\lvert_{I^1_{k-1}(t_1, \dots, t_{k-1},a_1, \dots, a_{k-1})} \cdot \ \gamma_k \cdot \ \dots \ \cdot \gamma_n  \right)\\
    &= \Phi_2\left( \Phi_2\left( r_{k-1}(t_1, \dots, t_{k-1},\alpha,\gamma_1, \dots,  \gamma_{k-1}), (\gamma_1 \cdot \ \dots \ \cdot \gamma_{k-1})\lvert_{I^1_{k-1}(t_1, \dots, t_{k-1},a_1, \dots, a_{k-1})} \right),\gamma_k \cdot \ \dots \ \cdot \gamma_n  \right)\\
    &= \Phi_2(\varphi_{k}(t_1, \dots, t_{k-1}, \alpha, \gamma_1, \dots, \gamma_{k-1}), \gamma_k \cdot \ \dots \ \cdot \gamma_n).
\end{align*}

\end{myproof}

This concludes the proof of i).

\end{myproof}

\begin{myproof}{of}{ii)}
    
To finish the proof of Theorem \ref{thm : morphisme induit commute avec iso}, we now prove that such a coherent homotopy induces a chain homotopy $v : C_*(X,\Xi,C_*(F_1)) \to C_{*+1}(E_2)$ between $\Psi_2 \circ \tilde{\varphi}$ and $\varphi_* \circ \Psi_1$. 

For $i \in \{0,1\}$, we denote $\F_i = C_*(F_i)$ endowed with the module structure induced by a transitive lifting function associated to $F_i \hookrightarrow E_i \to X$. Let $\Xi$ be a set of DG Morse data on $X$ with Morse function $f$. Let $\{s_{x,y} \in C_{|x|-|y|-1}(\trajb{x,y}), \ x,y \in \Crit(f)\}$ be a representing chain system and for each $x,y \in \Crit(f)$.

Define $\Gamma_{x,y} : \trajb{x,y} \to \mathcal{P}_{x \to y} X$ to be the parametrization map by the values of $f.$
The twisting cocycle $\{m'_{x,y} \in C_{|x|-|y|-1}(\Omega(\Xq)), \ x,y \in \Crit(f)\}$ is defined by $m'_{x,y} = p_* \circ \Gamma_{x,y,*}(s_{x,y})$, so that the Barraud-Cornea twisting cocycle satisfies $m_{x,y} = \theta_* m'_{x,y} \in C_{|x|-|y|-1}(\Omega X)$.
The twisting cocycle $\{m'_{x,y}\}$ satisfies the Maurer-Cartan equation \eqref{eq : Maurer-Cartan}

$$\partial m'_{x,y} = \sum_z (-1)^{|x|-|z|} m'_{x,z} m'_{z,y}.$$

We use the definition of the morphisms $\Psi_1$ and $\Psi_2$ from in Section \ref{subsection : Fibration Theorem}. Each of these morphisms is defined as a composition of three maps :

$$C_*(X,m_{x,y},\F_1) \overset{\sim}{\leftrightarrow} C_*(X, m'_{x,y},\theta^{*} \F_1) \overset{\Psi_1}{\to} C_*(\theta^{*}E_1) \to  C_*(E_1),$$

$$C_*(X,m_{x,y},\F_2) \overset{\sim}{\leftrightarrow} C_*(X, m'_{x,y}, \theta^{*} \F_2) \overset{\Psi_2}{\to} C_*(\theta^{*}E_2) \to  C_*(E_2).$$

We state here a lemma proving that morphisms of fibrations are stable by pullback.

\begin{lemme}\label{lemme : morphism induit poussé 
par une application continue}
Let $(Y, \star_Y)$ be a smooth, oriented, pointed, closed and connected manifold.
Let $F \hookrightarrow E \overset{\pi}{\to} X$ and $F' \hookrightarrow E'  \overset{\pi'}{\to} X$ be fibrations,  $F \hookrightarrow E_Y \to Y$, $F' \hookrightarrow E'_Y \to Y$ be the pullback fibrations by a continuous map $\psi : Y \to X$ such that $\psi(\star_Y) = \star$.  If $\varphi : E \to E'$ is a morphism of fibrations over $X$, then 

$$\psi^*\varphi : E_Y \to E'_Y, \ \psi^*\varphi(y, \alpha) = (y, \varphi(\alpha))$$

is a morphism of fibrations over $Y$.
\end{lemme}

\begin{proof}
Recall that $$E_Y = \{(y,\alpha) \in Y \times E, \ \pi(\alpha) = \psi(y)\}$$ $$E'_Y = \{(y,\alpha') \in Y \times E', \ \pi'(\alpha') = \psi(y)\}$$ and $\pi_Y : E_Y \to Y$, $\pi'_Y : E'_Y \to Y$ are defined by $\pi_Y(y,\alpha) = y$ and $\pi'_Y(y,\alpha') = y$.
The statement is therefore clear since for any $(y,\alpha) \in E_Y$, $\pi'(y,\varphi(\alpha)) = y = \pi(y,\alpha).$
\end{proof}

We will show that the following diagram is commutative up to chain homotopy.

$$
\xymatrix{
C_*(X,m_{x,y},\F_1) \ar[r]_-{\sim}\ar[d]^-{\Tilde{\varphi}} & C_*(X, m'_{x,y}, \theta^{*}\F_1) \ar[r]_-{\Psi_1}\ar[d]^-{\theta^*\Tilde{\varphi}} & C_*(\theta^{*}E_1) \ar[d]_-{\theta^*\varphi} \ar[r]_{\theta_*} & C_*(E_1) \ar[d]_-{\varphi}\\
C_*(X, m_{x,y}, \F_2) \ar[r]^-{\sim} & C_*(X, m'_{x,y}, \theta^*\F_2) \ar[r]^-{\Psi_2} & C_*(\theta^* E_2) \ar[r]^{\theta_*} & C_*(E_2).
}
$$

Since $\F_i \simeq \theta^*\F_i$ is just an identification of DG modules, it is clear that the first square commutes. The third square commutes by definition of the map $\theta^*\varphi : C_*(\theta^*E_1) \to C_*(\theta^*E_2)$ induced by $\varphi : C_*(E_1) \to C_*(E_2)$.\\

In order to complete this proof, we now have to prove that the middle square commutes up to chian homotopy :\\

We use the notation $\mathbf{m} \in \Hom_{-1}(\Z\Crit(f), C_*(\Omega(\Xq)) \otimes \Z\Crit(f))$ for the map
$\mathbf{m}(x) = \sum_y m'_{x,y} \otimes y$. We will extend $\mathbf{m}$ in the same manner as in Section \ref{subsection : Complex with coefficients in an Ai module}  to $$\mathbf{\Tilde{m}}\in \End_{-1}(C_*(F_1) \otimes TC_*(\Omega (\Xq)) \otimes \Z\textup{Crit}(f))$$ defined by
$$\mathbf{\Tilde{m}}(\alpha \otimes \sigma_1 \otimes \dots \otimes \sigma_k \otimes x) = (1^{\otimes k+1} \otimes \mathbf{m})(\alpha \otimes \sigma_1 \otimes \dots \otimes \sigma_k \otimes x).$$

Given $x\in \Crit(f)$, we can define $m_x \in C_{|x|}(\mathcal{P}_{\star \to \Xq} \Xq)$ by evaluating on $\Xq$ a suitable representative of the fundamental class of the Latour cell $\wb{u}{x}$  (see \cite[Lemma 7.3.2 and Lemma 7.3.3]{BDHO23}). It satisfies the equation

\begin{equation}\label{eq : m_x}
    \partial m_x = \sum_{y \in \Crit(f)} m'_{x,y} m_y.
\end{equation}

We denote $\mathbf{m^L} \in \Hom_0(\Z\Crit(f), C_*(\mathcal{P}_{\star \to \Xq} \Xq))$ the morphism defined by $$\mathbf{m^L}(x) = m_x \in C_{|x|}(\mathcal{P}_{\star \to \Xq} \Xq)$$ for all $x \in \textup{Crit}(f)$ and we extend it to \\
$$\mathbf{\Tilde{m}^L} \in \Hom_0(\F_i \otimes TC_*(\Omega (\Xq)) \otimes \Z\Crit(f), \F_i \otimes TC_*(\Omega (\Xq)) \otimes C_*(\mathcal{P}_{\star \to \Xq} \Xq))$$
by
$$\mathbf{\Tilde{m}^L}(\alpha \otimes \sigma_1 \otimes \dots \otimes \sigma_k \otimes x) = \alpha \otimes \sigma_1 \otimes \dots \otimes \sigma_k \otimes m_x.$$

Then, for $i \in \{1,2\}$, the definition of the quasi-isomorphism $\Psi_i : C_*(X,m'_{x,y}, \theta^*\F_i) \to C_*(\theta^*E_1)$ can be written
\begin{equation}\label{eq : preuve theorem C def iso fibration thm}
    \Psi_i = \Phi_{i,*} \mathbf{\Tilde{m}^L}.
\end{equation}

Let $\left\{\varphi_{n+1} : I^n \times F_1 \times \Omega (\Xq)^{n-1} \times \mathcal{P}_{\star \to \Xq} \Xq \to \theta^*E_2\right\}$ be a coherent homotopy for the morphism of fibrations $\theta^*\varphi$.
Denote $\phi_1 = \theta^*\varphi_{*} : C_*(\theta^*E_1) \to C_*(\theta^*E_2)$ and for all $n\geq 1$, define $$\phi_{n+1} :\theta^*\F_1 \otimes C_*(\Omega (\Xq))^{n-1} \otimes C_*(\mathcal{P}_{\star \to \Xq} \Xq) \to C_*(\theta^*E_2)$$ by $$\phi_{n+1}(\alpha \otimes \sigma_1 \otimes \dots \otimes \sigma_n) = \varphi_{n+1,*}(\Id_{I^n} \otimes \alpha \otimes \sigma_1 \otimes \dots \otimes \sigma_n).$$

Therefore $$\theta^*\tilde{\varphi} = \sum_{n\geq 0} \phi_{n+1} \m^{n}. $$

The proof of Proposition \ref{prop : Ai topo induit Ai} also proves the following lemma.

\begin{lemme}
    If we denote $\nu_1^A =\partial_{F_1}$ and $\nu_1^B = \partial$ the differentials of $\mathcal{A} = \theta^* C_*(F_1)$ and $\mathcal{B} = C_*(\theta^*E_2)$, $\nu^A_2 = \Phi_1$, $\nu^B_2 = \Phi_2$ and $\nu^A_k= \nu^B_k = 0$ for $k \geq 3$, then the $\Ai$-relation \eqref{eq : relation morphisme de Ai module} is satisfied as a functional equality on maps $$\mathcal{A} \otimes C_*(\Omega (\Xq))^{n-1} \otimes C_*(\mathcal{P}_{\star \to \Xq} \Xq) \to \mathcal{B}$$ for the family $$\{\phi_{n+1} :\mathcal{A} \otimes C_*(\Omega X)^{n-1} \otimes C_*(\mathcal{P}_{\star \to X} X) \to \mathcal{B}, n\geq 1 \}.$$ 
\end{lemme}
\begin{flushright}
    $\blacksquare$
\end{flushright}

In this case, \eqref{eq : relation morphisme de Ai module} is written for all $N \geq 1$,

\begin{equation}\label{eq : preuve Theorem C}
\begin{split}
    \partial \phi_{N+1} &= (-1)^{N}\Phi_2(\phi_{N} \otimes 1) + (-1)^{N} \phi_{N+1} (\partial_{F_1} \otimes 1^{\otimes N}) + \phi_{N}(\Phi_1 \otimes 1^{\otimes N-1})\\
    &+ \sum_{r=1}^{N} (-1)^N \phi_{N+1}(1^{\otimes r} \otimes \mu_1 \otimes 1^{\otimes N-r}) 
    + \sum_{r=1}^{N-1} (-1)^r \phi_{N}(1^{\otimes r} \otimes \mu_2 \otimes 1^{\otimes N-1-r}).
\end{split}
\end{equation}

We now define
$$v = \sum_{n \geq 1} (-1)^{n+1} \phi_{n+1} \mathbf{\Tilde{m}^L} \mathbf{\Tilde{m}}^{n-1} :C_*(X,m'_{x,y},\theta^*\F_1) \to C_{*+1}(\theta^*E_2)$$ and show that $$\varphi_*\Psi_1 - \Psi_2 \theta^*\Tilde{\varphi} = \partial v + v\partial.$$

We will denote $\mu_1$ the differential on $C_*(\Omega (\Xq))$ as well as the differential on $ C_*(\mathcal{P}_{\star \to \Xq}\Xq)$ and $\mu_2$ will denote the Pontryagin product $\mu_2 : C_*(\Omega (\Xq)) \otimes C_*(\Omega (\Xq)) \to C_*(\Omega (\Xq))$ as well as the concatenation $\mu_2 : C_*(\Omega (\Xq)) \otimes C_*(\mathcal{P}_{\star \to \Xq}\Xq) \to C_*(\mathcal{P}_{\star \to \Xq}\Xq).$
Let us first remark that the equation \eqref{eq : m_x} can be written \begin{equation}\label{eq : sys repr cell Ai}
    \mu_1 \mathbf{m^L} = \mu_2\mathbf{m^L} \mathbf{m}.
\end{equation}

We now use equation \eqref{eq : relation morphisme de Ai module} to compute

\begin{align*}
    \partial v &=  \sum_{n \geq 1} (-1)^{n+1} \partial \phi_{n+1} \mL \m^{n-1} \\
    &\overset{\eqref{eq : preuve Theorem C}}{=}\sum_{n\geq 1} (-1)^{n+1} \phi_n (\Phi_{1,*} \otimes 1^{\otimes n-1}) \mL \m^{n-1} - \sum_n  \Phi_{2,*} (\phi_n \otimes 1) \mL \m^{n-1}\\
    & + \sum_{n\geq 1} (-1)^{n+1}(-1)^n \phi_{n+1} (\partial_{F_1} \otimes 1^{\otimes n}) \mL \m^{n-1} - \sum_{n \geq 1} \sum_{r\geq 1} \phi_{n+1}(1^{\otimes r} \otimes \mu_1 \otimes 1^{\otimes n-r}) \mL \m^{n-1}\\
    & - \sum_{n\geq 2} \sum_{r \geq 1} (-1)^{r+n} \phi_n(1^{\otimes r} \otimes \mu_2 \otimes 1^{\otimes n-r-1}) \mL \m^{n-1}\\
    &\overset{(*)}{=} \varphi_*\Psi_1 - \sum_{k \geq 1} (-1)^{k+1} \phi_{k+1} \Phi_{1,*} \mL \m^{k} - \Psi_2 \theta^*\tilde{\varphi}\\
    &- \sum_{n \geq 1} (-1)^{n+1} \phi_{n+1} \mL \m^{n-1}\partial_{F_1} -  \sum_{n\geq 1} \sum_{r\geq 1} \phi_{n+1}(1^{\otimes r} \otimes \mu_1 \otimes 1^{\otimes n-r}) \mL \m^{n-1}\\
    & + \sum_{k \geq 1} \sum_{r\geq 1} (-1)^{r+k} \phi_{k+1}(1^{\otimes r} \otimes \mu_2 \otimes 1^{\otimes k-r}) \mL \m^{k}.
\end{align*}

To obtain the equality $(*)$ we performed the following operations:

\begin{enumerate}
    \item We split \begin{align*}
        \sum_{n\geq 1} (-1)^{n+1} \phi_n (\Phi_{1,*} \otimes 1^{\otimes n-1}) \mL \m^{n-1} &= \underbrace{\phi_1}_{= \varphi_*} \underbrace{\Phi_{1,*} \mL}_{\overset{\eqref{eq : preuve theorem C def iso fibration thm}}{=}{\Psi_1}} + \sum_{n\geq 2} (-1)^{n+1} \phi_n (\Phi_{1,*} \otimes 1^{\otimes n-1}) \mL \m^{n-1}\\
        &= \varphi_* \Psi_1 - \sum_{k \geq 1} (-1)^k \phi_{k+1}(\Phi_{1,*} \otimes 1^{\otimes k})\mL \m^{k}
    \end{align*}
    \item We noticed that $\Phi_{2,*} \mL (\phi_n \otimes 1) \m^{n-1} = \Phi_{2,*}(\phi_n \otimes 1) \mL \m^{n-1}$ for all $n \geq 1$ since $\phi_n$ will only take the first $n$ arguments. By definition, $$\sum_n\Phi_{2,*} \mL (\phi_n \otimes 1) \m^{n-1} = \Phi_{2,*} \mL\left( \sum_n (\phi_n \otimes 1) \m^{n-1}\right) = \Psi_2 \theta^*\tilde{\varphi}.$$
    \item We also noticed that $$\phi_{n+1} (\partial_{F_1} \otimes 1^{\otimes n}) \mL \m^{n-1} = (-1)^n \phi_{n+1}\mL \m^{n-1} \partial_{F_1}$$ for all $n \geq 1.$
\end{enumerate}

Using equation \eqref{eq : MC Ai module} and \eqref{eq : sys repr cell Ai}, we get a similar equation as in Lemma \ref{lemme : elimination using MC}: $$ (1^{\otimes r} \otimes \partial \otimes 1^{\otimes n-r}) \mL \m^n = (-1)^{r+n} (1^{\otimes r} \otimes \mu_2 \otimes 1^{\otimes n-r}) \mL \mathbf{\Tilde{m}}^{n+1}.$$ We conclude the computation of $\partial v$ with the formula

$$\partial v = \varphi_*\Psi_1 - \sum_{k \geq 1} (-1)^{k+1} \phi_{k+1} \Phi_{1,*} \mL \m^{k} - \Psi_2 \theta^*\tilde{\varphi} - \sum_{n \geq 1} (-1)^{n+1} \phi_{n+1} \mL \m^{n-1}\partial_{F_1}.$$

Moreover
\begin{align*}
    v \partial &= v(\partial_{F_1} + \Phi_{1,*} \m) \\
    & = \sum_{n \geq 1} (-1)^{n+1} \phi_{n+1} \mL \m^{n-1} \partial_{F_1} + \sum_{n \geq 1} (-1)^{n+1} \phi_{n+1} \mL \m^{n-1} \Phi_{1,*}\m\\
    &= \sum_{n \geq 1} (-1)^{n+1} \phi_{n+1} \mL \m^{n-1} \partial_{F_1} + \sum_{n \geq 1} (-1)^{n+1} \phi_{n+1}\Phi_{1,*} \mL \m^{n}.
\end{align*}

We therefore get the relation 

$$\partial v + v \partial = \varphi\Psi_1 - \Psi_2 \theta^*\Tilde{\varphi}.$$

\end{myproof}

We now state some corollaries of this theorem.

\begin{cor}\label{cor : homotopy euqivalence if different lifting functions}
    Let $\Phi_1 : E \ftimes{\pi}{\ev_0} \mathcal{P} X \to E$ and $\Phi_2 : E \ftimes{\pi}{\ev_0} \mathcal{P} X \to E$ be two transitive lifting functions associated to a fibration $F \hookrightarrow E \overset{\pi}{\to} X$. Let $\Xi$ be a set of DG Morse data on $X$. For $i \in \{1,2\}$, denote $\F_i = C_*(F)$ endowed with the $C_*(\Omega X)$-module structure induced by $\Phi_i$ and $\Psi_i : C_*(X,\Xi,\F_i) \to C_*(E)$ the quasi-isomorphism given by the Fibration Theorem. Then $\Id : E \to E$ induces a quasi-isomorphism $\widetilde{\Id} : C_*(X,\Xi,\F_1) \to C_*(X,\Xi,\F_2)$ such that the following diagram commutes up to homotopy,

    $$\xymatrix{
    C_*(X,\Xi,\F_1) \ar[rr]^{\widetilde{\Id}} \ar[dr]^{\Psi_1} & & C_*(X,\Xi,\F_2) \ar[dl]_{\Psi_2} \\
     & C_*(E). & 
    }$$

\end{cor}
\begin{flushright}
    $\blacksquare$
\end{flushright}
This result would be no much smaller task to prove on itself than Theorem \ref{thm : DG Thom iso} since one would need to complete a homotopy between $\Phi_1$ and $\Phi_2$ into a whole coherent homotopy for $\Id$ and then prove that $\widetilde{\Id}$ corresponds to $\Id : C_*(E) \to C_*(E).$

\begin{cor}\label{cor : induced map morphism of fibration well-defined in homology}
    If $\varphi : E_1 \to E_2$ is a morphism of fibrations, then the map $\tilde{\varphi} : H_*(X,C_*(F_1)) \to H_*(X,C_*(F_2))$ is well-defined.
\end{cor}

\begin{proof}
    Let $\Xi_0$ and $\Xi_1$ be sets of DG Morse data on $X$. Proving this corollary amounts to prove that the following diagram commutes 

    $$\xymatrix{
    H_*(X,\Xi_0,C_*(F_1)) \ar[r]^{\tilde{\varphi}} \ar[d]^{\Psi_{01}} & H_*(X,\Xi_0,C_*(F_2)) \ar[d]^{\Psi_{01}} \\
    H_*(X,\Xi_1,C_*(F_1)) \ar[r]^{\tilde{\varphi}} &  H_*(X,\Xi_1,C_*(F_2)).
    }$$

    We decompose this diagram in 

    $$\xymatrix{
    H_*(X,\Xi_0,C_*(F_1)) \ar[r]^{\tilde{\varphi}} \ar[d]^{\Psi_{E_1}} \ar@/_3pc/[dd]_{\Psi_{01}} & H_*(X,\Xi_0,C_*(F_2)) \ar@/^3pc/[dd]^{\Psi_{01}} \ar[d]^{\Psi_{E_2}} \\
    H_*(E_1) \ar[r]^{\varphi_*} \ar[d]^{\Psi_{E_1}^{-1}} & H_*(E_2) \ar[d]^{\Psi_{E_2}^{-1}} \\
    H_*(X,\Xi_1,C_*(F_1)) \ar[r]^{\tilde{\varphi}} &  H_*(X,\Xi_1,C_*(F_2)).
    }$$

    Theorem \ref{thm : morphisme induit commute avec iso} gives that both squares commute. It remains to prove that $\Psi_{E_i} \circ \Psi_{01} = \Psi_{E_i}$ in homology for $i \in \{0,1\}$. \cite[Proposition 10.2.1]{BDHO23} states that $\Psi_{01} = \Id_* : H_*(X,\Xi_0,C_*(F_1)) \to H_*(X,\Xi_1,C_*(F_1))$ and \cite[Proposition 9.8.1]{BDHO23} states that $ \Psi_{E_i}\Id_* = \Id_{E_i}\Psi_{E_i}$. This concludes the proof of this Corollary.
\end{proof}

\begin{cor}\label{cor : compatibility induced map by a morphism of fibration and direct maps}
    Let $F_1 \hookrightarrow E_1 \to X$ and $F_2 \hookrightarrow E_2 \to X$ be two fibrations and $\varphi : E_1 \to E_2$ a morphism of fibrations.
    Let $g: Y^m \to X^n$ be a continuous map. Denote $\F_1 = C_*(F_1)$ and $\F_2 = C_*(F_2)$.
    Then the following diagram commutes

    $$
    \xymatrix{
    H_*(Y, g^*\F_1) \ar[r]^-{g^*\Tilde{\varphi}} \ar[d]_{g_*} & H_*(Y,g^*\F_2) \ar[d]^{g_*}\\
    H_*(X,\F_1) \ar[r]_{\Tilde{\varphi}} & H_*(X, \F_2).
    }
    $$

\begin{rem}
    Using a more algebraic approach for any morphism of complexes $\tilde{\varphi}$ induced by a morphism of $\Ai$-modules over  $C_*(\Omega X)$, we will reprove this property as well as prove that the map $\tilde{\varphi}$ is compatible with shriek maps (Proposition \ref{prop : morphisme Ai commute avec direct et shriek}).
\end{rem}

\end{cor}

\begin{proof}
    We will prove that this diagram is a face of cube where every other faces are commutative.

    $$
    \xymatrix{
    H_*(Y,g^*\F_1) \ar[ddd]^{g_*} \ar[rr]^{g^*\tilde{\varphi}} \ar@<.5ex>[dr]^{\Psi_{g^*E_1}} & & H_*(Y, g^*\F_2) \ar@<.5ex>[dr]^{\Psi_{g^*E_2}} \ar@{-->}[ddd]^{g_*} &\\
     & H_*(g^*E_1) \ar@<.5ex>[ul]^{\Psi_{g^*E_1}^{-1}} \ar[rr]^{g^*\varphi_* \qquad} \ar[ddd]^{g_*} & & H_*(g^*E_2) \ar@<.5ex>[ul]^{\Psi_{g^*E_2}^{-1}} \ar[ddd]^{g_*} \\
     & & & \\
     H_*(X, \F_1) \ar@{-->}[rr]_{\qquad \tilde{\varphi}} \ar@<.5ex>[dr]^{\Psi_{E_1}} & & H_*(X,\F_2) \ar@{-->}@<.5ex>[dr]^{\Psi_{E_2}} & \\
     & H_*(E_1) \ar@<.5ex>[ul]^{\Psi_{E_1}^{-1}}\ar[rr]^{\varphi_*}  & & H_*(E_2)\ar@{-->}@<.5ex>[ul]^{\Psi_{E_2}^{-1}}.
    }
    $$

    The front face is commutative by general properties about pullback fibrations. Theorem \ref{thm : morphisme induit commute avec iso} states that the faces on the top and on the bottom are commutative.  \cite[Proposition 9.8.1]{BDHO23} states that the lateral faces are commutative. From there, a diagram chase concludes the proof.
\end{proof}

\begin{cor}\label{cor : equality of Ai morphism on total space gives equality of induced maps}
    If $\varphi : E_1\to E_2$ is a morphism of fibrations, then  $\tilde{\varphi} : H_*(X,C_*(F_1)) \to H_*(X,C_*(F_2))$ does not depend on the coherent homotopy but only on $\varphi_* : H_*(E_1) \to H_*(E_2) $. Moreover, if $\varphi_* : H_*(E_1) \to H_*(E_2)$ is an isomorphism, then so is $\tilde{\varphi} : H_*(X,C_*(F_1)) \to H_*(X,C_*(F_2))$.
\end{cor}
\begin{flushright}
    $\blacksquare$
\end{flushright}

\begin{cor}\label{cor : composition of morphism of fibrations}
    If $\varphi : E_0 \to E_1$ and $\psi : E_1 \to E_2$ are morphisms of fibrations, then $\psi \circ \varphi : E_0 \to E_2$ is a morphism of fibrations and $$\widetilde{\psi \circ \varphi} = \tilde{\psi} \circ \tilde{\varphi} : H_*(E_0) \to H_*(E_2).$$
\end{cor}

\begin{proof}
    Theorem \ref{thm : morphisme induit commute avec iso} proves that at the homology level, \begin{align*}
        \widetilde{\psi \circ \varphi} &= \Psi_{E_0} \circ (\psi \circ \varphi)_* \circ \Psi_{E_2}^{-1}\\
        &= \Psi_{E_0} \circ \psi_* \circ\varphi_* \circ \Psi_{E_2}^{-1}\\
        &= \Psi_{E_0} \circ \psi_* \circ \Psi_{E_1}^{-1} \circ \Psi_{E_1} \circ \varphi_* \circ \Psi_{E_2}^{-1}\\
        &= \tilde{\psi} \circ \tilde{\varphi}.
    \end{align*}
\end{proof}

\subsection{Compatibility with the \texorpdfstring{$\Ai$}{Ai}-Morse toolset, direct and shriek maps}

The main results of this section are that any morphism $\boldsymbol{\varphi} : \F \to \G$ of $\Ai$-modules over $C_*(\Omega X)$ induces a well-defined map $\tilde{\varphi} : H_*(X,\F) \to H_*(X,\G)$ in homology and that it is compatible with direct and shriek maps. Therefore, the reader only interested in the case of fibrations will only find here the proof of the compatibility of $\tilde{\varphi}$ with shriek maps. We have provided independent proofs of the compatibility with direct maps (Corollary \ref{cor : compatibility induced map by a morphism of fibration and direct maps}) and the fact that $\tilde{\varphi} : H_*(X,\F) \to H_*(X,\G)$ is well-defined (Corollary \ref{cor : induced map morphism of fibration well-defined in homology}) using Theorem \ref{thm : morphisme induit commute avec iso} in the case of a fibration.

\subsubsection{Compatibility with \texorpdfstring{$\Ai$}{Ai}-Morse toolset}

We prove here that the map induced by an $\Ai$-morphism of modules commutes in homology with the maps defined in Proposition \ref{Prop : continuation morphisms}. The statement and proof are quite technical since the proposition is aimed to be applied in various situations.

\begin{prop}\label{prop : compatibility induced map and tau maps}
   Let $f_0, f_1 : X \to \R$ be two Morse functions and $\F, \G$ two $\Ai$-modules over $C_*(\Omega X)$.
   
   Let $\{m^0_{x,y} \in C_{|x|-|y|-1}(\Omega X), \ x,y \in \Crit(f_0)\}$ and $\{m^1_{x',y'} \in C_{|x'|-|y'|}(\Omega X), x',y' \in \Crit(f_1)\}$ be two twisting cocycles on $X$. 
   Let $\{\tau_{x,y'} \in C_{|x|-|y'|}(\Omega X), x\in \textup{Crit}(f_0), y' \in \textup{Crit}(f_1)\}$ a cocycle satisfying \eqref{eq : mu_1 de tau}.
   
   Let $\Psi_{\F} : C_*(X,m^0,\F) \to C_*(X, m^1,\F)$, $\Psi_{\G} : C_*(X,m^0,\G) \to C_*(X,m^1,\G)$ be the morphisms defined in \ref{Prop : continuation morphisms} using $\{\tau_{x,y'}\in C_{|x|-|y'|}(\Omega X), \ x \in \Crit(f_0), \ y' \in \Crit(f_1)\}$.\\
   
   Let $\boldsymbol{\varphi} : (\F, \boldsymbol{\nu^F}) \to (\G, \boldsymbol{\nu^G})$ be a morphism of $\Ai$-modules and $\tilde{\varphi}_0 : C_*(X,m^0,\F) \to  C_*(X, m^0, \G)$, \\ $\tilde{\varphi}_1 : C_*(X,m^1,\F) \to  C_*(X, m^1, \G)$ be the induced morphisms of complexes as defined in Proposition \ref{prop : Morphisme de Ai module en morphisme de complexe}.

    Then, the following diagram commutes up to chain homotopy :

    $$\xymatrix{
    C_*(X,m^0,\F) \ar[d]_{\Psi_{\F}} \ar[r]^-{\Tilde{\varphi}_0} & C_*(X, m^0, \G) \ar[d]^{\Psi_{\G}} \\
    C_*(X,m^1,\F) \ar[r]_{\Tilde{\varphi}_1} & C_*(X,m^1,\G).
    }$$
\end{prop}

\begin{proof}
    Let $\kappa : C_*(X, m^0, \F) \to C_{*+1}(X,m^1, \G)$ defined by 

    $$\kappa = \sum_{n \geq 1} \sum_{u=1}^n (-1)^{u-1} (\varphi_{n+1} \otimes 1) \m_{(1)}^{n-u} \boldsymbol{\Tilde{\tau}} \m_{(0)}^{u-1}.$$

    We now prove that $\Psi_{\G}\Tilde{\varphi}_{0} - \Tilde{\varphi}_{1}\Psi_{\F} = \partial \kappa + \kappa \partial$.

    We compute 

    \begin{align*}
        \partial \kappa &= \sum_k \sum_{n \geq 1} \sum_{u=1}^n (-1)^{u-1} (\nu'_{k+1} \otimes 1) \m^k_{(1)} (\varphi_{n+1} \otimes 1) \m_{(1)}^{n-u} \boldsymbol{\Tilde{\tau}} \m_{(0)}^{u-1} \\
        &= \sum_k \sum_{n \geq 1} \sum_{u=1}^n (-1)^{u-1} (-1)^{kn} (\nu'_{k+1}\varphi_{n+1} \otimes 1) \m^{n+k-u}_{(1)} \boldsymbol{\Tilde{\tau}} \m_{(0)}^{u-1} \\
        &= \sum_k \sum_{n \geq 1} \sum_{u=1}^{n+k} (-1)^{u-1} (-1)^{kn} (\nu'_{k+1}\varphi_{n+1} \otimes 1) \m^{n+k-u}_{(1)} \boldsymbol{\Tilde{\tau}} \m_{(0)}^{u-1}\\
        &- \sum_k \sum_{n \geq 1} \sum_{u=n+1}^{n+k} (-1)^{u-1} (-1)^{kn} (\nu'_{k+1}\varphi_{n+1} \otimes 1) \m^{n+k-u}_{(1)} \boldsymbol{\Tilde{\tau}} \m_{(0)}^{u-1}
    \end{align*}

    and 

    \begin{align*}
        \kappa \partial &= \sum_k \sum_{n \geq 1} \sum_{u=1}^n (-1)^{u-1} (\varphi_{n+1} \otimes 1) \m_{(1)}^{n-u} \boldsymbol{\Tilde{\tau}} \m_{(0)}^{u-1} (\nu_{k+1} \otimes 1) \m_{(0)}^{k+1}\\
        &= \sum_k \sum_{n \geq 1} \sum_{u=1}^n (-1)^{u-1} (-1)^{(k-1)(n-1)} (\varphi_{n+1} \nu_{k+1} \otimes 1) \m_{(1)}^{n-u} \boldsymbol{\Tilde{\tau}} \m^{k+u-1}_{(0)} \\
        &= - \sum_k \sum_{n \geq 1} \sum_{u=k+1}^{n+k} (-1)^{u-1} (-1)^{(k-1)n} (\varphi_{n+1} \nu_{k+1} \otimes 1) \m^{n+k-u}_{(1)} \boldsymbol{\Tilde{\tau}} \m_{(0)}^{u-1} \\
        &= - \sum_k \sum_{n \geq 1} \sum_{u=1}^{n+k} (-1)^{u-1} (-1)^{(k-1)n} (\varphi_{n+1} \nu_{k+1} \otimes 1) \m^{n+k-u}_{(1)} \boldsymbol{\Tilde{\tau}} \m_{(0)}^{u-1}\\
        &+ \sum_k \sum_{n \geq 1} \sum_{u=1}^{k} (-1)^{u-1} (-1)^{(k-1)n} (\varphi_{n+1} \nu_{k+1} \otimes 1) \m^{n+k-u}_{(1)} \boldsymbol{\Tilde{\tau}} \m_{(0)}^{u-1}.
    \end{align*}

    Therefore,
   \begin{align*}
     \partial \kappa + \kappa \partial & = \sum_k \sum_{n \geq 1} \sum_{u=1}^{k} (-1)^{u-1} (-1)^{(k-1)n} (\varphi_{n+1} \nu_{k+1} \otimes 1) \m^{n+k-u}_{(1)} \boldsymbol{\Tilde{\tau}} \m_{(0)}^{u-1} \\
     & - \sum_k \sum_{n \geq 1} \sum_{u=n+1}^{n+k} (-1)^{u-1} (-1)^{kn} (\nu'_{k+1}\varphi_{n+1} \otimes 1) \m^{n+k-u}_{(1)} \boldsymbol{\Tilde{\tau}} \m_{(0)}^{u-1}\\
    &- \sum_{N \geq 1} \left[ \sum_{r=1}^{N}  (-1)^N \varphi_{N+1}(1^{\otimes r} \otimes \mu_1 \otimes 1^{\otimes N-r}) \right.\\
    & \left. +  \sum_{r=1}^{N-1} (-1)^r \varphi_{N}(1^{\otimes r} \otimes \mu_2 \otimes 1^{\otimes N-1-r}) \right] \left[ \sum_{u=1}^N (-1)^{u-1} \m^{n+k-u}_{(1)} \boldsymbol{\Tilde{\tau}} \m_{(0)}^{u-1} \right].\\
    &= \Tilde{\varphi}_D\Psi_{\F} - \Psi_{\G} \Tilde{\varphi}_C \\
    &- \sum_{N \geq 1} \left[ \sum_{r=1}^{N}  (-1)^N \varphi_{N+1}(1^{\otimes r} \otimes \mu_1 \otimes 1^{\otimes N-r}) \right.\\
    & \left. +  \sum_{r=1}^{N-1} (-1)^r \varphi_{N}(1^{\otimes r} \otimes \mu_2 \otimes 1^{\otimes N-1-r}) \right] \left[ \sum_{u=1}^N (-1)^{u-1} \m^{n+k-u}_{(1)} \boldsymbol{\Tilde{\tau}} \m_{(0)}^{u-1} \right].\\
    \end{align*}

    The same arguments as in the proof of Proposition \ref{Prop : continuation morphisms} concerning the relation between $(1^{\otimes r} \otimes \mu_1 \otimes 1^{\otimes N-r})\m^{n+k-u}_{(1)} \boldsymbol{\Tilde{\tau}} \m_{(0)}^{u-1}$ and $(1^{\otimes r} \otimes \mu_2 \otimes 1^{\otimes N-1-r})\m^{n+k-u}_{(1)} \boldsymbol{\Tilde{\tau}} \m_{(0)}^{u-1}$ depending on $u$ and $r$ show that the last sum is 0 and therefore conclude the proof.

\end{proof}

\begin{cor}\label{cor : induced morphism well-defined in homology}
    For any data $\Xi_0$ and $\Xi_1$ over $X$ and any morphism $\boldsymbol{\varphi} : (\G,\nu_n) \to (\G',\nu'_n)$ of $\Ai$-module over $C_*(\Omega X)$, the following diagram commutes up to chain homotopy:

    \[
    \xymatrix{
    C_*(X,\Xi_0, \G) \ar[d]_{\Psi_{01}} \ar[r]^-{\Tilde{\varphi}_0} & C_*(X,\Xi_0,\G') \ar[d]^{\Psi_{01}} \\
    C_*(X,\Xi_1,\G) \ar[r]_{\Tilde{\varphi}_1} & C_*(X,\Xi_1,\G').
    }
    \]

    In particular, the map $\Tilde{\varphi} : H_*(X,\G) \to H_*(X,\G')$ is then well-defined.
\end{cor}

\begin{proof}
   This is a consequence of Proposition \ref{prop : compatibility induced map and tau maps}, with $m^0$ the Barraud-Cornea twisting cocycle arising from $\Xi_0$, $m^1$ the Barraud-Cornea twisting cocycle arising from $\Xi_1$ and $\Psi_{01} : C_*(X, \Xi_0, \G) \to C_*(X, \Xi_1, \G), \ \Psi_{01} : C_*(X, \Xi_0, \G') \to C_*(X, \Xi_1, \G')$ the corresponding $\Ai$-continuation maps (see Theorem \ref{thm : Continuation morphism Ai invariance}).
\end{proof}

\begin{prop}\label{prop : induced morphism is limit of spectral sequence maps}
    Let $\boldsymbol{\varphi} : \F \to \F'$ be a morphism of $\Ai$-modules over $C_*(\Omega X)$ and denote $\varphi = \varphi_1$. Then, the induced map in homology $\Tilde{\varphi} : H_*(X,\F) \to H_*(X,\F')$ is a limit of morphisms $\varphi^{(r)} : E^r_{p,q} \to E'^r_{p,q}$ between the spectral sequences associated with those homologies and 

    $$\varphi^{(1)} = \varphi_{*} \otimes 1 : H_q(\F) \otimes C_p(f,\xi) \to H_q(\F') \otimes C_p(f,\xi).$$
\end{prop}

\begin{proof}
    Recall that the enriched complex $C_*(X,\mathcal{A})$ is filtered by

    $$F_p(C_k(X,\mathcal{A})) = \bigoplus_{\substack{i+j = k \\ j \leq p}} \mathcal{A}_i \otimes C_j(f,\xi)$$

    and that the $0$-th page is given by $E^0_{p,q} = \mathcal{A}_q \otimes C_p(f,\xi)$. 
    
    The morphism $\tilde{\varphi}$ respects the filtration since $\m$ strictly decreases the index of the critical point.    
    If $\alpha \otimes x \in \mathcal{F}_q \otimes C_p(f,\xi)$, the only term in $\Tilde{\varphi}(\alpha \otimes x)$ that belongs to $\mathcal{F}'_q \otimes C_p(f,\xi)$ is $\varphi_{*}(\alpha) \otimes x$. Therefore  

    $$\varphi^{(0)}(\alpha \otimes x) = \varphi_{*}(\alpha) \otimes x$$

    and this induces in homology $\varphi^{(1)}= \varphi_{*} \otimes 1 : H_q(\F) \otimes C_p(f,\xi) \to H_q(\F') \otimes C_p(f,\xi).$
\end{proof}

\begin{cor}\label{cor : iso sur fibre induit tilde iso}
    If $\boldsymbol{\varphi} : \F \to \F'$ is a quasi-isomorphism of $\Ai$-modules, then $$\tilde{\varphi} : H_*(X, \F) \to H_*(X,\F')$$ is an isomorphism.
\end{cor}
\begin{flushright}
    $\blacksquare$
\end{flushright}

\subsubsection{Compatibility with direct and shriek maps}\label{subsection : Compatibility with direct and shriek maps}
Let $X^n$ and $Y^m$ be two pointed, oriented, closed and connected manifolds and $(\F, \boldsymbol{\nu})$ be an $\Ai$-module over $C_*(\Omega Y)$. Let $f : X \to \R$ and $g : Y \to \R$ be Morse functions that are part of sets of DG Morse data $\Xi_X$ and $\Xi_Y$ on respectively $X$ and $Y$. Denote $\{m^X_{x,z}\}$ and $\{m^Y_{y',w'}\}$ the Barraud-Cornea cocycles associated to $\Xi_X$ and $\Xi_Y$.\\

Let $\varphi : X \to Y$ be a continuous map.

The first definition of direct map $\varphi_* : H_*(X,\varphi^*\F) \to H_*(Y,\F)$ and shriek map $\varphi_! : H_*(Y,\F) \to H_{*+n-m}(X,\varphi^*\F)$
are exactly the same in the $\Ai$ setting as in the DG case (see \cite[Section 9]{BDHO23}).\\

Let $\varphi : X \to Y$ be a smooth map.

The second definition of the direct map $\varphi_* : C_*(X, \Xi_X, \varphi^*\F) \to C_*(Y, \Xi_Y, \F) $ is described in \cite[Section 10]{BDHO23} by constructing a cocycle $\{\tau_{x,y'}\in C_{|x|-|y'|}(\Omega Y), \ x \in \Crit(f), \ y' \in \Crit(g)\}$ satisfying the equation \eqref{eq : mu_1 de tau} for the twisting cocycles $\{m^Y_{y',w'}\in C_{|y'|-|w'|-1}(\Omega Y)\}$ and $\{\varphi_*(m^X_{x,z}) \in C_{|x|-|z|-1}(\Omega Y)\}$ on $Y$.

The second definition of the shriek map $\varphi_! : C_*(Y,\Xi_Y, \F) \to C_{*-m+n}(X, \Xi_X, \varphi^*\F) $ is described by constructing a cocycle $$\{\tau_{y',x}\in C_{[y']-[x]}(\Omega Y), \ y' \in \Crit(g), \ x \in \Crit(f) \}$$ satisfying the equation \eqref{eq : mu_1 de tau} for the twisting cocycles $\{m^Y_{y',w'} \in C_{[y']-[w']-1}(\Omega Y)\}$ and $\{\varphi_*(m^X_{x,z}) \in C_{[x]-[z]-1}(\Omega Y)\}$ on $Y$ where we used the following grading for each $x \in Crit(f)$, $y' \in \Crit(g).$

$$[x] = |x|+m \textup{ and } [y'] = |y'|+n.$$

In our case, using Proposition \ref{Prop : continuation morphisms}, the direct map $\varphi_* : C_*(X, \varphi^*\F) \to C_*(Y,\F)$ is defined by 

$$\varphi_* = \sum_{n \geq 1} \sum_{u=1}^n (-1)^{u-1} (\nu_{n+1} \otimes 1) \m^{n-u}_Y \boldsymbol{\Tilde{\tau}} (\varphi_*\m_X)^{u-1}$$

and the shriek map $\varphi_! : C_*(Y,\F) \to C_{*-m+n}(X,\varphi^*\F)$ is defined by 

$$\varphi_! = \sum_{n \geq 1} \sum_{u=1}^n (-1)^{u-1} (\nu_{n+1} \otimes 1) (\varphi_*\m_X)^{n-u} \boldsymbol{\Tilde{\tau}^!} \m_Y^{u-1}.$$

If $\varphi : X \to Y$ is continuous, let $\varphi' :  X \to Y$ be a smooth map which is homotopic to $\varphi$ through a basepoints preserving homotopy $H$ which is $C^0$-close to the constant homotopy. The identification morphism $$\Psi^H : C_*(X,\Xi_X, \varphi^*\F) \to C_*(X, \Xi_X, \varphi'^*\F)$$ defined in \cite[Proposition 8.2.1]{BDHO23} carries over in our setting using again Proposition \ref{Prop : continuation morphisms} and we also define  $$\varphi_* = \varphi'_* \circ \Psi^H. $$

Using the $\Ai$ homotopy property \ref{Prop : homotopy Criterion Ai} instead of the DG one, the property stating in \cite{BDHO23} that the first and second definition of the direct and shriek maps are equivalent in homology still holds.

\begin{prop}\label{prop : morphisme Ai commute avec direct et shriek}
    Let $h: X^n \to Y^m$ be a continuous map, $\G$ and $\G'$ be two $\Ai$-modules over $C_*(\Omega Y)$, and $\{\varphi_{n}\} : \G \to \G'$ be a morphism of $\Ai$-modules.\\

    Then the following diagrams commute

{
\begin{minipage}{7cm}
    $$
    \xymatrix{
    H_*(X, h^*\G) \ar[r]^-{h^*\Tilde{\varphi}} \ar[d]_{h_*} & H_*(X,h^*\G') \ar[d]^{h_*}\\
    H_*(Y,\G) \ar[r]_{\Tilde{\varphi}} & H_*(Y, \G')
    }
    $$
\end{minipage}
\begin{minipage}{7cm}
   $$
    \xymatrix{
    H_{*-m+n}(X, h^*\G) \ar[r]^-{h^*\Tilde{\varphi}}  & H_{*-m+n}(X,h^*\G') \\
    H_*(Y,\G) \ar[u]^{h_!} \ar[r]_{\Tilde{\varphi}} & H_*(Y, \G') \ar[u]_{h_!}
    }
    $$
\end{minipage}
} 
   
\end{prop}

\begin{proof}
   It is a direct consequence of Proposition \ref{prop : compatibility induced map and tau maps}.
\end{proof}

\section{DG Künneth formula. Proof of Theorem C}\label{section : DG Künneth formula}

Let $X,Y$ be two pointed, oriented, closed and connected manifolds endowed with DG Morse data $\Xi_X$ and $\Xi_Y$. The goal of this section is to define a cross-product $C_*(X, \Xi_X, \F) \otimes C_*(Y,\Xi_Y, \G) \to C_*(X \times Y, \Xi_{X \times Y}, \mathcal{H})$ where the DG Morse data $\Xi_{X \times Y}$ is defined below and the coefficients $\mathcal{H}$ will be discussed in Section \ref{subsection : Künneth map}.

We will construct a representing chain system adapted to the cartesian product and evaluate it by a family of evaluation maps $q^K_{(x,y),(x',y')} : \trajb{(x,y),(x',y')} \to \Omega X \times \Omega Y$ also adapted to the Cartesian product to obtain a twisting cocycle $$m^K_{(x,y),(x',y')} = q^K_{(x,y),(x',y'),*}(s_{(x,y),(x',y')}) \in C_{|x|+|y|-|x'|-|y'|-1}(\Omega X \times \Omega Y)).$$ This twisting cocycle is not the Barraud-Cornea cocycle associated with the set of DG Morse data $\Xi_{X\times Y}$ but is well-adapted to chain-level computations and enables to define a chain-level cross product for enriched Morse complexes. We will prove that the enriched Morse complex constructed with the Barraud-Cornea cocycle is chain homotopy equivalent to the one constructed with $m^K$.

\subsection{Morse data on a Cartesian product}\label{subsection : Morse data on a Cartesian product}

Given $\Xi_X= (f,\xi_X, s_{x,x'}, o_X, \mathcal{Y}_X, \theta_X)$ a set of Morse data on $(X, \star_X)$ and $\Xi_Y = (g,\xi_Y, s_{y,y'}, o_Y, \mathcal{Y}_Y, \theta_Y)$ a set of Morse data on $(Y, \star_Y)$, we construct a set of Morse data $\Xi_{X \times Y} = (H, \xi, s_{(x,y),(x',y')}, o, \mathcal{Y}, \theta)$ on $(X \times Y, (\star_X, \star_Y))$ that consists in the following :

\begin{enumerate}
    \item $H(x,y) = f(x) + g(y)$. Note that $H$ is a Morse function on $X \times Y$ which satisfies $|(x,y)| = |x| + |y| $.\\
    \item $\xi(x,y) = (\xi_X(x), \xi_Y(y))$ is a pseudo-gradient associated to $H$.\\
    \item There is a canonical identification $\overline{W^u_H}(x,y) \simeq \overline{W^u_f}(x) \times \overline{W^u_g}(y)$. We therefore use the orientation $\Or \ \wbi{u}{H}{x,y} = \left( \Or \ \wbi{u}{f}{x} ,  \Or \ \wbi{u}{g}{y} \right) $.\\
    \item Choose $\mathcal{Y}$ the tree in $X \times Y$ rooted at $(\star_X,\star_Y)$ whose vertices are $\Crit(H) = \Crit(f) \times \Crit(g)$ such that its projection over $X$ is the tree $\mathcal{Y}_X$ and its projection over $Y$ is the tree $\mathcal{Y}_Y$.\\
    \item The homotopy inverse $\theta = (\theta_X, \theta_Y) : (X \times Y)/\mathcal{Y} \to X \times Y  $ of the canonical projection $p : X \times Y \to (X \times Y)/\mathcal{Y}$.
\end{enumerate}

The goal of this section is to define a suitable representing chain system $s_{(x,y),(x',y')}$ of $\trajbi{H}{(x,y),(x',y')}.$ First, let us fix our orientation conventions for the spaces of trajectories.

\subsubsection{Orientation conventions}

We use the same orientation conventions and rules as \cite{BDHO23}.

Let $(f,\xi)$ be a Morse-Smale pair and $x \in \Crit(f)$. The orientation $o_x$ of the unstable manifolds $W^u(x)$ is fixed by the set of DG Morse data. This induces a co-orientation of the stable manifolds $$\Coor \ W^s(x) = \Or \  W^u(x).$$ 

We orient a tranverse intersection $Z \pitchfork W$ between an oriented manifold $Z$ and a co-oriented manifold $W$ by 

$$\left( \Or \ Z \pitchfork W, \ \Coor \ W \right) = \Or \ Z.$$

Therefore, for $x,y \in \Crit(f)$, the manifold $\mathcal{M}(x,y) = W^u(x) \pitchfork W^s(y)$ is oriented via the rule 

\begin{equation}\label{eq : orientation param space of traj}
\left( \Or \ \mathcal{M}(x,y) , \ \Or \ W^u(y) \right)  = \Or \ W^u(x).
\end{equation}

Define $$S^s(y) := W^s(y) \pitchfork f^{-1}(f(y) + \epsilon)$$ for small $\epsilon>0$ the stable sphere associated to a critical point $y \in \Crit(f)$. We co-orient $S^s(y) \subset W^s(y)$ by 

$$\Coor \ S^s(y) = (-\xi, \Coor \ W^s(y)) = \left(-\xi, \Or \ W^u(y)\right).$$

There is an identification $$\traj{x,y} = W^u(x)
\pitchfork S^s(y),$$ which yields the orientation 

$$\left(\Or \ \traj{x,y}, -\xi, \Or \ W^u(y)\right) = \Or \ W^u(x). $$

We use the same orientation rules for the compactifications $\trajb{x,y}$ and $\wb{u}{x}$. 

\begin{equation}
    \left(\Or \ \trajb{x,y}, -\xi, \Or \ \wb{u}{y}\right) = \Or \ \wb{u}{x}.
\end{equation}

\subsubsection{Representing chain system on a Cartesian product}

\begin{lemme}
    If $x= x'$, $\trajbi{H}{(x,y),(x,y')} = \{x\} \times \trajbi{g}{y,y'}$ and the orientations differ by the sign $(-1)^{|x|(|y|-|y'|)}$.\\

    If $y=y'$, $\trajbi{H}{(x,y),(x',y)} = \trajbi{f}{x,x'} \times \{y\}$ and the orientations coincide.
\end{lemme}

\begin{proof}
The two identifications are clear. It remains to compare orientations.

The rule \eqref{eq : orientation param space of traj} implies for $x = y$ that a constant trajectory is oriented positively.

For any two critical points $(x,y),(x,y')$ of $H$,
$$\left( \Or \ \trajbi{H}{(x,y),(x,y')}, -\xi, \Or \ \wbi{u}{H}{x,y'} \right) = \Or \ \wbi{u}{H}{x,y}.$$

If $x=x'$, the equalities
\begin{align*}
        \left( \Or \ \{x\} , \Or \ \trajbi{g}{y,y'}, - \xi, \Or \ \wbi{u}{H}{x,y'}\right) &= (-1)^{|x|(|y|-|y'|)} \left( \Or \ \wbi{u}{f}{x}, \Or \ \trajbi{g}{y,y'}, - \xi, \Or \ \wbi{u}{g}{y'}\right)\\
        &= (-1)^{|x|(|y|-|y'|)} \Or \ \wbi{u}{H}{x,y}.
\end{align*}

show the orientation difference between $\trajbi{F}{(x,y),(x,y')}$ and $ \{x\} \times \trajbi{g}{y,y'}$.\\

If $y=y'$, the equalities 
\begin{align*}
    \left(\Or \ \trajbi{f}{x,x'}, \Or \ \{y\}, - \xi, \Or \ \wbi{u}{H}{x',y}\right) = \left( \Or \ \wbi{u}{f}{x}, \Or \ \wbi{u}{g}{y'}\right) = \Or \ \wbi{u}{H}{x,y}.
\end{align*}

show that $\trajbi{H}{(x,y),(x',y)} $ and $ \trajbi{f}{x,x'} \times \{y\}$ have the same orientation.
\end{proof}

Remark that 

\begin{align*}
\mathcal{M}_H((x,y),(x',y'))&:= W^u_H(x,y) \pitchfork W^s_H(x',y')\\
&= \left(W^u_f(x) \pitchfork W^s_f(y)\right) \times \left(W^u_f(x) \pitchfork W^s_f(y)\right) =: \mathcal{M}_f(x,x') \times \mathcal{M}_g(y,y'),
\end{align*}
and therefore there exist a projection $\pi$ and a section $i$,

$$\traji{H}{(x,y),(x',y')} = \faktor{\mathcal{M}_f(x,x') \times \mathcal{M}_g(y,y')}{\R_{diag}} \underset{\pi}{\overset{i}{\leftrightarrows}}  \mathcal{M}_f(x,x') / \R \times \mathcal{M}_g(y,y')/\R = \traji{f}{x,x'} \times \traji{g}{y,y'} $$

which can be written $\pi([a,b]_{X\times Y}) = ([a]_X , [b]_Y)$ and $i(\lambda_X , \lambda_Y) = [i_X(\lambda_X), i_Y(\lambda_Y)]_{X \times Y}$, where $i_X(\lambda_X) = \lambda_X \cap f^{-1}\left(\frac{f(x)+f(x')}{2}\right)$ and $i_Y(\lambda_Y) = \lambda_Y \cap g^{-1}\left(\frac{g(y)+g(y')}{2}\right).$

Remark that $\pi = \textup{Id}$ if $x=x'$ or $y=y'$.

\begin{lemme}\label{lemme : representing chain system of a Cartesian product}
    Let $\{s^X_{x,x'}\}$ and $\{s^Y_{y,y'}\}$ be representing chain systems for the moduli spaces of trajectories $\trajbi{f}{x,x'}$ and $\trajbi{g}{y,y'}$ respectively.
    There exists a representing chain system $\{ s_{(x,y),(x',y')} \}$ for the moduli spaces $\trajbi{H}{(x,y),(x',y')}$ such that 

    \begin{enumerate}
        \item $s_{(x,y),(x',y')} = (-1)^{|x|(|y|-|y'|)} (\{x\}, s^Y_{y,y'})$ if $x=x'$.\\
        \item $s_{(x,y),(x',y')} = (s_{x,x'}^X, \{y\})$ if $y=y'$\\
        \item $\pi_{*}(s_{(x,y),(x',y')}) = 0$ if $x \neq x'$ and $y\neq y'$.
    \end{enumerate}
\end{lemme}

\begin{proof}
    $\bullet$ If $x=x'$ or $y=y'$ we can just choose $s_{(x,y),(x,y')} = (-1)^{|x|(|y|-|y'|)} (\{x\}, s^Y_{y,y'})$, $s_{(x,y),(x',y)} = (s_{x,x'}^X, \{y\})$ and complete by induction as in Proposition 5.6 of \cite{BDHO23} in order to obtain $\{s_{(x,y),(x',y')}\}$ a representing chain system for $\trajbi{H}{(x,y),(x',y')}$ that satisfies conditions 1 and 2. We just have to check that $$\partial s_{(x,y),(x,y')} = \sum_{w \in \Crit(g)} (-1)^{|y|-|w|} s_{(x,y),(x,w)} \times s_{(x,w), (x,y')} $$ and $$\partial s_{(x,y),(x',y)} =  \sum_{z \in \Crit(f)} (-1)^{|x|-|z|} s_{(x,y),(z,y)} \times s_{(z,y), (x',y)}.$$

    We only establish the first equality since the second one is analogous,

    \begin{align*}
        \partial s_{(x,y),(x,y')} &= (-1)^{|x|(|y|-|y'|)} (\{x\}, \partial s_{y,y'})\\
        &= (-1)^{|x|(|y|-|y'|)} \sum_w (-1)^{|y|-|w|} (\{x\}, s_{y,w} \times s_{w,y'})\\
        &= \sum_{w} (-1)^{|y| - |w|} (-1)^{|x|(|y|-|w|)} (-1)^{|x|(|w|-|y'|)} (\{x\}, s_{y,w}) \times (\{x\}, s_{w,y'})\\
        &= \sum_{w} (-1)^{|y| - |w|} s_{(x,y),(x,w)} \times s_{(x,w),(x,y')}.
    \end{align*}

\noindent $\bullet$ If $x\neq x'$ and $y \neq y'$, we build $\{s_{(x,y),(x',y')}\}$ by induction on $|x|+|y|-|x'|-|y'| =l$. If $x \neq x'$, $y \neq y'$ and $l = 2$, then $|x|-|x'| = 1$ and $|y|-|y'| = 1$ or $\trajbi{H}{(x,y),(x',y')}$ is empty. Any representative $s_{(x,y),(x',y')} \in C_1(\trajbi{H}{(x,y),(x',y')})$ of the fundamental class satisfies condition 3. since $\pi_{*}(s_{(x,y),(x',y')}) \in C_1(\trajbi{f}{x,x'} \times \trajbi{g}{y,y'}) $ is a one chain in a $0$-dimensional space and is thus constant and degenerate.\\

Suppose a representing chain system $s_{(a,b),(a',b')}$ respecting the conditions 1., 2. and 3. has been constructed for every $|a|+|b|-|a'|-|b'| \leq l$. Let $|x|+|y|-|x'|-|y'| =l+1$ and $x \neq x'$, $y \neq y'$.
Proceed as in \cite[Proposition 5.2.6]{BDHO23} to build a chain $$s'_{(x,y),(x',y')} \in C_{l}(\trajbi{H}{(x,y),(x',y')})$$ that satisfies 

$$\partial s'_{(x,y),(x',y')} = \sum_{(z,w)}  (-1)^{|x|+|y|-|z|-|w|} s_{(x,y),(z,w)} \times s_{(z,w),(x',y')}.$$

Then $$\pi_*(s'_{(x,y),(x',y')})\in C_{l}(\trajbi{f}{x,x'} \times \trajbi{g}{y,y'})$$ is a cycle. Indeed,

\begin{align*}
    \partial \pi_*(s'_{(x,y),(x',y')}) &= \sum_{(z,w)}  (-1)^{|x|+|y|-|z|-|w|} \pi_*s_{(x,y),(z,w)} \times \pi_*s_{(z,w),(x',y')}\\
    &= (-1)^{|y|-|y'|} s_{(x,y),(x,y')} \times s_{(x,y'),(x',y')} + (-1)^{|x|-|x'|} s_{(x,y),(x',y)} \times s_{(x',y),(x',y')}\\
    &= (-1)^{|x|+|x'|(1 + |y|-|y'|) +1} \left((s_{x,x'},s_{y,y'}) -  (s_{x,x'},s_{y,y'})\right) = 0.
\end{align*}

Since $\trajbi{f}{x,x'} \times \trajbi{g}{y,y'}$ is a manifold of dimension $l-1$ every $l$-cycle is a boundary. Hence, there exists $b \in C_{l+1}(\trajbi{f}{x,x'} \times \trajbi{g}{y,y'})$ such that $\partial b = \pi_*(s'_{(x,y),(x',y')})$. We then define $s_{(x,y),(x',y')} = s'_{(x,y),(x',y')} - \partial i_*(b) $, which satisfies condition 3. The resulting representing chain system $\{s_{(x,y),(x',y')}\}$ satisfies all the conditions.
    
\end{proof}

\subsubsection{Künneth twisting cocycle}

We now define evaluation maps adapted to the cartesian product.

\begin{lemme}\label{lemma : q^K}
The family 
   $$\xymatrix{q^K_{(x,y),(x',y')} : \trajb{(x,y),(x',y')} \ar[r]^-{\pi} & \trajb{x,x'} \times \trajb{y,y'} \ar[rr]^{\quad (q^X_{x,x'}, q^Y_{y,y'})} & & \Omega X \times \Omega Y }.$$ 
    is a family of evaluation maps in the sense of Definition \ref{defi : Evaluation maps}, ie it satisfies the following two conditions

\begin{enumerate}
    \item If $|x| + |y| - |x'|-|y'| = 1$ and $\lambda \in \trajb{(x,y),(x',y')} = \traj{(x,y),(x',y')}$, the class $[q^K_{(x,y),(x',y')}(\lambda)] \in \pi_1(X \times Y)$ is the one associated to $\lambda$ in the lifted Morse complex.
    \item If $(\lambda, \lambda') \in \trajb{(x,y),(z,w)} \times \trajb{(z,w),(x',y')} \subset \partial\trajb{(x,y),(x',y')}$, then $$q^K_{(x,y),(x',y')}(\lambda, \lambda') = q^K_{(x,y),(z,w)}(\lambda) \# q^K_{(z,w),(x',y')}(\lambda').$$
\end{enumerate}

\end{lemme}

\begin{proof}
   Let $(x,y),(z,w),(x',y') \in \Crit(H)$ and  $(\lambda, \lambda') \in \trajb{(x,y),(z,w)} \times \trajb{(z,w),(x',y')} $. If we denote $\pi_X = pr_1 \circ \pi : \trajb{(x,y),(x',y')} \to \trajb{x,x'}$ and $\pi_Y = pr_2 \circ \pi : \trajb{(x,y),(x',y')} \to \trajb{y,y'}$, then $\pi_X(\lambda \# \lambda') = (\pi_X(\lambda), \pi_X(\lambda'))$ and $\pi_Y(\lambda \# \lambda') = (\pi_Y(\lambda), \pi_Y(\lambda'))$. Condition 2 is then clear because $q^X$ and $q^Y$ satisfy such concatenation relations.
    Condition 1 is a consequence of the fact that, if $|x| + |y| - |x'|-|y'| = 1$, then $x=x'$ or $y=y'$.
\end{proof}

\begin{defi}\label{defi : m^K}

The family of evaluation maps $q^K$ yields a twisting cocycle 

$m^K_{(x,y),(x',y')} = q^K_{(x,y),(x',y'),*}(s_{(x,y),(x',y')}) \in C_{|x|+|y|-|x'|-|y'|-1}(\Omega X \times \Omega Y))$ that satisfies 

$$m^K_{(x,y),(x',y')} = \left\{ \begin{array}{ll}
   (-1)^{|x|(|y|-|y'|)} (\star, m_{y,y'}) & \textup{if} \ x=x',  \\
     (m_{x,x'}, \star)& \textup{if} \ y=y', \\
     0 & \textup{otherwise}
\end{array} \right..$$

This is the preferred twisting cocycle associated to the set of DG Morse data $\Xi_{X \times Y}$. We will refer to this cocycle as the \textbf{Künneth twisting cocycle}.

\end{defi}

The evaluation map defining the Barraud-Cornea twisting cocycle $m^0_{(x,y),(x',y')}$ is $$q^0_{(x,y),(x',y')} = \theta \circ p \circ \Gamma^{X\times Y}_{(x,y),(x',y')} : \trajb{(x,y),(x',y')} \to \Omega(X \times Y) \subset \Omega X \times \Omega Y,$$ where the parametrization map $\Gamma^{X\times Y}_{(x,y),(x',y')} : \trajb{(x,y),(x',y')} \to \mathcal{P}_{(x,y) \to (x',y')} X \times Y$ is defined by 
$$\Gamma^{X\times Y}_{(x,y),(x',y')}(\lambda) : \begin{array}{ccc}
   [0, H(x,y) - H(x',y')] & \to & X \times Y   \\
    t & \mapsto &  \lambda \cap H^{-1}(H(x,y) -t).
\end{array}$$

\begin{prop}\label{prop : dg Kunneth m^0 and m^1 quasi-iso}
    For any right $\Ai$-module $\mathcal{H}$ over $C_*(\Omega X \times \Omega Y)$, there exists a chain homotopy equivalence $$C_*(X \times Y, m^0_{(x,y),(x',y')},\mathcal{H}) \underset{\Psi_{0K}}{\overset{\Psi_{K0}}{\leftrightarrows}} C_*(X \times Y, m^K_{(x,y),(x',y')}, \mathcal{H}).$$ 
\end{prop}

\subsection{Invariance with respect to the parametrization}\label{subsection : Invariance with respect to the parametrization}

The proof that the cocycles $m^0$ and $m^K$ defined above induce the same homology will be given in the second subsection. We first prove that we can use any Morse function that admits the same pseudo-gradient $\xi$ to parametrize the trajectories. The subsequent complexes will be chain homotopy equivalent.

\subsubsection{Parametrizing Morse trajectories by different Morse functions}

Let $\Xi = (f,\xi, s_{x,y}, o , \mathcal{Y}, \theta)$ be a set of DG Morse data on a pointed, oriented, closed, and connected manifold $(X,\star)$ and let $h,h' : X \to \R$ be  Morse functions that admit the vector field $\xi$ as an adapted pseudo-gradient. We prove that the complex built using $h$ for parametrizing the trajectories in the construction of the twisting cocycle (see Definition \ref{defi : Evaluation maps}) is chain homotopy equivalent to the complex built using the Morse function $h'$. This is equivalent to saying that the enriched complex does not depend on the Morse function used to parametrize the trajectories of a given adapted pseudo-gradient up to chain homotopy equivalence.

This does not prove Proposition \ref{prop : dg Kunneth m^0 and m^1 quasi-iso} since there is no Morse function $h : X \times Y \to \R$ that corresponds to parametrizing independently on $X$ and on $Y$, but this invariance is worth noticing and its proof will inspire the proof of Proposition \ref{prop : dg Kunneth m^0 and m^1 quasi-iso}. 

\begin{prop}\label{prop : Parametrizing by different Morse functions}

Let $h,h' : X \to \R$ be two Morse functions and $\xi$ a pseudo-gradient adapted to both $h$ and $h'$. Let $\F$ be an $\Ai$-module over $C_*(\Omega X)$.
Let $C_*(X,\Xi_{h'}, \F)$ and $C_*(X,\Xi_{h}, \F)$ be the two complexes built using respectively $h'$ and $h$ for parameterizing the trajectories in the definition of the evaluation maps.

There exists a quasi-isomorphism $\Phi_{h,h'} : C_*(X,\Xi_h, \F) \to C_*(X,\Xi_{h'}, \F)$.

\end{prop}

\begin{proof}
    We use the same construction as in the section 6.2 of \cite{BDHO23} where the authors prove that the constant homotopy between a Morse function and itself induces a continuation map that is homotopic to the identity. Let $$\pi : [0,1] \times X \to X$$ and $$p : X \to \Xq$$ be the canonical projections and let $$\theta :  \Xq \to X $$ be a homotopy inverse of $p.$
    
    We follow the proof of invariance for Morse theory \cite[Section I.3.4]{AD14} to build $a : [-\epsilon,1+\epsilon] \to \R$ a strictly decreasing function on $[0,1]$ such that $a$ is a Morse function and $\Crit(a) = \{0,1\}$. We can assume that $a$ is decreasing enough for $F = h + a : [-\epsilon, 1+\epsilon] \times X \to \R$ to be a Morse function whose critical points are, for any $k \in \N$, $$\Crit_k(F) = \{0\} \times \Crit_{k-1}(h) \cup \{1\} \times \Crit_k(h)$$ and let $\xi_F$ be an adapted pseudo-gradient for $F$ on $[0,1] \times X.$
    
    If $x \in \Crit_k(h)$, we denote $x_0 \in \{0\} \times \Crit_{k}(h) \subset \Crit_{k+1}(F)$ and $x_1 \in \{1\} \times \Crit_k(h) \subset \Crit_k(F)$ the corresponding critical points of $F$.
    Let $\{s_{x,y}\}$ be a representing chain system for the moduli spaces of Morse trajectories. There exists a representing chain system $\{s^F_{x_i,y_j}\}$ for the moduli spaces of Morse trajectories in $[0,1] \times X$ such that for any $x,y \in \textup{Crit}(h)$ :

    \begin{itemize}
        \item $s^F_{x_0,y_0} = (-1)^{|x|-|y|} s_{x,y} \in C_{|x|-|y|-1}(\trajb{x_0,y_0})$.
        \item $s^F_{x_1,y_1} = s_{x,y} \in C_{|x|-|y|-1}(\trajb{x_1,y_1})$.
        \item We denote $\sigma^{\textup{Id}}_{x,y} = s^F_{x_0, y_1} \in C_{|x|-|y|}(\trajb{x_0,y_1})$.
        \item Since $a$ strictly decreases between $0$ and $1$, then $\trajb{x_1,y_0} = \emptyset.$ \\
    \end{itemize}

    Define now an evaluation map $q^{h,h'}_{x_i,y_j} : \trajb{x_i,y_j} \to \Omega X$ such that 
    \begin{itemize}
        \item $q^{h,h'}_{x_0,y_0} := q^0_{x,y} := \theta \circ p \circ \pi \circ \Gamma^h_{x_0,y_0}$ where $\Gamma^h_{x_0,y_0} : \trajb{x_0,y_0} \to \mathcal{P}_{x_0 \to y_0} (\{0\} \times X)$ parametrizes the trajectory using the Morse function $h$.
        This is the evaluation map that arises from the data $\Xi_h$ and therefore $q^{h,h'}_{x_0,y_0,*}(s_{x,y}) = m^h_{x,y}$ is the twisting cocycle defining $C_*(X, \Xi_h, \F)$.\\
        
        \item $q^{h,h'}_{x_1,y_1} := q^1_{x,y} := \theta \circ p \circ \pi \circ \Gamma^{h'}_{x_1,y_1}$ where $\Gamma^{h'}_{x_1,y_1} : \trajb{x_1,y_1} \to \mathcal{P}_{x_1 \to y_1} (\{1\} \times X)$ parametrizes the trajectory using the Morse function $h'$. This is the evaluation map that arises from the data $\Xi_{h'}$ and therefore $q^{h,h'}_{x_1,y_1,*}(s_{x,y}) = m^{h'}_{x,y}$ is the twisting cocycle defining $C_*(X, \Xi_{h'}, \F)$.\\
    \end{itemize}

    It remains to define $q^{h,h'}_{x_0,y_1}$. Let $\lambda \in \trajb{x_0,y_1}$. Since $a$ is strictly decreasing between $0$ and $1$, the trajectory $\lambda$ transversely intersects $\{\frac{1}{2}\} \times X$ exactly once. Denote $ z \in \lambda \cap X \times \{\frac{1}{2}\}$. We define $$\Gamma_{x_0,y_1}^{h,h'} : \trajb{x_0,y_1} \to \mathcal{P}_{x_0 \to y_1} ([0,1] \times X)$$ by parametrizing the piece of $\lambda$ going from $x_0$ to $z$ using $h$ and then parametrizing the other piece of $\lambda$ going from $z$ to $y_1$ with $h'$.

    Let us reformulate this idea with some formulas:
    
    We assume w.l.o.g that $A=a(0)>0$ and $a(1)=0$.
Define
    $$\Gamma^h_{x_0,y_1}(\lambda) : [0, A + h(x)-h(y)] \to [0,1] \times X, $$
    $$\Gamma^h_{x_0,y_1}(\lambda)(t) = (h+a)^{-1}( A + h(x) -t) \cap \lambda$$

    and 

    $$\Gamma^{h'}_{x_0,y_1}(\lambda) : [0, A + h'(x)-h'(y)] \to X \times [0,1], $$
    $$\Gamma^{h'}_{x_0,y_1}(\lambda)(t) = (h'+a)^{-1}( A + h'(x) -t) \cap \lambda.$$

    There exists a unique $t_h \in [0, A + h(x)-h(y)]$ such that $\Gamma^h_{x_0,y_1}(\lambda) \in X \times \{1/2\}$ and a unique $t_{h'} \in [0, A + h'(x)-h'(y)]$ such that $\Gamma^{h'}_{x_0,y_1}(\lambda) \in X \times \{1/2\}$.

    Let,
    $$\Gamma_{x_0,y_1}^{h,h'}(\lambda) : [0, A + h'(x)-h'(y) + t_h -t_{h'}] \to X \times [0,1],$$
    $$\Gamma_{x_0,y_1}^{h,h'}(\lambda)(t) = \left\{ \begin{array}{cc}
       \Gamma^{h}_{x_0,y_1}(\lambda)(t)  & \textup{if} \ t\leq t_h  \\
        \Gamma^{h'}_{x_0,y_1}(\lambda)(t-t_h + t_h')  & \textup{if} \ t\geq t_h
    \end{array} \right..$$

    It is clear from the definition that these evaluation maps satisfy : \begin{enumerate}
        \item \underline{The concatenation relations} : \begin{itemize} \item for any $(\lambda, \lambda') \in \trajb{x_0,z_0} \times \trajb{z_0, y_1}\subset \partial \trajb{x_0,y_1}$, $q^{h,h'}_{x_0,y_1}(\lambda, \lambda') = q^0_{x,z}(\lambda) \# q^{h,h'}_{z_0, y_1}(\lambda')$. 
        \item for any $(\lambda, \lambda') \in \trajb{x_0,z_1} \times \trajb{z_1, y_1}\subset \partial \trajb{x_0,y_1}$, $q^{h,h'}_{x_0,y_1}(\lambda, \lambda') = q^{h,h'}_{x_0,z_1}(\lambda) \# q^{1}_{z, y}(\lambda')$.
    \end{itemize}
    \item \underline{Compatibility with the lifted Morse complex} : Let $\{q^F_{x_0,y_1} : \trajb{x_0,y_1} \to \Omega X, x,y \in \Crit(h)\}$ be the family of evaluation maps defined by parametrizing by the values of $F$. Since it is only the parametrization of the path that changes between the evaluation maps $q^{h,h'}$ and $q^F$, then $q^{h,h'}_{x_0,y_1}$ also satisfies the condition that, if $x$ and $y$ have the same Morse index, the homotopy class $[q^{h,h'}_{x_0,y_1}(\lambda)] \in \pi_1(X)$ is the same as the one associated with $\lambda \in \trajb{x_0,y_1}$ in the lifted Morse complex $\tilde{C}_*(F,\xi')$ (see Lemma 6.2.4 of \cite{BDHO23}).
    \end{enumerate}
    For any $x,y \in \Crit(h)$, define $\nu^{h,h'}_{x,y} = - q^{h,h'}_{x_0,y_1,*}(\sigma^{\textup{Id}}_{x,y}) \in C_{|x|-|y|}(\Omega X).$

    Since $\sigma^{\textup{Id}}_{x,y}$ satisfies the relation $\partial \sigma^{\textup{Id}}_{x,y} = \displaystyle \sum_z s_{x,z} \sigma^{\textup{Id}}_{z,y} - \sum_{w} (-1)^{|x|-|w|} \sigma^{\textup{Id}}_{x,w} s_{w,y}$, we have

    $$\partial \nu^{h,h'}_{x,y} = \sum_z m^h_{x,z} \nu^{h,h'}_{z,y} - \sum_w (-1)^{|x|-|w|} \nu^{h,h'}_{x,w} m^{h'}_{w,y}.$$

    The quasi-isomorphism criterion \ref{Prop : QIso Criterion Ai} implies that $\{\nu^{h,h'}_{x,y}\}$ yields a morphism of complexes $$\Phi_{h,h'} : C_*(X,\Xi_h,\F) \to C_*(X, \Xi_{h'}, \F)$$ which is a quasi-isomorphism.
    \newpage
    Indeed, the map between the associated lifted Morse complexes is the continuation map for the lifted Morse complexes (as in the proof \cite[Proposition 6.2.2]{BDHO23}) which is therefore a quasi-isomorphism.    
\end{proof}

In fact more is true :

\begin{prop}\label{prop : equivalence homotopie param par Morse function}
    The map $\Phi_{h,h'} : C_*(X,\Xi_h, \F) \to C_*(X,\Xi_{h'}, \F)$ built in Proposition \ref{prop : Parametrizing by different Morse functions} is actually a chain homotopy equivalence.
\end{prop}

\begin{proof}
    We prove here that $\Phi_{h,h'} \circ \Phi_{h',h} : C_*(X,\Xi_h, \F) \to C_*(X,\Xi_h, \F)$ is homotopic to the composition $$ C_*(X,\Xi_h, \F) \overset{\Id_*}{\to} C_*(X,\Xi_{h'}, \F) \overset{\Id_*}{\to} C_*(X,\Xi_{h}, \F).$$ The last map being chain homotopic to the identity, the proposition follows.

    We will use the notations introduced in the previous proof and follow the idea of the proof of \cite[Proposition 6.3.8]{BDHO23} to construct a homotopy cocycle (see Proposition \ref{Prop : homotopy Criterion Ai}).

    Define $K = h + a_t + a_{\tau} : [0,1]^2 \times X  \to \R$ by $K(t,\tau,x) = h(x) + a(t) + a(\tau)$. This is a Morse function whose critical points are 
    $$\Crit_{k+1}(K) =  \{(0,0)\} \times \Crit_{k-1}(h)  \cup  \{(0,1)\} \times \Crit_{k}(h)  \cup  \{(1,0)\} \times \Crit_{k}(h)  \cup \{(1,1)\} \times \Crit_{k+1}(h).$$

    For each critical point $x \in \Crit(K)$, we use the notation $x_{ij}$ if $x \in \{(i,j)\} \times  \Crit(h)$  for $i,j \in \{0,1\}$.

    Following the proof of \cite[Lemma 6.3.9]{BDHO23}, we can construct a representing chain system $(s^K_{x,y})$ for the moduli spaces of trajectories $\trajbi{K}{x,y}$ such that,
    
        if we denote $S_{x,y} = s^K_{x_{00}, y_{11}} \in C_{|x|-|y|+1}(\trajbi{K}{x_{00}, y_{11}})$,

        \begin{align*}
            \partial S_{x,y} &= \sum_{z_{00}} (-1)^{|x|-|y|} s^f_{x,z}\cdot S_{z,y} + \sum_{w_{11}} (-1)^{|x|-|w|+2} S_{x,w} \cdot s^f_{w,y}\\
            &- \sum_{u_{10}} \sigma^{Id}_{x,u} \cdot \sigma^{Id}_{u,y} + \sum_v \sigma^{Id}_{x,v}\cdot\sigma^{Id}_{v,y}.\\
        \end{align*}
   
We now want to evaluate these chains into $C_*(\Omega X)$. 

Define now for each $x,y \in \Crit(K)$ a parametrization map $\Gamma_{x,y} : \trajbi{K}{x, y} \to \mathcal{P}_{x \to y}([0,1]^2 \times X )$. Let $\lambda \in \trajbi{K}{x, y}$ and $C = [1/2,1] \times [0,1/2] \times X $. The trajectory $\lambda$ intersects $C$ transversely since $a$ is strictly decreasing and $\Gamma_{x,y}(\lambda)$ is the path parametrizing $\lambda$ by $h$ outside of $C$ and by $h'$ inside $C$ (see the following figure).\\
    
\begin{center}
\begin{overpic}[scale=0.65]{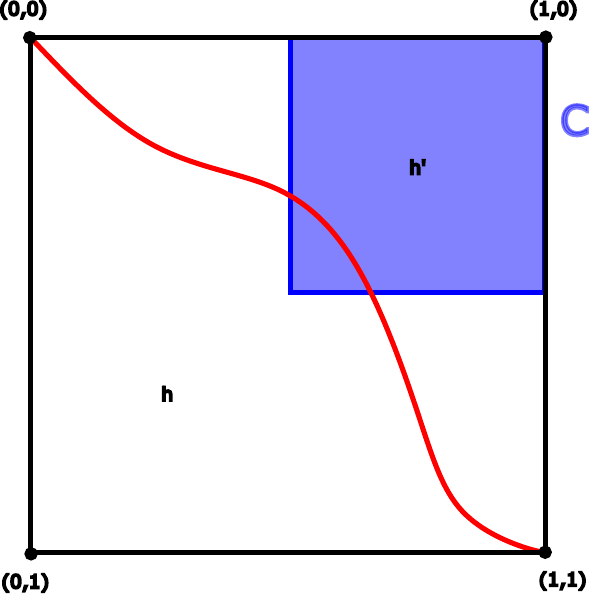}
    \put(38,59){\red{$\lambda$}}
\end{overpic}    

    Figure 3 - Parametrization map $\Gamma_{x,y}$\\
    
\end{center}

\noindent We define for each $x,y \in \Crit(K)$ the evaluation maps $$Q_{x, y} = \pi \circ (Id_{[0,1]^2} \times \theta) \circ (Id_{[0,1]^2} \times p) \circ \Gamma_{x,y}: \trajbi{K}{x, y} \to \Omega X,$$ where $\pi : [0,1]^2 \times X \to X $ is the canonical projection.
In particular, if $\lambda \in \trajbi{K}{x_{00}, y_{10}}$ and $\lambda' \in \trajbi{K}{x_{10}, y_{11}}$, then $Q_{x,y}(\lambda) = q^{h,h'}_{x,y}(\lambda)$ and  $Q_{x,y}(\lambda') = q^{h',h}_{x,y}(\lambda')$.\\

These evaluation maps again satisfy the concatenation relations and compatibility with the lifted Morse complex.

Denote $H_{x,y} = Q_{x,y,*}(S_{x,y}) \in C_{|x|-|y|+1}(\Omega X),$

\begin{align*}
\partial H_{x,y} &= \sum_z (-1)^{|x|-|z|} m^h_{x,z} \cdot H_{z,y} + \sum_w (-1)^{|x|-|w|} H_{x,w} \cdot m^h_{w,y} \\
&- \sum_u \nu^{h,h'}_{x,u} \cdot \nu^{h',h}_{u,y} + \sum_v \nu^{h,h}_{x,v} \cdot \nu^{h,h}_{v,y}.
\end{align*}

Using Proposition \ref{Prop : homotopy Criterion Ai}, this proves that the map $$\deffct{H}{C_*(X, \Xi_h, \F)}{C_{*+1}(X, \Xi_h, \F)}{\alpha \otimes x}{\displaystyle \sum_y \alpha \cdot H_{x,y} \otimes y}$$ is a chain homotopy between the map $$\deffct{\Phi_{h',h} \circ \Phi_{h,h'}}{C_*(X, \Xi_h, \F)}{C_*(X, \Xi_h, \F)}{\alpha \otimes x}{ \displaystyle \sum_y \alpha \cdot \left( \sum_u \nu^{h,h'}_{x,u} \cdot \nu^{h',h}_{u,y} \right) \otimes y }$$

and the map $$\deffct{\textup{Id}_* \circ \textup{Id}_*}{C_*(X, \Xi_h, \F)}{C_*(X, \Xi_h, \F)}{\alpha \otimes x}{\displaystyle \sum_y \alpha \cdot \left( \sum_v \nu^{h,h}_{x,v} \cdot \nu^{h,h}_{v,y} \right) \otimes y}.$$

\end{proof}

\subsubsection{Parametrizing independently on a Cartesian product}\label{subsubsection : Parametrizing independantly on a Cartesian product}

Let $f : X \to \R$, $g: Y \to \R$ and $H = f\oplus g : X \times Y \to \R$ be Morse functions.

Let $\{s_{z,w} \in C_{|z|-|w|-1}(\trajbi{H}{z,w}), \ z,w \in \Crit(H)\}$ be the twisting cocycle constructed in Section \ref{subsection : Morse data on a Cartesian product}.

Recall the notation of the previous section : $m^K_{z,w} = q^K_{z,w,*}(s_{z,w})$ is the Künneth twisting cocycle from Definition \ref{defi : m^K} and $m^0_{z,w} = q^0_{z,w,*}(s_{z,w})$ is the Barraud-Cornea twisting cocycle associated with the DG Morse data set $\Xi_{X\times Y}.$

\begin{prop}
    For any $\Ai$-module $\mathcal{H}$ over $C_*(\Omega X \times \Omega Y)$, there exists a quasi-isomorphism $$\Psi^{0K} : C_*(X\times Y, \Xi_{X \times Y},\mathcal{H}) \to C_*(X\times Y, m^K,\mathcal{H}).$$
\end{prop}

\begin{proof}
Proceed as in the proof of Proposition \ref{prop : Parametrizing by different Morse functions} to define $W : [0,1] \times X \times Y \to \R$ a Morse function such that $\Crit_k(W) = \{0\} \times \Crit_{k-1}(H) \cup \{1\} \times \Crit_k(H)$, a pseudo-gradient $\xi_W$ and a representing chain system $\{s^W_{z_i, w_j}, \ z,w \in \Crit(H), \ i,j \in \{0,1\}\}$ such that

\begin{itemize}
        \item $s^W_{z_0,w_0} = (-1)^{|z|-|w|} s_{z,w} \in C_{|z|-|w|-1}(\trajbi{W}{z_0,w_0})$.
        \item $s^W_{z_1,w_1} = s_{z,w} \in C_{|z|-|w|-1}(\trajbi{W}{z_1,w_1})$.
        \item We denote $\sigma_{z,w} = s^W_{z_0, w_1} \in C_{|z|-|w|}(\trajbi{W}{z_0,w_1})$.
\end{itemize}

Denote $\pi^{X \times Y} : [0,1] \times X \times Y  \to X \times Y$, $\pi^X : [0,1] \times X \to X$ and $\pi^Y : [0,1] \times Y \to Y$ the canonical projections. We can assume w.l.o.g that $F = f+a : [0,1] \times X \to \R$ and $G = g +a : [0,1] \times Y \to \R$ are Morse functions. Denote $\xi_F = \xi_f - a' \frac{\partial}{\partial t}$ and $\xi_G = \xi_g - a' \frac{\partial}{\partial t}$ adapted pseudo-gradients for $F$ and $G$ respectively.

There is an inclusion $$ I : \mathcal{M}_W((x,y)_0, (x',y')_1) \to \mathcal{M}_F(x_0,x'_1) \times \mathcal{M}_G(y_0,y'_1), \quad I(\tau,(a,b))= \left((\tau,a), (\tau,b) \right)$$

which induces a map $$J = (J_X, J_Y) : \trajbi{W}{(x,y)_0, (x',y')_1} \to \trajbi{F}{x_0,x'_1} \times \trajbi{G}{y_0,y'_1}.$$

Define $\Gamma^K_{(x,y)_0,(x',y')_1}: \trajbi{W}{(x,y)_0, (x',y')_1} \to \mathcal{P}_{x_0 \to x'_1} ([0,1] \times X) \times \mathcal{P}_{y_0 \to y'_1} ([0,1] \times Y)  $ by 
$$\Gamma^K_{(x,y)_0,(x',y')_1}(\lambda) = \left(\Gamma^F_{x_0,x'_1}\circ J_X(\lambda), \Gamma^G_{y_0,y'_1}\circ J_Y(\lambda) \right),$$

where $\Gamma^F_{x_0,x'_1} : \trajbi{F}{x_0,x_1} \to \mathcal{P}_{x_0 \to x'_1} ([0,1] \times X)$ and $\Gamma^G_{y_0,y'_1} : \trajbi{G}{y_0,y'_1}\to \mathcal{P}_{y_0 \to y'_1}([0,1] \times Y)$ are the parametrization maps by the values of the Morse functions $F$ and $G$ respectively. See Definition \ref{defi : Evaluation maps}.\\

Denote $$\Gamma^0_{(x,y)_0, (x',y')_1} : \trajbi{W}{(x,y)_0, (x',y')_1} \to \mathcal{P}_{x_0 \to x'_1} ([0,1] \times X) \times \mathcal{P}_{y_0 \to y'_1} ([0,1] \times Y)$$
the composition of the map $$\Gamma^W_{(x,y)_0, (x',y')_1} : \trajbi{W}{(x,y)_0, (x',y')_1} \to \mathcal{P}_{(x,y)_0 \to (x',y')_1} ([0,1] \times X \times Y )$$ which parametrizes by the values of $W$ and  the inclusion $$\mathcal{P}_{(x,y)_0 \to (x',y')_1} ([0,1] \times X \times Y ) \hookrightarrow \mathcal{P}_{x_0 \to x'_0} ([0,1] \times X ) \times \mathcal{P}_{y_0 \to y'_1} ([0,1] \times Y).$$

Now define $$\Gamma^{0K}_{(x,y)_0, (x',y')_1} : \trajbi{W}{(x,y)_0, (x',y')_1} \to \mathcal{P}_{x_0 \to x'_1} ([0,1] \times X ) \times \mathcal{P}_{y \to y'} ([0,1] \times Y)$$  the parametrization map that shifts the parametrization from $\Gamma^0$ to $\Gamma^K$ when the trajectory crosses $X \times Y \times \{\frac{1}{2}\}$.

\begin{lemme}

The family $q^{0K}_{(x,y)_0, (x',y')_1} =  \theta \circ p \circ (\pi^X, \pi^Y) \circ \Gamma^{0K}_{(x,y)_0, (x',y')_1} : \trajb{(x,y)_0, (x',y')_1} \to \Omega X \times \Omega Y$ is a family of evaluation maps, ie  it satisfies

\begin{enumerate}
    \item \underline{The concatenation relations} : For any $(\lambda,\lambda') \in \trajb{(x,y)_0, (z,w)_0} \times \trajb{(z,w)_0 , (x',y')_1}$, $$ q^{0K}_{(x,y)_0, (x',y')_1}(\lambda, \lambda') = q^0_{(x,y), (z,w)}(\lambda) \# q^{0K}_{(z,w)_0, (x',y')_1}(\lambda') $$
and for any $(\lambda,\lambda') \in \trajb{(x,y)_0, (z,w)_1} \times \trajb{(z,w)_1 , (x',y')_1}$, $$q^{0K}_{(x,y)_0, (x',y')_1}(\lambda, \lambda') = q^{0K}_{(x,y)_0, (z,w)_1}(\lambda) \# q^{K}_{(z,w)_1, (x',y')_1}(\lambda'). $$
\item \underline{The compatibility with the lifted Morse complex} : if $(x,y)$ and $(x',y')$ have the same index, then the homotopy class $[q^{0K}_{(x,y)_0,(x',y')_1}(\lambda)] \in \pi_1(X \times Y)$ is the same as the one associated with $\lambda \in \trajb{(x,y)_0,(x',y')_1}$ in the lifted Morse complex of $X \times Y \times [0,1]$.
\end{enumerate}
\end{lemme}

\begin{proof}
    
Condition 1. is clearly satisfied.

Condition 2. is satisfied because, even if $\Gamma^{0K}_{(x,y)_0, (x',y')_1}(\lambda)$ and $\Gamma^{0}_{(x,y)_0, (x',y')_1}(\lambda)$ do not have the same image, they are homotopic. Indeed, if we denote $(\gamma^{0K}_X, \gamma^{0K}_Y) = \Gamma^{0K}_{(x,y)_0, (x',y')_1}(\lambda)$ and $(\gamma^0_X, \gamma^0_Y) = \Gamma^{0}_{(x,y)_0, (x',y')_1}(\lambda)$, then $\gamma^{0K}_X = \gamma_X^0$ and $\gamma^{0K}_Y = \gamma_Y^0$ up to reparametrization but $(\gamma^{0K}_X, \gamma^{0K}_Y)$ and $ (\gamma^{0}_X, \gamma^{0}_Y)$ do not have the same image. However, it is enough to reparametrize $\gamma^{0K}_Y$ in $\gamma'^{0K}_Y$ so that $(\gamma^{0K}_X, \gamma'^{0K}_Y) = (\gamma^{0}_X, \gamma^{0}_Y)$ up to parametrization. Therefore, since $q^0$ satisfies the 
condition 2, then so does $q^{0K}$.
\end{proof}

From the concatenation relations, it follows that $$\nu^{0K}_{(x,y),(x',y')} = q^{0K}_{(x,y)_0, (x',y')_1,*}(\sigma_{(x,y), (x',y')}) \in C_{|x|+|y|-|x'|-|y'|}(\Omega X \times \Omega Y)$$ satisfies equation \eqref{eq : mu_1 de tau}

$$\partial \nu^{0K}_{(x,y),(x',y')} = \sum_{z,w} m^{0}_{(x,y),(z,w)} \nu^{0K}_{(z,w),(x',y')} - \sum_{z',w'}(-1)^{|x|+|y|-|z'|-|w'|} \nu^{0K}_{(x,y),(z',w')} m^{K}_{(z',w'),(x',y')}.$$

This concludes the proof because the compatibility with the lifted Morse complex ensures that the continuation cocycle satisfies the quasi-isomorphism criterion \ref{Prop : QIso Criterion Ai}.

Therefore

$$\deffct{\Psi_{0K}}{C_*(X\times Y, \Xi_{X\times Y},\mathcal{H})}{C_*(X\times Y, m^K, \mathcal{H})}{\alpha \otimes (x,y)}{\displaystyle\sum_{(x',y')} \alpha \cdot \nu^{0K}_{(x,y),(x',y')} \otimes (x',y')}$$ is a quasi-isomorphism for any right $\Ai$-module $\mathcal{H}$ over $C_*(\Omega X \times \Omega Y)$.

\end{proof}

The following lemma concludes the proof of Proposition \ref{prop : dg Kunneth m^0 and m^1 quasi-iso}.

\begin{prop}\label{lemme : Psi0K is a chain homotopy equivalence}
    The map $\Psi_{0K}$ is a chain homotopy equivalence.
\end{prop}

\begin{proof}
    We will use the notation introduced in the previous proof.
    
    We can define for any $x,x'\in \Crit(f)$, $y,y'\in \Crit(g)$ the parametrization map 

    $$\Gamma^{K0}_{(x,y),(x',y')} : \trajbi{W}{(x,y)_0, (x',y')_1} \to \mathcal{P}_{x_0 \to x'_1}([0,1] \times X) \times \mathcal{P}_{y_0 \to y'_1}([0,1] \times Y)$$ that shifts the parametrization from $\Gamma^K$ to $\Gamma^0$ when $\lambda$ crosses $\{\frac{1}{2}\} \times X \times Y$.\\

    Let us describe this map in more detail. Let $\lambda \in \trajbi{W}{(x,y)_0, (x',y')_1}$ and let $(\frac{1}{2},z_X,z_Y) \in \{\frac{1}{2}\} \times X \times Y$ be the point where $\lambda$ crosses transversely $\{\frac{1}{2}\} \times X \times Y$. The paths $(\gamma^K_X,\gamma^K_Y):=\Gamma^{K}_{(x,y)_0, (x',y')_1}(\lambda)$ can be restricted to paths that we still denote $$(\gamma^K_X,\gamma^K_Y) \in \mathcal{P}_{x_0\to (\frac{1}{2},z_X)}([0,1] \times X) \times \mathcal{P}_{y_0 \to (\frac{1}{2},z_Y)}([0,1] \times Y).$$
    
    Denote $$(\gamma^0_X,\gamma^0_Y) \in \mathcal{P}_{(\frac{1}{2},z_X) \to x'_1}([0,1] \times X) \times \mathcal{P}_{(\frac{1}{2},z_Y) \to y'_1}([0,1] \times Y),$$ the paths obtained by parametrizing the half-infinite gradient line starting at $(\frac{1}{2},z_X,z_Y)$ using $\Gamma^0$, or equivalently, obtained by restricting $\Gamma^W_{(x,y)_0, (x',y')_1}(\lambda)$ to start at $(\frac{1}{2},z_X,z_Y)$ and applying the inclusion $\mathcal{P}_{ (\frac{1}{2},z_X,z_Y) \to (x',y')_1} ([0,1] \times X \times Y ) \hookrightarrow \mathcal{P}_{(\frac{1}{2},z_X) \to x'_1} ([0,1] \times X ) \times \mathcal{P}_{(\frac{1}{2},z_Y) \to y'_1} ([0,1] \times Y)$.

    We define $$\Gamma^{K0}_{(x,y),(x',y')}(\lambda)= (\gamma^{K}_X \cdot \gamma^0_X,\gamma^{K}_Y \cdot \gamma^0_Y) \in \mathcal{P}_{x_0 \to x'_1}([0,1] \times X) \times \mathcal{P}_{y_0 \to y'_1}([0,1] \times Y)  .$$

    We can therefore construct $\nu^{K0}_{(x,y),(x',y')} \in C_{|x|+|y|-|x'|-|y'|}(\Omega X \times \Omega Y)$ such that $$\partial \nu^{K0}_{(x,y),(x',y')} = \sum_{z,w} m^{K}_{(x,y),(z,w)} \nu^{K0}_{(z,w),(x',y')} - \sum_{z',w'}(-1)^{|x|+|y|-|z'|-|w'|} \nu^{K0}_{(x,y),(z',w')} m^{0}_{(z',w'),(x',y')}$$

    and obtain a quasi-isomorphism $$\deffct{\Psi_{K0}}{C_*(X\times Y, m^K,\mathcal{H})}{C_*(X\times Y, \Xi_{X \times Y}, \mathcal{H})}{\alpha \otimes (x,y)}{\displaystyle\sum_{(x',y')} \alpha \cdot \nu^{K0}_{(x,y),(x',y')} \otimes (x',y')}$$ for any  right $\Ai$-module $\mathcal{H}$ over $C_*(\Omega X \times \Omega Y)$.
    From there, the same proof as Proposition \ref{prop : equivalence homotopie param par Morse function} shows that $\Psi_{K0} \circ \Psi_{0K} \simeq \Id$ and $\Psi_{0K} \circ \Psi_{K0} \simeq \Id$.
\end{proof}

In view of Proposition \ref{prop : dg Kunneth m^0 and m^1 quasi-iso}, we may and will assume that the complex $C_*(X\times Y, \Xi_{X \times Y}, \mathcal{H})$ is constructed using the Künneth twisting cocycle for any right $\Ai$-module $\mathcal{H}$ over $C_*(\Omega X \times \Omega Y).$

\subsection{Künneth formula. Proof of Theorem C}\label{subsection : Künneth map}

The goal of this section is to define an analogue of the Künneth formula for enriched
 Morse homology using the twisting cocycle defined above. In Section
\ref{subsection : Morse data on a Cartesian product}
and \ref{subsection : Invariance with respect to the parametrization}
we dealt with how sets of DG Morse data $\Xi_X$ and $\Xi_Y$ on $X$ and $Y$
induce a preferred set of DG Morse data $\Xi_{X \times Y}$ on $X \times Y$
and how to construct the preferred Künneth twisting cocycle on $X \times Y$ computing the same homology as the Barraud-Cornea twisting cocycle. It remains to understand how DG systems
  $\F$ over $C_*(\Omega X)$ and $\G$ over $C_*(\Omega Y)$
   induce a DG system over $C_*(\Omega( X \times Y)).$ \\

\subsubsection{The cross products \texorpdfstring{$K^{alg}$}{Kalg} and \texorpdfstring{$K^{top}$}{Ktop}}

We will start by an algebraic construction $\F \otimes_{\Z} \G$ which can be performed for any DG
systems $\F$ and $\G$.

\begin{defi}
    Let $\F$ be a right $C_*(\Omega X)$-module and $\G$ be a right $C_*(\Omega Y)$-module. The tensor product $\F \otimes_{\Z} \G$ has a natural $C_*(\Omega X) \otimes C_*(\Omega Y)$-module structure $(\alpha \otimes \beta)\cdot (\gamma_X \otimes \gamma_Y) = 
    (-1)^{|\beta||\gamma_{X}|} (\alpha \cdot \gamma_X \otimes \beta \cdot  \gamma_Y)$. 

    Consider the Serre diagonal 
    $\Delta :  C_*(\Omega X \times \Omega Y)  \to C_*(\Omega X) \otimes C_*(\Omega Y)$, ie the cubic equivalent to the Alexander-Whitney diagonal for singular complexes (see \cite[Section 2.4]{KASA05} and \cite{Serre51}).

 This can be used to transfer the natural module
     structure of $\F \otimes \G$ over $C_*(\Omega X) \otimes C_*(\Omega Y)$ described
      above to a module structure over $C_*(\Omega (X \times Y)) \hookrightarrow
       C_*(\Omega X \times \Omega Y)$ : 
    $$(\alpha \otimes \beta) \cdot (\omega_X, \omega_Y) := (\alpha \otimes \beta) \cdot 
    \Delta(\omega_X, \omega_Y).$$
\end{defi}

On the other hand, if $F \hookrightarrow E \to X$ and $G \hookrightarrow E' \to Y$
are two fibrations with respective transitive functions $\Phi_X$ and $\Phi_Y$,
then the fibration $F \times G \hookrightarrow E \times E' \to X \times Y$ is
naturally endowed with the transitive lifting function $(\Phi_X, \Phi_Y)$.
Therefore, to study fibrations, if $\F = C_*(F)$ and $\G = C_*(G)$, it is natural
to consider the complex $C_*(X\times Y, C_*(F \times G))$ where the module structure
is defined by $(\Phi_X, \Phi_Y)$. The main goal of this paper is to define and study
the DG Chas-Sullivan product, and this heavily uses considerations on fibrations.
Therefore we will prefer the DG system  $C_*(F \times G)$ in this case so
that we can use morphisms of fibrations to study this product. 

\begin{defi}\label{def : module structure on F times G}
    Let $\mathcal{F} = C_*(F)$ and $\G = C_*(G)$ be the cubical complexes of
    some topological spaces $F$ and $G$, where $F$ has an associative topological module structure
    $F \times \Omega X \to F$ and $G$ has an associative topological module structure
    $G \times \Omega Y \to G$. Let $\pi_X : X \times Y \to X$ and $\pi_Y : X \times Y \to Y$
    be the canonical projections. For a loop $\gamma \in \Omega(X \times Y)$,
    we denote $\gamma_X = \pi_X(\gamma)$ and $\gamma_Y = \pi_Y(\gamma)$.\\

    For any topological space $A,B$, we use the notation $\EZ : C_*(A) \otimes C_*(B) \to C_*(A \times B)$ for the Eilenberg-Zilber map $\EZ(a \otimes b) = (a,b)$ on cubic complexes.
    
    Topologically, the right $\Omega(X \times Y)$-module $F \times G$ is defined by $$
    \begin{array}{ccccc}
        F \times G \times \Omega(X \times Y) & \to & F \times \Omega X \times G \times \Omega Y & \to & F \times G  \\
        (\alpha, \beta , \gamma) &\mapsto & (\alpha,\gamma_X, \beta, \gamma_Y) & \mapsto & (\alpha.\gamma_X, \beta.\gamma_Y).
    \end{array} $$

    This gives rise to a module structure 

    $$C_*(F \times G) \otimes C_*(\Omega(X \times Y)) \overset{\EZ}{\to} C_*(F \times G \times \Omega(X \times Y)) \to C_*(F \times \Omega X \times G \times \Omega Y) \to C_*(F \times G).$$
\end{defi}

\begin{rem}\label{rem : on the module structure over F times G}
\textbf{1.} Let $F$ and $G$ be fibers of fibrations $F \hookrightarrow E_X \to X$ and $G \hookrightarrow E_Y \to Y$ equipped with transitive lifting functions $$\Phi_X : E_X \ftimes{\pi}{\ev_0} \mathcal{P} X, \textup{ and } \Phi_Y : E_Y \ftimes{\pi}{\ev_0} \mathcal{P} Y.$$

The module structure described above corresponds to the one that arises from the natural transitive lifting function $$\deffct{\Phi_{X \times Y}}{ E_X \times E_Y \ftimes{(\pi_X,\pi_Y)}{\ev_0} \mathcal{P}( X \times Y)}{E_X \times E_Y}{(\alpha, \beta), (\gamma, \sigma)}{\left(\Phi_X(\alpha, \gamma), \Phi_Y(\beta, \sigma) \right)}$$ of the fibration $F \times G \hookrightarrow E_X \times E_Y \to X \times Y.$\\

\textbf{2.} The general definition of this module structure is better understood at the topological level. Indeed, we cannot really write explicit formulas for every chain in $ C_*(\Omega(X \times Y))$ due to the fact that the cubical chains on a Cartesian product need not have split parametrizations. Therefore, at the chain level, a sign would appear to keep track of the Kozsul signs corresponding to the switch of parameters.\\

However, if $\gamma = \EZ( \gamma_1 \otimes \gamma_2) = (\gamma_1, \gamma_2) \in C_*(\Omega (X \times Y))$, then for every $(\alpha, \beta) = \EZ (\alpha \otimes \beta) \in C_*(F \times G)$, 

\begin{equation}\label{eq : module structure when splitted}
    (\alpha, \beta).\gamma = (-1)^{|\beta||\gamma_1|} (\alpha.\gamma_1, \beta.\gamma_2).
\end{equation} 

In this case, we will say that $\gamma \in C_*(\Omega (X \times Y))$ and $(\alpha , \beta) \in C_*(F \times G)$ are \textbf{split}. Remark that if $\{m^X_{x,x'} \in C_{|x|-|x'|-1}(\Omega X)\}$ and $\{m^Y_{y,y'} \in C_{|y|-|y'|-1}(\Omega Y)\}$ are Barraud-Cornea cocycles respectively on $X$ and $Y$, the Künneth twisting cocycle $\{m^K_{(x,y),(x',y')} \in C_{|x|+|y|-|x'|-|y'|-1} (\Omega(X \times Y))\}$ on $X \times Y$  is a family of split chains.

\end{rem}

We now restate and prove Theorem C. 

\begin{thm}\label{thm : Künneth formula}
Let $\Xi_X$, $\Xi_Y$ be sets of DG Morse data on $X$ and $Y$ with respective Morse functions $f : X \to \R$ and $g : Y \to \R$.\\

i) The Künneth twisting cocycle $\{m^K_{z,w} \in C_{|z|-|w|-1}(\Omega X \times \Omega Y), \ z,w \in \Crit(f \oplus g)\}$ associated to $\Xi_X$ and $\Xi_Y$ computes the same homology as the Barraud-Cornea cocycle associated to the set of DG data $\Xi_{X \times Y}$ constructed in Section \ref{subsection : Morse data on a Cartesian product}. \\

ii) Let $\F$ and $\G$ be DG modules over $C_*(\Omega X)$ and  $C_*(\Omega Y)$. Then, $$\deffct{K^{alg}}{C_*(X,\Xi_X, \F) \otimes C_*(Y, \Xi_Y, \G)}{C_*(X \times Y, m^K_{z,w}, \F \otimes \G)}{(\alpha \otimes x) \otimes (\beta \otimes y)}{(-1)^{|\beta||x|}(\alpha \otimes \beta) \otimes (x,y)}$$ is an isomorphism of complexes.\\

iii) If $C_*(F)$ and $C_*(G)$ are complexes of cubical chains of topological spaces equipped with topological module structures respectively over $\Omega X$ and $\Omega Y$, then $$\deffct{K^{top}}{C_*(X, \Xi_X, C_*(F)) \otimes C_*(Y,\Xi_Y, C_*(G))}{C_*(X \times Y, m^K_{z,w}, C_*(F \times G))}{(\alpha \otimes x) \otimes (\beta \otimes y)}{(-1)^{|\beta||x|}(\alpha , \beta) \otimes (x,y)}$$ is a quasi-isomorphism of complexes.\\
\end{thm}

\begin{proof}

   i) We proved in Proposition \ref{prop : dg Kunneth m^0 and m^1 quasi-iso} that the complexes constructed with the Barraud-Cornea cocycle and the Künneth cocycle are chain homotopy equivalent.
   
   ii) The map $K^{alg}$ is clearly bijective. We now check that it is a morphism of complexes :

    \begin{align*}
        K \partial\left((\alpha \otimes x) \otimes (\beta \otimes y)\right) &= K\left[ \left( \left(\partial \alpha \otimes x + (-1)^{|\alpha|} \sum_{x'} \alpha \cdot m_{x,x'} \otimes x'\right) \otimes (\beta \otimes y) \right) \right.\\
        & + \left. (-1)^{|\alpha| + |x|}\left( (\alpha \otimes x) \otimes \left(\partial \beta \otimes y + (-1)^{|\beta|}  \sum_{y'} \beta \cdot m_{y,y'} \otimes y' \right) \right) \right]\\
        &= (-1)^{|\beta||x|} (\partial \alpha \otimes \beta) \otimes (x,y) + (-1)^{|\alpha| + |x||\beta|} (\alpha \otimes \partial \beta) \otimes (x,y) \\
        &+ \sum_{x'} (-1)^{|\alpha| + |\beta||x'|} (\alpha \cdot m_{x,x'} \otimes \beta) \otimes (x',y) \\\
        &+ \sum_{y'} (-1)^{|\alpha| + |\beta| + |\beta||x| + (|y| - |y'|)|x|} (\alpha \otimes \beta \cdot m_{y,y'}) \otimes (x,y').
    \end{align*}

    \begin{align*}
        \partial K\left((\alpha \otimes x) \otimes (\beta \otimes y) \right) &= (-1)^{|\beta||x|} (\partial \alpha \otimes \beta)\otimes (x,y) + (-1)^{|\beta||x|}(-1)^{|\alpha|}(\alpha \otimes \partial \beta) \otimes (x,y) \\
        &+  (-1)^{|\beta||x|}(-1)^{|\alpha| +|\beta|} \sum_{x',y'} (\alpha \otimes \beta) \cdot m^K_{(x,y),(x',y')} \otimes(x',y')\\
        &= (-1)^{|\beta||x|} (\partial \alpha \otimes \beta)\otimes (x,y) + (-1)^{|\beta||x|}(-1)^{|\alpha|}(\alpha \otimes \partial \beta) \otimes (x,y) \\
        &+ \sum_{x'}(-1)^{|\beta||x| + |\beta|(|x|-|x'|-1)}(-1)^{|\alpha| + |\beta|} (\alpha \cdot m_{x,x'} \otimes \beta)\otimes(x',y)\\
        &+ \sum_{y'}(-1)^{|\alpha| + |\beta| + |\beta||x| + |x|(|y|-|y'|)}(\alpha \otimes \beta \cdot m_{y,y'})\otimes (x,y')\\
        &= K \partial \left((\alpha \otimes x) \otimes (\beta \otimes y) \right).
    \end{align*}

    iii) The proof that $K^{top}$ is a morphism of complexes is exactly the same as for $K^{alg}$ since every chain in this computation is split and we can use equation \eqref{eq : module structure when splitted} for every multiplication. It will be proven in Corollary  \ref{cor : K top is quasi iso} that $K^{top}$ is a quasi-isomorphism.
\end{proof}

 In particular, this shows the following corollary.

\begin{cor}\label{cor : coeff C(F) otimes C(G) equals C(F times G)}
    If $C_*(F)$ and $C_*(G)$ are complexes of cubical chains of topological spaces equipped with topological module structures respectively over $\Omega X$ and $\Omega Y$, then $H_*(X \times Y, C_*(F \times G))$ and $ H_*(X\times Y, C_*(F) \otimes C_*(G))$ are both isomorphic to $H_*\left(C_*(X,\Xi_X,C_*(F)) \otimes C_*(Y,\Xi_Y,C_*(G))\right)$ and therefore there is an isomorphism

    $$H_*(X \times Y, C_*(F \times G)) \simeq  H_*(X\times Y, C_*(F) \otimes C_*(G)).$$
\end{cor}
\begin{flushright}
    $\blacksquare$
\end{flushright}

By coupling Theorem \ref{thm : Künneth formula} with the general Künneth short exact sequence, we obtain a Künneth formula for enriched Morse complexes.

\begin{cor}
   There is a short exact sequence that splits,
    $$0 \to \bigoplus_{k+k'=l} H_k(X,\F) \otimes_{\Z} H_{k'}(Y,\G) \overset{K}{\to} H_l(X \times Y, \mathcal{F} \otimes_{\Z} \G) \to \bigoplus_{k+k'=l-1} \textup{Tor}^{\Z}_1( H_k(X,\F), H_{k'}(Y,\G)) \to 0.$$        
\end{cor}
\begin{flushright}
    $\blacksquare$
\end{flushright}

\subsubsection{Properties of \texorpdfstring{$K^{alg}$}{Kalg} and \texorpdfstring{$K^{top}$}{Ktop}}

We now prove that the Künneth maps are compatible with the direct and shriek maps (Lemma \ref{lemme : commutativité K avec direct et shriek}), with maps induced by a morphism of fibrations (Lemma \ref{lemme : induced maps on a Cartesian product commutes with K}), and with the Fibration Theorem (Lemma \ref{lemme : Compatibility K and Fibration Theorem}). We will also prove Corollary \ref{cor : K top is quasi iso} stating that $K^{top}$ is a quasi-isomorphism, and this will complete the proof of Theorem \ref{thm : Künneth formula} and its Corollary \ref{cor : coeff C(F) otimes C(G) equals C(F times G)}.

The next Lemmas and Corollary are stated and proved using $K^{alg}$ but also apply for $K^{top}$ if $\F_* = C_*(F)$, $\G_* = C_*(G)$ since the module structure will only be applied to split chains and we can therefore use equation \eqref{eq : module structure when splitted}. We will therefore use the notation $K$.

\begin{lemme}\label{lemme : commutativité K avec direct et shriek}
    Let $\varphi : Y^{n_Y} \to X^{n_X}$ and $\psi: Z^{n_Z} \to W^{n_W} $ be two continuous maps. Then the following diagram for the direct maps commutes at the chain level 

\[
\xymatrix{
C_*(Y, \varphi^*\mathcal{F}) \otimes C_*(Z, \psi^*\G) \ar[r]^-{K} \ar[d]_{\varphi_* \otimes \psi_*} & C_*(Y \times Z, \varphi^*\mathcal{F} \otimes \psi^*\G) \ar[d]^{(\varphi \times \psi)_*} \\
C_*(X, \mathcal{F}) \otimes C_*(W, \G) \ar[r]_{K} & C_*(X \times W, \mathcal{F} \otimes \G)
}
\]

and the following diagram for the shriek maps commutes at the chain level up to sign :

\[
\xymatrix{
C_{i+n_Y-n_X}(Y, \varphi^*\mathcal{F}) \otimes C_{j+n_Z - n_W}(Z, \psi^*\G) \ar[r]^-{K}  & C_{i+j +n_Y + n_Z -n_X -n_W}(Y \times Z, \varphi^*\mathcal{F} \otimes \psi^*\G)  \\
C_i(X , \mathcal{F}) \otimes C_j(W, \G) \ar[u]^{\varphi_! \otimes \psi_!} \ar[r]_{K} & C_{i+j}(X \times W, \mathcal{F} \otimes \G) \ar[u]_{(\varphi\times \psi)_!}.
}
\]

More precisely, $$K( \varphi_!(\alpha \otimes x) \otimes \psi_!(\beta \otimes w)) = (-1)^{(n_Y-n_X)(n_W  -|\beta|-|w|) + (n_Z-n_W)(|\alpha| + |x|)} (\varphi_! \times \psi_!)K(\alpha \otimes x \otimes \beta \otimes w).$$

\end{lemme}

\begin{proof}
    We prove the result for the shriek map, the same arguments apply for the direct map.

    Let $f_X,f_Y,f_Z,f_W$ be Morse functions respectively on $X,Y,Z,W$ and $x,y,z,w$ critical points.

    We start by computing $(\varphi \times \psi)_!$:
    
    \begin{align*}
        \mathcal{M}^{(\varphi \times \psi)_!}(x,w,y,z) &= \wb{s}{y,z} \cap (\varphi \times \psi)^{-1}(\wb{u}{x,w}) \\
        &=\wb{s}{y} \times \wb{s}{z} \cap \varphi^{-1}\wb{u}{x} \times \psi^{-1}\wb{u}{w}\\
        &= \mathcal{M}^{\varphi_!}(x,y) \times \mathcal{M}^{g_!}(w,z).
    \end{align*}

    Let's now compare orientations : 

    $$\left( \Or \ \wb{s}{x,w}, \Or \ \mathcal{M}^{(\varphi \times \psi)_!}(x,w,y,z) \right) = \Or \ \wb{s}{y,z}.$$

    Since $$\left( \Or \ \wb{s}{x,w}, \Or \ \wb{u}{x,w} \right) = \Or \ X \times W,$$
and
    \begin{align*}
        \left(\Or \ \wb{s}{x}, \Or \ \wb{s}{w}, \Or \ \wb{u}{x,w} \right) &= (-1)^{|x|(n_W -|w|)} \left(\Or \ \wb{s}{x}, \Or \ \wb{u}{x}, \Or \ \wb{s}{w}, \Or \ \wb{u}{w} \right)\\
        &= (-1)^{|x|(n_W -|w|)}  \Or \ X \times W,
    \end{align*}

    it follows that $$ \Or \ \wb{s}{x,w} =  (-1)^{|x|(n_W -|w|)} \left(\Or \ \wb{s}{x}, \Or \ \wb{s}{w} \right).$$ For the same reason, $$ \Or \ \wb{s}{y,z} =  (-1)^{|y|(n_Z -|z|)} \left(\Or \ \wb{s}{y}, \Or \ \wb{s}{z} \right).$$

We then compute
    \begin{align*}
        &\left(\Or \ \wb{s}{x,w}, \Or \ \mathcal{M}^{\varphi_!}(x,y), \Or \ \mathcal{M}^{\psi_!}(w,z) \right)\\
        &= (-1)^{|x|(n_W-|w|)} (-1)^{(n_W - |w|)(|x|-|y|+ n_Y -n_X)} \left(\Or \ \wb{s}{y}, \Or \ \wb{s}{z}\right)\\
        &= (-1)^{(n_W - |w|)(|y|+ n_Y -n_X)} (-1)^{|y|(n_Z-|z|)} \Or \ \wb{s}{y,z}.
    \end{align*}

    Therefore, if we choose $\sigma_{x,y}^{\varphi_!}$ and $\sigma_{w,z}^{\psi_!}$ two representing chain systems corrected by a sign (see \cite[section 10.4]{BDHO23}) of respectively $\mathcal{M}^{\varphi_!}(x,y)$ and $\mathcal{M}^{\psi_!}(w,z)$, then $$\sigma^{(\varphi \times \psi)_!}_{(x,w),(y,z)} = (-1)^{(n_W - |w|)(|y|+ n_Y -n_X) +1} (-1)^{|y|(n_Z-|z|)} (\sigma_{x,y}^{\varphi_!}, \sigma_{w,z}^{\psi_!}) $$ is a corrected representing chain system of $\mathcal{M}^{(\varphi \times \psi)_!}(x,w,y,z)$. Moreover, it is clear from the definition of the evaluation maps \cite[Section 10.4]{BDHO23} that for every $x \in \Crit(f_X), y \in \Crit(f_Y),w \in \Crit(f_W), z \in \Crit(f_Z)$, 
$q^{(\varphi \times \psi)_!}_{(x,w),(y,z)} =(q^{\varphi_!}_{x,y},q^{\psi_!}_{w,z}).$

Therefore $$\nu^{(\varphi \times \psi)_!}_{(x,w),(y,z)} = (-1)^{(n_W - |w|)(|y|+ n_Y -n_X)} (-1)^{|y|(n_Z-|z|)} (\nu_{x,y}^{\varphi_!},\nu_{w,z}^{\psi_!}) $$ where $\nu_{x,y}^{\varphi_!} := -q^{\varphi_!}_{x,y,*}(\sigma_{x,y}^{\varphi_!})$, $\nu_{w,z}^{\psi_!} := -q^{\psi_!}_{w,z,*}(\sigma_{z,w}^{\psi_!})$ and $ \nu^{(\varphi \times \psi)_!}_{(x,w),(y,z)} = -q^{(\varphi \times \psi)_!}_{(x,w),(y,z),*}(\sigma^{(\varphi \times \psi)_!}_{(x,w),(y,z)}). $

We can now compute 
    \begin{align*}
        (\varphi \times \psi)_!K(\alpha \otimes x \otimes \beta \otimes w) &= (-1)^{|\beta||x|} \sum_{y,z} (-1)^{(n_W - |w|)(|y|+ n_Y -n_X)} (-1)^{|y|(n_Z-|z|)} (\alpha \otimes \beta). (\nu_{x,y}^{\varphi_!}, \nu_{w,z}^{\psi_!}) \otimes (y,z) \\
        &= (-1)^{|y|(|w|+|z|+|\beta| + n_Z-n_W)}(-1)^{(n_Y-n_X)(n_W - |\beta|-|w|)} \sum_{y,z} (\alpha.\nu_{x,y}^{\varphi_!} \otimes \beta . \nu_{z,w}^{\psi_!}) \otimes (y,z).
    \end{align*}

    \begin{align*}
        K (\varphi_! \otimes \psi_!)(\alpha \otimes x \otimes \beta \otimes w) &= (-1)^{(n_Z -n_W)(|\alpha| + |x|)}  \sum_{y,z} K \left((\alpha \cdot \nu^{\varphi_!}_{x,y} \otimes y) \otimes (\beta \cdot \nu^{\psi_!}_{w,z} \otimes z) \right)\\
        &= (-1)^{(n_Z -n_W)(|\alpha| + |x|)} (-1)^{|y|(|w|+|z|+|\beta| + n_Z-n_W)} \sum_{y,z} (\alpha \cdot \nu_{x,y}^{\varphi_!} \otimes \beta \cdot  \nu_{z,w}^{\psi_!}) \otimes (y,z).
    \end{align*}
\end{proof}

\begin{cor}
    The maps $K^{alg}$ and $K^{top}$ are well-defined in homology.
\end{cor}

\begin{proof}
    Let $\Xi_0^X$ and $\Xi_1^X$ be Morse data sets on $X$ and $\Xi_0^Y$ and $\Xi_1^Y$ be Morse data sets on $Y$. Let $\F$ and $\G$ be DG systems on respectively $X$ and $Y$. 
    \cite[Proposition 10.2.1]{BDHO23} states that $\Id_{X,*} : C_*(X,\Xi_0^X,\F) \to C_*(X,\Xi_1^X, \F)$, $\Id_{Y,*} : C_*(Y,\Xi_0^Y,\G) \to C_*(Y,\Xi_1^Y, \G)$ and $\Id_{X\times Y,*} : C_*(X\times Y, \Xi_0^{X \times Y}, \F \otimes \G) \to C_*(X\times Y, \Xi_1^{X \times Y}, \F \otimes \G)$ are respectively chain homotopic to the continuation maps $$\Psi_{01}^X :C_*(X,\Xi_0^X,\F) \to C_*(X,\Xi_1^X, \F), $$ 
    $$ \Psi_{01}^{Y} : C_*(Y,\Xi_0^Y,\G) \to C_*(Y,\Xi_1^Y, \G),$$
    $$\Psi_{01}^{X \times Y} : C_*(X\times Y, \Xi_0^{X \times Y}, \F \otimes \G) \to C_*(X\times Y, \Xi_1^{X \times Y}, \F \otimes \G)$$
    
    of Theorem \ref{thm : continuation morphism Psi_01}. Since $K \circ (\Id_{X,*} \otimes \Id_{Y,*}) = \Id_{X \times Y,*} \circ K$, it follows that $$K : H_*(X, \F) \otimes H_*(Y,\G) \to H_*(X\times Y, \F \otimes \G)$$ is well-defined in the sense that the diagram

    $$\xymatrix{
    H_*(X, \Xi_0^X, \F) \otimes H_*(Y, \Xi_0^Y,\G) \ar[d]^{\Psi_{01}^X \otimes \Psi_{01}^Y} \ar[r]^{K}& H_*(X \times Y, \Xi^{X \times Y}_0, \F \otimes \G) \ar[d]_{\Psi_{01}^{X \times Y}} \\
    H_*(X, \Xi_1^X, \F) \otimes H_*(Y, \Xi_1^Y,\G)  \ar[r]^{K}& H_*(X \times Y, \Xi^{X \times Y}_1, \F \otimes \G)
    }$$
    commutes.
\end{proof}

\begin{lemme}\label{lemme : K is limit of morphism of spectral sequences}
    Let $\Xi_X$ and $\Xi_Y$ be sets of DG Morse data respectively on $X$ and $Y$.
    The maps $$K^{alg} : H_*\left(C_*(X,\Xi_X,\F)\otimes C_*(Y,\Xi_Y,\G)\right) \to
    H_*(X \times Y, \F \otimes \G)$$ and
    $$K^{top} : H_*\left(C_*(X,\Xi_X,C_*(F)\right)\otimes C_*(Y,\Xi_Y,C_*(G))) \to
    H_*(X \times Y, C_*(F\times \G))$$
    are limits of morphisms of spectral
    sequences $$K^{alg,(r)} : E^r_{s,t}((X,\F) \otimes (Y,\G))
    \to E^r_{s,t}(X \times Y, \F \otimes \G)$$ and
     $$K^{top,(r)} : E^r_{s,t}((X,F) \otimes (Y,G))
    \to E^r_{s,t}(X \times Y, F \times \G).$$
\end{lemme}

\begin{proof}

    Let $f: X \to \R$ and $g : Y \to \R$ be the respective Morse
    functions of $\Xi_X$ and $\Xi_Y$.
    Let $\Xi_{X \times Y}$ be the DG Morse data set on $X \times Y$ defined
    in Section \ref{subsection : Morse data on a Cartesian product}.

    Recall that the spectral sequence $E^r_{s,t}(X\times Y,\F \otimes \G)$
    is induced by the filtration 

    $$F_s(C_k(X\times Y,\Xi_{X\times Y},\F \otimes \G)) =
     \bigoplus_{\substack{i +j = k\\ i \leq s}} 
    (\F \otimes \G)_j \otimes \Z\Crit_i(f \oplus g).$$

    Define now define a filtration of the complex
    $C_*(X,\Xi_X,\F) \otimes C_*(Y,\Xi_Y,\G)$. Let $s,k \in \N$,

    $$F_s((C_*(X,\Xi_X,\F) \otimes C_*(Y,\Xi_Y,\G))_k) =
    \bigoplus_{\substack{p+p'+q+q'=k\\ p+p' \leq s}} \left(
     (\F_q \otimes \Z \Crit_p(f)) \otimes
     (\G_{q'} \otimes \Z \Crit_{p'}(g)) \right).$$

    This induces a spectral sequence $E^r_{p,q}((X,\F) \otimes (Y,\G))$
    converging towards $H_{p+q}(C_*(X,\Xi_X,\F) \otimes C_*(Y,\Xi_Y,\G))$
    whose first page is 

    $$E^1_{s,t}((X,\F) \otimes (Y,\G)) \simeq \bigoplus_{p+p' = s}
     H_{t}(\F \otimes \G) \otimes \Z\Crit_p(f) \otimes
    \Z \Crit_{p'}(g).$$

    Moreover, since 
    $$(\F \otimes \G)_j = \bigoplus_{q + q'=j} \F_q \otimes \G_{q'}
    \textup{ and } \Crit_i(f \oplus g) = \bigoplus_{p+p'=i} 
     \Crit_p(f) \times \Crit_{p'}(g),$$
    it follows that 
     $$K^{alg}(F_s((C_*(X,\Xi_X,\F) \otimes C_*(Y,\Xi_Y,\G))_k))
     = F_{s}(C_k(X\times Y, \Xi_{X\times Y}, \F \otimes \G)).$$

    The same approach applies for $K^{top}$ for
    $\F = C_*(F)$ and $\G = C_*(G)$.
    Indeed, since $$\EZ\left(\bigoplus_{q+q'=j} C_q(F) \otimes C_{q'}(G)\right)
    \subset C_j(F\times G),$$ 
    $$K^{top}\left(F_s((C_*(X,\Xi_X,C_*(F)) \otimes C_*(Y,\Xi_Y,\G))_k)\right)
     \subset F_{s}(C_k(X\times Y, \Xi_{X\times Y}, C_*(F \times G))).$$

\end{proof}

\begin{cor}\label{cor : K top is quasi iso}
    For any sets of DG Morse data $\Xi_X$ on $X$ and $\Xi_Y$ on $Y$, the map $K^{top} : C_*(X,\Xi_X,C_*(F))\otimes C_*(Y,\Xi_Y,C_*(G)) \to C_*(X \times Y, \Xi_{X \times Y}, C_*(F \times G))$ is a quasi-isomorphism.
\end{cor}

\begin{proof}
Denote $f : X \to \R$ the Morse function of $\Xi_X$ and $g : Y \to \R$ the Morse function of $\Xi_Y$.
The map $K^{top}$ induces an isomorphism on the first pages 

    $$K^{top,(1)} : H_{t}(C_*(F) \otimes C_*(G)) \otimes
    \left(\bigoplus_{p+p' = s} \Z\Crit_p(f) \otimes
   \Z \Crit_{p'}(g)\right) \to H_{t}(F\times G)
    \otimes \Z\Crit_{s}(f \oplus g).$$

Indeed, $$\EZ : H_{t}(C_*(F) \otimes C_*(G)) \to H_{t}(F\times G)$$
is an isomorphism.
\end{proof}

We finished to prove that $H_*(X\times Y, C_*(F) \otimes C_*(G)) \simeq H_*(X \times Y, C_*(F \times G))$. We will now denote $K = K^{top}$ for the two next lemmas that state compatibility properties of $K$ with the Fibration Theorem and with maps induced by morphisms of fibrations.

\begin{lemme}\label{lemme : Compatibility K and Fibration Theorem}
    Let $F \hookrightarrow E_X \to X$ and $G \hookrightarrow E_Y \to Y$ be two fibrations over two pointed oriented closed and connected manifolds equipped with transitive lifting functions $\Phi^X$ and $\Phi^Y$. Let $\Phi^{X \times Y} = (\Phi^X, \Phi^Y)$ be a transitive lifting function for the fibration $F \times G \hookrightarrow E_X \times E_Y \to X \times Y$. Then the following diagram commutes

    \[
    \xymatrix{
    H_*(X, C_*(F)) \otimes H_*(Y,C_*(G)) \ar[r]^-{K} \ar[d]^{\Psi_{E_X} \otimes \Psi_{E_Y}} & H_*(X \times Y, C_*(F \times G)) \ar[d]^{\Psi_{E_{X} \times E_{Y}}} \\
    H_*(E_X) \otimes H_*(E_Y) \ar[r]^-{\EZ} & H_*(E_X \times E_Y),
    }
    \]
    where $\Psi_{E_X}$, $\Psi_{E_Y}$ and $\Psi_{E_X \times E_Y}$ are the isomorphisms given by the Fibration Theorem using the transitive lifting functions $\Phi^X$, $\Phi^Y$ and $\Phi^{X \times Y}$ respectively.
    \end{lemme}

\begin{proof}
Let $f : X \to \R$ and $g : Y \to \R$ be Morse functions. Consider the Morse function $H  = f + g : X \times Y \to \R$.
Once again, it suffices to prove the statement for the pullback fibrations by the chosen homotopy inverses $\theta_X,\theta_Y$ and $\theta_{X \times Y}$  of the canonical projections $p_X : X \to X/\mathcal{Y}_X, \ p_Y : Y \to Y /\mathcal{Y}_Y$ and $p_{X \times Y} : X \times Y \to (X \times Y) /\mathcal{Y}$ (see Section \ref{subsection : Fibration Theorem} for explanation about the Fibration Theorem or \cite[Sections 7.2, 7.3]{BDHO23} for more details). We therefore only have to check that, for any $\alpha \in C_*(F), \ \beta \in C_*(G)$, $x \in \textup{Crit}(f)$ and $y \in \textup{Crit}(g)$, 
$$K(\Phi^X_*(\alpha \otimes m_x) \otimes \Phi^Y_*(\beta \otimes m_y)) = (-1)^{|\beta||x|} \Psi_{E_{X} \times E_Y}((\alpha,\beta) \otimes (x,y)).$$

   Since $\wb{u}{x,y} = \wb{u}{x} \times \wb{u}{y}$, we can choose the representing chain system $$m_{(x,y)}= (m_x, m_y) \in C_{|x|+|y|}(\mathcal{P}_{\star \to X \times Y/\mathcal{Y}}( X \times Y / \mathcal{Y}))$$ of the Latour cells in $X \times Y$. Moreover, $\Phi^{X \times Y}_*$ is given at the chain-level by $$\Phi_*^{X \times Y}((\alpha,\beta), (m_x, m_Y)) = (-1)^{|\beta||m_x|} (\Phi^X_*(\alpha, m_x), \Phi^Y_*(\beta, m_y)).$$ Since $|m_x|=|x|$, 
   $$ (\Phi^X_*(\alpha \otimes m_x), \Phi^Y_*(\beta \otimes m_y)) = (-1)^{|\beta||x|} \Phi^{X \times Y}_*((\alpha, \beta) \otimes m_{(x,y)}) = (-1)^{|\beta||x|} \Psi_{E_{X} \times E_Y}((\alpha,\beta) \otimes (x,y)),$$ and this concludes the proof.
\end{proof}

\begin{lemme}\label{lemme : induced maps on a Cartesian product commutes with K}
    Let $F_i \hookrightarrow E^X_i \to X$ and $G_i \hookrightarrow E^Y_i \to Y$ be fibrations over two pointed oriented closed and connected manifolds $(X,\star)$ and $(Y,\star_Y)$ with respective transitive lifting functions $\Phi^X_i$ and $\Phi^Y_i$ for $i \in \{1,2\}$.
    Let $\varphi : E^X_1 \to E^X_2$ and $\psi: E^Y_1 \to E^Y_2$ be morphisms of fibrations over $X$ and $Y$ respectively. 
    Endow the fibrations $E^X_i \times E^Y_i \to X \times Y$ with the transitive lifting functions $$\Phi_i^{X \times Y} = (\Phi_i^X,\Phi_i^Y) \ \textup{for } i\in \{1,2\}.$$
     Then $\varphi \times \psi : E^X_1 \times E^Y_1 \to  E^X_2 \times E^Y_2$ is a morphism of fibrations and the following diagram commutes
$$\xymatrix{
H_*(X,C_*(F_1)) \otimes H_*(Y, C_*(G_1)) \ar[d]^{\tilde{\varphi} \otimes \tilde{\psi}} \ar[r]^-K & H_*(X \times Y, C_*(F_1 \times G_1)) \ar[d]^{\widetilde{\varphi \times \psi}} \\
H_*(X,C_*(F_2)) \otimes H_*(Y, C_*(G_2)) \ar[r]^-{K} & H_*(X \times Y, C_*(F_2 \times G_2)).
}$$

\end{lemme}

\begin{proof}

Using the compatibility between $K$ and the Fibration Theorem (Lemma \ref{lemme : Compatibility K and Fibration Theorem}) and the compatibility between induced maps and the Fibration Theorem (Theorem \ref{thm : morphisme induit commute avec iso}), we obtain the commutativity of the diagram 

$$\xymatrix{
H_*(X,C_*(F_1)) \otimes H_*(Y,C_*(G_1)) \ar[r]^-K \ar@/_5pc/[ddd]_{\tilde{\varphi} \otimes \tilde{\psi}} \ar[d]^{\Psi_{E_1^X} \otimes \Psi_{E^Y_1}} & H_*(X \times Y, C_*(F_1 \times G_1)) \ar[d]_{\Psi_{E^X_1 \times E_1^Y}} \ar@/^5pc/[ddd]^{\widetilde{\varphi \times \psi}}\\
H_*(E_1^X) \otimes H_*(E^Y_1) \ar[r]^-{\EZ} \ar[d]^{\varphi_* \otimes \psi_*} & H_*(E_1^X \times E^Y_1) \ar[d]_{(\varphi \times \psi)_*}\\
H_*(E_2^X) \otimes H_*(E^Y_2) \ar[r]^-{\EZ} \ar[d]^{\Psi^{-1}_{E^X_2} \otimes \Psi^{-1}_{E_2^Y}} & H_*(E_2^X \times E_2^Y) \ar[d]_{\Psi^{-1}_{E_2^X \times E_2^Y}} \\
H_*(X,C_*(F_2)) \otimes H_*(Y, C_*(G_2)) \ar[r]^-{K} & H_*(X \times Y, C_*(F_2 \times G_2)),
}$$

where, for $i \in \{1,2\}$, $\Psi_{E^X_i}$, $\Psi_{E^Y_i}$ and $\Psi_{E^X_i \times E^Y_i}$ are the isomorphisms given by the Fibration Theorem using the transitive lifting functions $\Phi^X_i$, $\Phi^Y_i$ and $\Phi^{X \times Y}_i$ respectively.

\end{proof}

\section{Morse description and generalization of the Chas-Sullivan product. Proof of Theorem A}\label{section : Morse description and generalization of the Chas-Sullivan product}

In this section we will define a product on the homology of the total space $E$ of a fibration $F \hookrightarrow E \overset{\pi}{\to} X^n$ over a $n$-dimensional closed, connected, oriented and smooth manifold $X$, endowed with a morphism of fibrations $m : E \ftimes{\pi}{\pi} E \to E$ using DG Morse theory. For that, we will use : \begin{itemize}
    \item The cross-product $K := K^{top} : H_*(X,C_*(F))^{\otimes 2} \to H_*(X^2, C_*(F^2))$ defined in Theorem \ref{thm : Künneth formula}.
    \item The shriek map $\Delta_! : H_*(X^2, C_*(F^2)) \to H_{*-n}(X, \Delta^*C_*(F^2))$ of the diagonal $\Delta : X \to X^2$. As proved by Theorem \ref{thm : DG Thom iso}, this map is the equivalent in our model to the Pontryagin-Thom construction used in \cite{GS07}.
    \item The morphism $\tilde{m} : H_*(X,\Delta^*C_*(F^2)) \to H_*(X,C_*(F))$ induced by the morphism of fibrations $m : E \ftimes{\pi}{\pi} E \to E$. We proved in Theorem \ref{thm : morphisme induit commute avec iso} that this morphism is the equivalent in our model to the map $m_* : H_*(E \ftimes{\pi}{\pi} E) \to H_*(E).$
\end{itemize}

\subsection{Definition}
Let $E \overset{\pi}{\to} X$ a fibration over a closed, connected, oriented and smooth manifold $X$ and let $\star\in X$ a preferred point. Denote $F = \pi^{-1}(\star).$
Let $\Phi : E \ftimes{\pi}{\ev_0} \mathcal{P} X \to E$ be a transitive lifting function and consider $\F = C_*(F)$, the cubical complex of the fiber $F$ endowed with the DG module structure over $C_*(\Omega X)$ induced by $\Phi$. Suppose that there exists a morphism of fibrations $m : E \ftimes{\pi}{\pi} E \to E$ (see Definition \ref{defi : Ai morphism of fibration}).

We restate Theorem A and we will prove it in the next section.

\begin{thm}\label{thm : A}

    The  morphism of fibrations $m : E \ftimes{\pi}{\pi} E \to E$ induces a degree $-n$ product

$$\CS : H_*(X, \F)^{\otimes 2} \to H_*(X, \F).$$

such that the following properties hold :
    \begin{itemize}
    \item \textbf{\textup{Associativity :}}  If $m_*$ is associative in homology, then so is $\CS$.
    \item \textbf{\textup{Commutativity :}}  If $m_*$ is commutative in homology, then $\CS$ is commutative up to sign $$\CS(\gamma \otimes \tau) = (-1)^{(n-|\gamma|)(n-|\tau|)} \CS(\tau \otimes \gamma).$$
    \item \textbf{\textup{Neutral element :}}  If $\pi$ admits a section $s : X \to E$ such that $m(s(\pi(e)), e) = m(e, s(\pi(e))) = e$ for all $e \in E$, then $\CS$ admits a neutral element.
    \item \textbf{\textup{Functoriality :}}\\ $\bullet$ For any pointed, oriented, closed and connected manifold $Y^k$, any continuous map $g : Y \to X$ induces a degree $-k$ product for the fibration $F \hookrightarrow g^*E \to Y$

    $$\CS^Y : H_i(Y,g^*\F)\otimes H_j(Y, g^*\F) \to H_{i+j-k}(Y,g^*\F),$$ such that $g_! : H_*(X,\F) \to H_{*+n-k}(Y, g^*\F)$ is a morphism of rings up to sign.

     $\bullet$ If $g$ is an orientation-preserving homotopy equivalence then, $g_!$ and
     $g_* : H_*(Y,g^*\F) \to H_{*}(X, \F)$ are  isomorphisms of rings inverse of each other.
     \item \textbf{\textup{Spectral sequence :}} Let $\Xi$ be a set of DG Morse data on $X$. The canonical filtration
    $$F_p(C_*(X,\Xi, \F)) = \bigoplus_{\substack{i +j = k\\ i \leq p}} \F_j \otimes \Z\Crit_i(f)$$
    induces a spectral sequence $E^r_{p,q}$ that is endowed with an algebra structure
    $$E^r_{p,q} \otimes E^r_{l,m} \to E^r_{p+l-n, q+m}$$ induced by a chain-level model $\CS : C_*(X,\Xi,\F)^{\otimes 2} \to C_*(X,\Xi,\F)$ and converges towards $H_*(X, \F)$ as algebras.  For $s,t \geq 0$ $E^2_{s,t} = H_s(X,H_t(\F))$ and the algebra structure is given, up to sign, by the intersection product on $X$ with coefficient in $H_t(\F)$.\\
    
    \end{itemize}

This product corresponds in homology, via the Fibration Theorem, to the product $\mu_* : H_i(E) \otimes H_j(E) \to H_{i+j-n}(E)$ defined in \cite{GS07} and therefore (re)proves that the product $\mu_*$ satisfies those properties. In particular, if the considered fibration is the loop-loop fibration $\Omega X \hookrightarrow \mathcal{L}X \to X$ and the morphism of fibration $m : \ls{X} \ftimes{\ev}{\ev} \ls{X} \to \ls{X}$ is the concatenation, then $\CS$ corresponds to the Chas-Sullivan product.

\end{thm}

We defined in Section \ref{subsection : Künneth map} a $C_*(\Omega X^2)$-module structure on $\F^2:=C_*(F^2)$, which is in this case the module structure defined by the natural holonomy of the fibration $F^2 \hookrightarrow E^2 \to X^2$. We also denote, when the context is clear, $\F^2 = \Delta^* \F^2$ the $C_*(\Omega X)$-module structure on $\F^2$ obtained by pulling back by the diagonal $\Delta : X \to X^2$. It corresponds to the natural holonomy of the pullback fibration $F^2 \hookrightarrow E \ftimes{\pi}{\pi} E \to X$. The morphism of fibrations $m : E \ftimes{\pi}{\pi} E \to E$ induces a morphism of complexes $\Tilde{m} : C_*(X, \Xi, \F^2) \to C_*(X, \Xi, \F)$ that is compatible in homology with the Fibration Theorem (see Theorem \ref{thm : morphisme induit commute avec iso}).

The product $$\CS : H_*(X, \F)^{\otimes 2} \to H_*(X, \F)$$ is defined in homology by the composition
\[
\xymatrix
@C=10pt
{
H_i(X, \F) \otimes H_j(X, \F) \ar[r]^-{K} &  H_{i+j}(X \times X , \F^2) \ar[r]^-{\Delta_!} &
H_{i+j-n}(X, \Delta^*\F^2) \ar[r]^-{\Tilde{m}}  & H_{i+j-n}(X, \F)
}
\]

with the Dold sign. In other words,

$$\deffct{\CS}{H_i(X, \F) \otimes H_j(X, \F)}{H_{i+j-n}(X, \F)}{\gamma \otimes \tau}{(-1)^{n(n-|\tau|)}\Tilde{m}\circ \Delta_! \circ K(\gamma \otimes \tau).} $$

\subsection{Properties of the product}

\subsubsection{Associativity}

\begin{prop}\label{prop : associativité CSDG}

Suppose that $m : E \ftimes{\pi}{\pi} E \to E$ is associative in homology, ie the diagram

$$\xymatrix{
 & H_*(E \ftimes{\pi}{\pi} E \ftimes{\pi}{\pi} E) \ar[dl]^{(\Id \times m)_*} \ar[dr]_{(m \times \Id)_*} & \\
H_*( E \ftimes{\pi}{\pi} E ) \ar[dr]^{m_*} & & H_*( E \ftimes{\pi}{\pi} E) \ar[dl]_{m_*} \\
 & H_*(E) &
}$$

is commutative. Then, the product $\CS$ is associative in homology.
\end{prop}

\begin{proof} 
Let $\sigma, \omega, \delta \in H_*(X,\F).$ Recall that if $\varphi : Y \to Z$ is a continuous function and $\G$ is a DG-module over $C_*(\Omega Z)$, then $\varphi^*\G$ is $\G$ endowed with the following DG-module structure over $C_*(\Omega Y)$ 

$$\forall \alpha \in \G, \ \forall \gamma \in C_*(\Omega Y), \ \alpha \cdot \gamma := \alpha \cdot \underbrace{\varphi_*(\gamma)}_{\in C_*(\Omega Z)}.$$

Consider now the diagram 
\[
\xymatrix@C=0pt{
 & & H_*(X , \F)^{\otimes 3}\ar[dl]_{\textup{Id} \otimes K} \ar[dr]^{K \otimes \textup{Id}} & & \\
 & H_*(X , \F) \otimes H_{*}(X^2, \F^2) \ar[dr]_{K} \ar[dl]^{\textup{Id} \otimes \Delta_!} & 1. & H_{*}(X^2, \F^2) \otimes H_*(X , \F) \ar[dl]^{K} \ar[dr]_{\Delta_! \otimes \textup{Id}} & 
\\
 H_*(X,\F) \otimes H_*(X,\F^2) \ar[dr]^K \ar[d]_{\textup{Id} \otimes \Tilde{m}} & 2. & H_*(X^3,  \F^3) \ar[dl]_{(\textup{Id} \times \Delta)_!} \ar[dr]^{(\Delta \times \textup{Id})_!} & 2'. & H_*(X,\F^2) \otimes H_*(X,\F) \ar[d]^{ \Tilde{m} \otimes \textup{Id}} \ar[dl]_{K} \\
 H_*(X,\F)^{\otimes 2} \ar[d]_{K} & H_{*}(X^2, (\textup{Id} \times \Delta)^*\F^3) \ar[dr]_{\Delta_!} \ar[dl]^{\widetilde{\textup{Id} \times m}} & 4. & H_{*}(X^2, (\Delta \times \textup{Id})^*\F^3) \ar[dl]^{\Delta_!} \ar[dr]_{\widetilde{m\times \textup{Id}}} & H_*(X,\F)^{\otimes 2} \ar[d]^{K}\\
 H_*(X^2,\F^2) \ar[dr]^{\Delta_!} & 5. & H_*(X,\F^3) \ar[dl]_{\Delta^*(\widetilde{\textup{Id} \times m})} \ar[dr]^{\Delta^*(\widetilde{m \times \textup{Id}})} & 5. & H_*(X^2, \F^2)\ar[dl]_{\Delta_!}\\
 & H_*(X,\F^2) \ar[dr]_{\Tilde{m}} & 6. & H_*(X, \F^2)\ar[dl]^{\Tilde{m}} &\\
 & & H_*(X,\F). & & 
}
\]

The associativity of $\CS$ is equivalent to the fact that this diagram commutes up to the Kozsul sign $(-1)^{n(|\sigma|)}$ and the Dold sign $(-1)^{n(n-|\delta|)}$. Indeed, we can compute

\begin{align*}
    &\CS(\sigma \otimes \CS(\omega \otimes \delta)) \\
    &=  (-1)^{n(n-(|\omega|+|\delta|-n))} (-1)^{n(n-|\delta|)} \Tilde{m} \circ \Delta_! \circ K\left( \sigma \otimes \left( \Tilde{m} \circ \Delta_! \circ K(\omega \otimes \delta)\right) \right)\\
    &= (-1)^{n|\sigma|} (-1)^{n(n-|\omega|)} \Tilde{m} \circ \Delta_! \circ K \circ (\textup{Id} \otimes \Tilde{m}) \circ (\textup{Id} \otimes \Delta_!) \circ (\textup{Id} \otimes K)(\sigma \otimes \omega \otimes \delta)
\end{align*}

and, 

\begin{align*}
    &\CS(\CS(\sigma \otimes \omega) \otimes \delta) \\
    &= (-1)^{n(n-|\omega|)} (-1)^{n(n-|\delta|)} \Tilde{m} \circ \Delta_! \circ K\left(\left( \Tilde{m} \circ \Delta_! \circ K(\sigma \otimes\omega) \right) \otimes \delta\right)\\
    &=  (-1)^{n(|\omega|+|\delta|)} \Tilde{m} \circ \Delta_! \circ K \circ (\Tilde{m} \otimes \textup{Id}) \circ  (\Delta_! \otimes \textup{Id}) \circ (K\otimes \textup{Id}) (\sigma \otimes \omega \otimes \delta).
\end{align*}

We then prove the commutativity of this diagram up to the wanted sign in 6 steps :

$\bullet$ \textit{\textbf{Step 1 : The square 1. is commutative.}}

This is an easy computation at the chain level. Let $\Xi_1,\Xi_2,\Xi_3$ be sets of DG Morse data on $X$. Let $\alpha \otimes x \in C_*(X,\Xi_1,\F)$, $\beta \otimes y \in C_*(X,\Xi_2, \F)$ and  $\epsilon \otimes z \in C_*(X,\Xi_3,\F).$ \begin{align*}
    K \circ \left( \Id \otimes K\right) ((\alpha \otimes x) \otimes (\beta \otimes y) \otimes (\gamma \otimes z)) &= (-1)^{|\gamma||y|} K\left( \alpha \otimes x \otimes ((\beta, \gamma) \otimes (y,z)) \right)\\
    &= (-1)^{(|\beta|+|\gamma|)|x| + |\gamma||y|} (\alpha, \beta, \gamma) \otimes (x,y,z).
\end{align*}

\begin{align*}
    K \circ \left( K \otimes \Id \right) ((\alpha \otimes x) \otimes (\beta \otimes y) \otimes (\gamma \otimes z)) &= (-1)^{|\beta||x|} K\left( (\alpha, \beta) \otimes (x,y)) \otimes \gamma \otimes z \right)\\
    &= (-1)^{|\beta||x| + |\gamma|(|x|+|y|)} (\alpha,\beta,\gamma) \otimes (x,y,z)\\
    &= K \circ \left( \Id \otimes K\right) ((\alpha \otimes x) \otimes (\beta \otimes y) \otimes (\gamma \otimes z))
\end{align*}

$\bullet$ \textit{\textbf{Step 2 : The square 2'. is commutative up to the Dold sign and 2. is commutative up to the Kozsul sign.}}

This is a direct application of Lemma \ref{lemme : commutativité K avec direct et shriek}.\\

$\bullet$ \textit{\textbf{Step 3 : The rightmost and leftmost triangles are commutative.}}

This is a direct application of Lemma \ref{lemme : induced maps on a Cartesian product commutes with K} since $\Id : E \to E$ is a morphism of fibrations and $\widetilde{\Id} = \Id : C_*(X, \F) \to C_*(X,\F).$\\

$\bullet$ \textit{\textbf{Step 4 : The square 4. is commutative.}}

This directly follows from the composition property of shriek maps \cite[Theorem 8.2]{BDHO23} since $$(\Delta \times \textup{Id})\Delta = (\textup{Id} \times \Delta)\Delta.$$

$\bullet$ \textit{\textbf{Step 5 : The squares 5. are commutative.}}

This is a particular case of Proposition \ref{prop : morphisme Ai commute avec direct et shriek}.\\

$\bullet$ \textit{\textbf{Step 6 : The square 6. is commutative.}}

 We proved in Lemma \ref{lemme : morphism induit poussé par une application continue} that $\Delta^*(m\times \textup{Id})$ and $\Delta^*(\textup{Id} \times m)$ are morphisms of fibrations. This step is therefore a consequence of Corollary \ref{cor : composition of morphism of fibrations} and Corollary \ref{cor : equality of Ai morphism on total space gives equality of induced maps}. Indeed, since we assumed $m(m\times \textup{Id})_* = m(\textup{Id} \times m)_* : H_*( E \ftimes{\pi}{\pi} E \ftimes{\pi}{\pi} E) \to H_*(E)$, \begin{align*}
     \tilde{m} \circ \Delta^* \widetilde{m \times \Id} &= \widetilde{m (m\times \Id)}\\
     &= \widetilde{m(\Id \times m)}\\
     &= \tilde{m} \circ \Delta^* \widetilde{\Id \times m}.
 \end{align*}

\end{proof}

\subsubsection{Commutativity}

We will use the following notations for the switches of coordinates :

\begin{itemize}
    \item[$\bullet$] $\tau : X^2 \to X^2$.\\
    \item[$\bullet$] Consider $\tau^*F^2 \to \tau^*E^2 \to X^2$ the pullback fibration by $\tau$ whose natural transitive lifting function is $$\tau^*\Phi^2((\alpha, \beta)\cdot(\gamma_1,\gamma_2)) = (\Phi(\alpha, \gamma_2), \Phi(\beta, \gamma_1))$$ for all $(\alpha, \beta) \in F^2$ and $(\gamma_1, \gamma_2) \in \mathcal{P}_{\star \to X}X$. The switch of coordinates $\tau_E : \tau^*E^2 \to E^2$ is an isomorphism of fibrations (bijective morphism of fibrations) that commutes with the transitive lifting functions $\tau^*\Phi_2$ and $$\Phi_2=(\Phi,\Phi) : E^2 \ftimes{(\pi,\pi)}{\ev_0} \mathcal{P}X^2 \to E^2$$ Therefore, $$\tilde{\tau_E} : C_*(X^2,\tau^*C_*(F^2)) \to C_*(X^2,C_*(F^2)), \quad  \tilde{\tau_E}((\alpha, \beta) \otimes (x,y)) = \tau_{E,*}(\alpha, \beta) \otimes (x,y).$$  If $(\alpha, \beta) = \EZ(\alpha \otimes \beta)$, then $\tilde{\tau_E}((\alpha, \beta) \otimes (x,y)) = (-1)^{|\alpha||\beta|} (\beta, \alpha) \otimes (x,y).$\\
    \item[$\bullet$] $\tau_F : E \ftimes{\pi}{\pi} E \to E \ftimes{\pi}{\pi} E$ is also an isomorphism of fibrations and $$\tilde{\tau_F} = \Delta^*\tilde{\tau_E} : C_*(X, \Delta^*C_*(F^2)) \to C_*(X, \Delta^*C_*(F^2)).$$
\end{itemize}

\begin{prop}\label{prop : commutativité CSDG}
     We suppose that $m$ is commutative in homology, ie $m_*\tau_{E,*} = m_* : H_*(E \ftimes{\pi}{\pi} E) \to H_*(E)$.
     Then, for any $\sigma,\omega \in H_*(X,\F)$,  
    $$\CS(\sigma, \omega) = (-1)^{(n-|\sigma|)(n-|\omega|)} \CS(\omega, \sigma).$$ 
\end{prop}

In order to prove this property, we have to understand how to switch variables on the manifold $X^2$. Let $\Xi$ be a set of Morse data on $X$.

Let $\Xi' = (f',\xi', s'_{x',y'}, o',\mathcal{Y}', \theta')$ be another Morse DG data set on $X$ and define $\Xi_1$, $\Xi_2$ to be the Morse DG data sets on $X^2$ described in Section \ref{subsection : Morse data on a Cartesian product} with $Y=X$ and respectively $(\Xi_X,\Xi_Y) =(\Xi, \Xi')$ and $(\Xi_X,\Xi_Y) = (\Xi',\Xi)$.

\begin{lemme}
    The shriek map $\tau_! : C_*(X^2, \Xi_1, C_*(F^2)) \to C_*(X^2, \Xi_2, \tau^*C_*(F^2))$ is given up to chain homotopy by $$\tau_!((\alpha,\beta) \otimes (x,x')) = (-1)^{|x||x'|+n} (\alpha, \beta) \otimes (x',x). $$ 
\end{lemme}

\begin{proof}

Let us recall how the direct and shriek maps of a diffeomorphism $\varphi : Y \to Z$ have been defined in \cite[Section 9.1]{BDHO23}.

Given $\Xi = (f, \xi, \mathcal{Y}, o, s_{x,y}, \theta)$, define $\varphi^* \Xi = ( f \circ \varphi^{-1}, \varphi_*\xi, \varphi\mathcal{Y}, \varphi(o), \varphi^*s_{\varphi(x),\varphi(y)}, \varphi\circ\theta)$ where the orientations $\varphi(o)$ of the unstable manifolds are defined for any $x,y \in \Crit(f)$ by

    $$\Or \ \overline{W^{u}}(\varphi(x)) := \Or \ \wb{u}{x},$$

    and the representing chain system is  $$\varphi^*s_{\varphi(x),\varphi(y)} = \varphi_*(s_{x,y}) \in C_{|x|-|y|-1}\left( \trajb{\varphi(x), \varphi(y)} \right).$$

Then, the direct map and shriek map are defined for any DG system $\G$ on $Z$ by $$\varphi_* : C_*(Y, \Xi, \varphi^*\G) \to C_*(Z, \varphi^*\Xi, \G), \ \varphi_*(\alpha \otimes x) = \alpha \otimes \varphi(x)$$

and $$\varphi_! : C_*(Z, \varphi^*\Xi, \G) \to C_*(Y, \Xi, \varphi^*\G), \ \varphi_!(\alpha \otimes x) = (-1)^{deg \ \varphi} \ \alpha \otimes \varphi^{-1}(x).$$

Therefore $$\tau_! : C_*(X^2, \tau^*\Xi_2, C_*(F^2)) \to C_*(X^2, \Xi_2, \tau^*C_*(F^2)), \ \tau_!((\alpha, \beta) \otimes (x,x')) = (-1)^n (\alpha, \beta) \otimes (x',x).$$

However, $\Xi_1 \neq \tau^*\Xi_2$. Indeed, \begin{align*}
    o_1(x,x') &:= \Or \  \overline{W}^{u}_{f,f'}(x,x') = \left( \Or \ \overline{W}^u_{f}(x), \Or \ \overline{W}^u_{f'}(x') \right)\\
    &= (-1)^{|x||x'|} \left( \Or \ \overline{W}^u_{f'}(x'), \Or \ \overline{W}^u_{f}(x) \right) = (-1)^{|x||x'|} \tau^* o_2(x,x').
\end{align*} 

The other data are the same. A corollary of the proof of the first step of \cite[Theorem 6.3.1]{BDHO23} shows that the continuation map $\Psi : C_*(X^2,\Xi_1, C_*(F^2)) \to C_*(X^2, \tau^*\Xi_2, C_*(F^2))$ is $$\Psi((\alpha \otimes \beta) \otimes (x,x')) = (-1)^{|x||x'|} (\alpha \otimes \beta) \otimes (x,x').$$

Therefore, the shriek map  $\tau_! : C_*(X^2, \Xi_1, \F^2) \to C_*(X^2, \Xi_2, \tau^*\F^2)$ is given up to chain homotopy by $$\tau_!((\alpha,\beta) \otimes (x,x')) = (-1)^{|x||x'|+n} (\alpha, \beta) \otimes (x',x). $$ 

\end{proof}

\begin{myproof}{of Proposition}{\ref{prop : commutativité CSDG}} We first remark that $\tau \circ \Delta = \Delta$ and therefore the composition property for the shriek map gives $\Delta_! = \Delta_! \circ \tau_!$ up to chain homotopy.\\

We compute \begin{align*}
    \CS( \alpha \otimes x, \beta \otimes x') &= (-1)^{n(n-|\beta|-|x'|)}(-1)^{|\beta||x|} \Tilde{m} \circ \Delta_!( (\alpha,\beta) \otimes (x,x') ) \\
    &= (-1)^{|\beta|(|\alpha|+|x|+n) + (n+|x|)|x'|} \Tilde{m} \circ \Delta_! \circ \tilde{\tau_E} \circ \tau_! ( (\beta, \alpha) \otimes (x',x))\\
    &= (-1)^{|\beta|(|\alpha|+|x|+n) + (n+|x|)|x'|}  \Tilde{m}\circ \tilde{\tau_F} \circ \Delta_! ((\beta, \alpha) \otimes (x',x))\\
    &= (-1)^{|\beta|(|\alpha|+|x|+n) + (n+|x|)|x'|}  \Tilde{m} \circ \Delta_! ((\beta, \alpha) \otimes (x',x))\\
    &= (-1)^{(n-|\alpha|-|x|)(n-|\beta| -|x'|)} \CS(  \beta \otimes x' \otimes \ \alpha \otimes x). 
    \end{align*}

The equality $\Delta_! \circ \tilde{\tau_E} = \Delta^*\tilde{\tau} \circ \Delta_! = \tilde{\tau_F} \circ \Delta_! $ is a direct consequence of Proposition \ref{prop : morphisme Ai commute avec direct et shriek}.  The equality $\Tilde{m} \circ \tilde{\tau_F} = \Tilde{m}$ is a direct consequence of Corollary \ref{cor : composition of morphism of fibrations} and Corollary \ref{cor : equality of Ai morphism on total space gives equality of induced maps}.

\end{myproof}

\subsubsection{Neutral element}

\begin{prop}\label{prop : Neutral element}
     Suppose that the fibration $F \hookrightarrow E \to X$ admits a section $s: X \to E$ such that

     $$m(s(\pi(e)), e) = m(e, s(\pi(e))) = e $$ for any $e \in E.$
     Then, $\tilde{s}([X]) \in H_n(X, \F)$ is a neutral element for $\CS$.
\end{prop}

\begin{proof}

A section $s : X \to E$ is a morphism of fibrations when one considers $X$ as the total space of the trivial fibration $ \star \hookrightarrow X \overset{\Id}{\to} X$. Denote $\F = C_*(F)$ and $\F^2 = C_*(F^2).$
    
Consider the following diagram 
    $$\xymatrix{H_*(X,\F) \otimes H_*(X,\Z) \ar[r]^-K \ar[d]^{\textup{Id} \otimes \tilde{s}} & H_*(X^2, C_*(F \times \{\star\})) \ar[r]^-{\Delta_!} \ar[d]^{\widetilde{\Id \times s}} & H_*(X, C_*(F \times \{\star\})) \ar[d]^{\Delta^* \widetilde{\Id \times s}} \ar[dr]^{\tilde{\sigma}} &  \\
    H_*(X,\F) \otimes H_*(X,\F) \ar[r]^-K  & H_*(X^2, \F^2)) \ar[r]^-{\Delta_!} & H_*(X, \Delta^*\F^2)) \ar[r]^{\tilde{m}} & H_*(X, \F),}$$

    where $\sigma : E \ftimes{\pi}{\Id} X \to E$, $(e,\pi(e)) \mapsto e$. Since $m \circ (\Id \times s) = \sigma$, Corollary \ref{cor : composition of morphism of fibrations} proves that the last triangle commutes.  The first diagram commutes by Lemma \ref{lemme : induced maps on a Cartesian product commutes with K}. The second diagram commutes by Proposition \ref{prop : morphisme Ai commute avec direct et shriek}.

    It remains to prove that $\tilde{\sigma} \circ \Delta_! \circ K(\gamma \otimes [X]) = \gamma$ for any $\gamma \in H_*(X,\F).$
    Let $\Xi$ be a set of DG Morse data on $X$. We can and will assume that the Morse function $f : X \to \R$ has only one local maximum $x_{max}$. In this case, $[(\star \otimes x_{\max})] = [X] \in H_n(X,\Z).$ 

    Let $\alpha \otimes x \in C_*(X,\Xi,\F)$.
    We start by computing $\Delta_!\left((\alpha, \star) \otimes (x,x_{max})\right)$ :
    \begin{align*}
        \overline{\mathcal{M}}^{\Delta_!}((x,x_{\max}),y)
        &= \overline{W}^s(y) \cap \Delta^{-1}(\overline{W}^u(x,x_{\max}))\\
        &= \overline{W}^s(y) \cap \overline{W}^u(x_{\max}) \cap \overline{W}^u(x) \\
        &= \overline{W}^s(y) \cap \overline{W}^u(x) = \overline{\mathcal{M}}^{\Id_!}(x,y).
    \end{align*}

 We now compare orientations.
  The orientation rule $(\Or \ \overline{W}^s(x_{\max}), \Or \ \overline{W}^u(x_{\max})) = \Or \ X$ gives $\Or \ \overline{W}^s(x_{\max}) = +$ and 

  $$\left(\Or \ \overline{W}^s(x,x_{\max}), \Or \ \overline{\mathcal{M}}^{\Delta_!}((x,x_{\max}),y)\right) = \Or \ \overline{W}^s(y) \Leftrightarrow \left( \Or \ \overline{W}^s(x), \Or \ \overline{\mathcal{M}}^{\Delta_!}((x,x_{\max}),y)\right) = \Or \ \overline{W}^s(y). $$ Therefore $\Or \ \overline{\mathcal{M}}^{\Delta_!}((x,x_{max}),y) = \Or \ \overline{\mathcal{M}}^{\Id_!}(x,y)$ and it follows that for any $$\tilde{\sigma} \circ\Delta_! \circ K(\alpha \otimes x \otimes \star \otimes x_{max}) = \sigma(\alpha, \star) \otimes x = \alpha \otimes x$$ up to chain homotopy. Indeed, $\sigma_* : C_*(F \times \{\star\}) \to C_*(F), \ (\alpha, \star) \mapsto \alpha$ is a morphism of DG-modules and therefore $\tilde{\sigma} : C_*(X,C_*(F \times \{\star\})) \to C_*(X,C_*(F))$, $(\alpha,\star) \otimes x \mapsto \alpha \otimes x.$ \\
\end{proof}

\begin{rem}
    The inclusion of constant loops $s : X \to \mathcal{L}X$ is a section of the loop-loop fibration $\Omega X \hookrightarrow \ls{X} \overset{\ev}{\to} X$ and it clearly satisfies the assumptions of Proposition \ref{prop : Neutral element}.
\end{rem}

\subsubsection{Functoriality}

Let $(Y^k, \star_Y)$ be a $k$-dimensional, smooth, pointed, oriented, closed, connected manifold.
 Let $g : Y \to X$ be a continuous map and $F \hookrightarrow E_Y \overset{\pi_Y}{\to} Y$ be the fibration obtained by pulling back by $g$ the fibration $F \hookrightarrow E \overset{\pi}{\to} X$. This fibration is endowed with the morphism of fibrations $$g^*m : E_Y \ftimes{\pi_Y}{\pi_Y} E_Y \to E_Y.$$

\begin{prop}\label{prop : Functoriality property}
    The shriek map $g_! : H_*(X, C_*(F)) \to H_{*-n+k}(Y, g^*C_*(F))$ is a morphism of rings up to sign.
    More precisely, up to chain homotopy, $$\CS^Y\left(g_!(\alpha \otimes x) \otimes g_!(\beta \otimes x')\right) = (-1)^{(k-n)(n + |\alpha| + |x|-|\beta| - |x'|)} g_!\left(\CS^X(\alpha \otimes x \otimes \beta \otimes x')\right).$$
   \end{prop}

\begin{proof}

We denote $l = k-n \in \Z$.

The proof amounts to show that the following diagram is commutative up to the wanted sign :

    \[
\xymatrix{
H_{i+l}(Y,g^*\F) \otimes H_{j+ l}(Y,g^*\F) \ar[r]^-K & H_{i+j + 2l}\left(Y^2 , (g^2)^*\F^2 \right) \ar[r]^-{\Delta_!} & H_{i+j-n+l}(Y, g^*\F^2) \ar[r]^-{g^*\Tilde{m}} & H_{i+j-n+l}(Y, g^*\F)  \\
H_i(X,\F) \otimes H_j(X,\F) \ar[r]_K \ar[u]^{g_! \otimes g_!} & H_{i+j}\left(X^2 , \F^2 \right) \ar[r]_{\Delta_!}\ar[u]^{(g \times g)_!} & H_{i+j-n}(X, \F^2)\ar[r]_{\Tilde{m}}\ar[u]^{g_!} & H_{i+j-n}(X, \F) \ar[u]^{g_!}
}
\]

The first square commutes up to the wanted sign by Lemma \ref{lemme : commutativité K avec direct et shriek}, the second square commutes by the composition property \cite[Theorem 8.1.1]{BDHO23} and the third square commutes according to Proposition \ref{prop : morphisme Ai commute avec direct et shriek}.

\end{proof}

\begin{rem}
    In particular, if $Y$ is $n$-dimensional, then $g_! : H_*(X,\F) \to H_*(Y, g^*\F)$ is a morphism of rings.
\end{rem}

\begin{cor}\label{cor : shriek map of an homotopy equivalence is ring iso with pullbacked module structure}
    If $g : Y \to X$ is an orientation-preserving homotopy equivalence between two manifolds of same dimension, then $g_! : H_*( X, C_*(F)) \to H_*(Y, g^*C_*(F))$ and $g_* : H_*(Y, g^*C_*(F)) \to H_*( X, C_*(F))$ are isomorphisms of rings.

\end{cor}

\begin{proof}
    Using \cite[Corollary 10.6.4]{BDHO23}, $g_!$ and $g_*$ are isomorphisms inverse to each other. Therefore $g_!$ is an isomorphism of rings and so is $g_* = (g_!)^{-1}$.
    
\end{proof}

\subsubsection{Spectral sequence, chain description of the product}

We now describe a chain-level model $\CS : C_i(X,\Xi, \F) \otimes C_j(X,\Xi,\F) \to C_{i+j-n}(X,\Xi,\F)$ for the product $\CS : H_i(X,\F) \otimes H_i(X,\F) \to H_{i+j-n}(X,\F)$.

We will prove that this preserves the canonical filtrations associated to these complexes and therefore prove that our construction also endows the canonical spectral sequence $E^r_{p,q}$ associated to an enriched Morse complex $C_*(X,\Xi_0, \F)$ with an algebra structure which converges towards $H_*(X,\F)$ as algebras.

This a DG Morse equivalent to \cite[Theorem 1]{CoJoYan} for the Chas-Sullivan product that has been generalized in \cite[Theorem 3.6]{GS07}.

\begin{prop}\label{prop : spectral sequence CSDG}
     Let $\Xi_0$ be a set of DG Morse data on $X$. The canonical filtration
    $$F_p(C_*(X,\Xi_0, \F)) = \bigoplus_{\substack{i +j = k\\ i \leq p}} \F_j \otimes \Z\Crit_i(f)$$
   induces a spectral sequence $E^r_{p,q}$ that is endowed with an algebra structure
    $$E^r_{p,q} \otimes E^r_{l,m} \to E^r_{p+l-n, q+m}$$
   and converges towards $H_*(X, \F)$ as algebras. For $s,t \geq 0$, $E^2_{s,t} = H_s(X,H_t(\F))$ and the algebra structure is given up to sign by the intersection product on $X$ with coefficients in $H_t(\F)$.
\end{prop}

\begin{proof}

We will use for $\Delta_!$ the second definition of the shriek maps given in \cite[Section 10.4]{BDHO23} and take $\Xi_1$ and $\Xi_2$ generic sets of Morse data on $X$ such that $$\Delta\lvert_{W^s_{f_2}(y)} \pitchfork W^u_{f_0 + f_1}(x,x')$$ for all $y \in \Crit(f_2)$, $x \in \Crit(f_0)$, $x' \in \Crit(f_1)$.

We define a chain-level product
$$\begin{array}{ccl}
   C_*(X, \Xi_0, \F) \otimes C_*(X, \Xi_0, \F)  & \overset{\Id \otimes \Psi_{01}}{\longrightarrow} &  C_*(X, \Xi_0, \F) \otimes C_*(X, \Xi_1, \F)  \\
    & \overset{K}{\longrightarrow} & C_*(X^2, \Xi_{01}, \F^2) \\
    & \overset{\Delta_!}{\longrightarrow} & C_*(X, \Xi_2, \Delta^*\F^2)\\
    & \overset{\Psi_{20}}{\longrightarrow} & C_*(X, \Xi_0, \Delta^*\F^2) \\
    & \overset{\tilde{m}}{\longrightarrow} & C_*(X, \Xi_0, \F).
\end{array}$$

Every map preserves the canonical fibrations associated to an enriched complex.
Indeed, \cite[Theorem 8.1.1]{BDHO23} states that all direct and shriek maps
preserve the filtrations. The continuation maps $\Psi_{01}$ and $\Psi_{20}$
are defined by equation \cite[equation (33)]{BDHO23} and clearly preserves filtrations.
We proved in Proposition \ref{prop : induced morphism is limit of spectral sequence maps}
that every morphism of complexes induced by a morphism of $\Ai$-modules preserves filtrations.

It only remains to prove that $K$ preserves filtrations.
It suffices to remark that if $\alpha \otimes x \in \F_j \otimes
     \Z\Crit_i(f)$ and $\beta \otimes y \in \G_l \otimes
      \Z\Crit_k(g)$, then
     $$K^{alg}((\alpha \otimes x) \otimes (\beta \otimes y)) =
      (-1)^{li} (\alpha \otimes \beta) \otimes(x,y) \in
      (\F \otimes \G)_{j+l} \otimes \Z\Crit_{i+k}(f \oplus g) $$
     and in particular $$K^{alg} (F_p(C_*(X,\Xi_X,\F)) \otimes
     F_q(C_*(Y,\Xi_Y,\G))) = F_{p+q}(C_*(X \times Y, \Xi_{X \times Y},
     \F \otimes \G)).$$ If $\F = C_*(F)$ and $\G= C_*(G)$, we also have
     $$K^{top}((\alpha \otimes x) \otimes (\beta \otimes y)) =
      (-1)^{li} (\alpha, \beta) \otimes(x,y) \in C_{j+l}(F \times G) \otimes \Z\Crit_{i+k}(f \oplus g)$$
     and therefore
     $$K^{top} (F_p(C_{i+j}(X,\Xi_X,C_*(F))) \otimes
     F_q(C_{k+l}(Y,\Xi_Y,C_*(G)))) \subset F_{p+q}(C_{i+j+k+l}(X \times Y, \Xi_{X \times Y},
     C_*(F \times G))).$$

It follows that for any $a,b\in \N$, $$\CS ( F_p(C_a(X, \Xi_0, \F)) \otimes F_q(C_b(X, \Xi_0, \F)))
 \subset F_{p+q-n}(C_{a+b-n}(X, \Xi_0, \F))$$  and \cite[Theorem 2.14]{McC01}
 proves that $\CS$ induces an algebra structure $$\CS^{(r)} : E^r_{p,q}
  \otimes E^r_{l,m} \to E^r_{p+l-n,q+m}$$ such that $E^r_{p,q}$  converges towards
  $H_*(X,\F)$ as algebras.

On the second page, we infer using the description of the cross products above, \cite[Remark 9.2.3]{BDHO23} stating that $\Delta_!$ induces on the second page $\Delta_!^{(2)}$ the usual shriek morphism with local coefficients and Proposition \ref{prop : induced morphism is limit of spectral sequence maps} stating that $\tilde{m}$ is a limit of morphism of spectral sequences, that the algebra structure on $H_s(X,H_t(\F))$ is given up to sign by the intersection product on $X$ with coefficients in $H_t(\F).$

\end{proof}

\subsubsection{Equivalence with the Grüher-Salvatore definition }

\begin{prop}\label{prop : CS < CSDG} 

The product $\CS$ corresponds, up to sign, at the homology level via the Fibration Theorem to the product $\mu_* : H_*(E)^{\otimes 2} \to H_*(E)$ defined in \cite{GS07}.
\end{prop}

\begin{proof}
We prove that the following diagram commutes up to sign.

$$
\xymatrix@C=15pt{
H_i(X,\F) \otimes H_j(X, \F) \ar[r]^-K \ar[d]^{\sim} & H_{i+j}(X^2, \F^2) \ar[r]^-{i_{\eta_{\Delta},!}} \ar[d]^{\sim} \ar@/^2pc/[rr]^-{\Delta_!}&  H_{i+j}(\overline{\eta_{\Delta}}, \partial \overline{\eta_{\Delta}}, \F^2) \ar@{=}[r] \ar[d]^{\sim} & H_{i+j-n}(X, \F^2) \ar[r]^-{\Tilde{m}} \ar[d]^{\sim} & H_{i+j-n}(X, \F) \ar[d]^{\sim}\\
H_i(E) \otimes H_j(E) \ar[r]^-K & H_{i+j}(E^2) \ar[r]^-{\tau_*} & \Tilde{H}_{i+j}\left( (E \ftimes{\pi}{\pi} E)^{\pi^* TX}\right) \ar[r]^-{u_*} & H_{i+j-n}(E \ftimes{\pi}{\pi} E) \ar[r]^-{m_*} & H_{i+j-n}(E). 
\label{diag : extension }
}
$$ 
The vertical arrows are given by the Fibration Theorem.
We used the following notations from \cite{GS07}:
\begin{itemize}
    \item $\eta_{\Delta}$ is a tubular open neighborhood of the diagonal in $X^2$ seen as a bundle over $X$ whose normal bundle is denoted by $\nu_{\Delta}$. Then $\eta_{\Delta}$ is homeomorphic to the total space of $\nu_{\Delta}$.
   
    \item $(E \ftimes{\pi}{\pi} E)^{\pi^* TX}$ is the Thom space of the normal bundle $\pi^*\nu_{\Delta} \simeq \pi^*TX$ associated to the inclusion $E \ftimes{\pi}{\pi} E \hookrightarrow E^2$. We have an identification $(E \ftimes{\pi}{\pi} E)^{\pi^* \nu_{\Delta}} \simeq \faktor{E^2}{E^2 \setminus \pi^* \eta_{\Delta}}$.
    
    \item $\tau_*$ is the Pontrjagin-Thom collapse map and $u_*$ is the Thom isomorphism.
\end{itemize}
    $\bullet$ Lemma \ref{lemme : Compatibility K and Fibration Theorem} shows that the first square commutes.

$\bullet$ The second and third squares commute up to sign. This is a direct application of Theorem \ref{thm : DG Thom iso}.

$\bullet$ The fourth square commutes by Theorem \ref{thm : morphisme induit commute avec iso} applied to the morphism of fibrations $m : E \ftimes{\pi}{\pi} E \to E$.
\end{proof}

\begin{cor}
    In particular, the product 

    $$\CS : H_*(X, C_*(\Omega X)^{ad})^{\otimes 2} \to H_*(X, C_*(\Omega X)^{ad})$$

    corresponds, up to sign, via the Fibration Theorem to the Cohen-Jones \cite{cohen2002homotopy} definition of the Chas-Sullivan product.
\end{cor}
\begin{flushright}
    $\blacksquare$
\end{flushright}

We can reprove using our setting the homotopy invariance of the Chas-Sullivan product.

\begin{prop} \textbf{(Homotopy invariance)}
    Let $(Y, \star_Y)$ be a smooth, oriented, pointed, closed and connected manifold. Let $g : Y \to X$ be an orientation-preserving homotopy equivalence such that $\star = g(\star_Y)$.
    Consider the fibrations $$\Omega X \hookrightarrow \ls{X} \to X \ \textup{and} \ \Omega Y \hookrightarrow \ls{Y} \to Y.$$ 
    
    There exists an isomorphism of rings $$g_{\#} : H_*(X,C_*(\Omega X)) \to H_*(Y, C_*(\Omega Y)).$$
\end{prop}

\begin{proof}
    We proved in Corollary \ref{cor : shriek map of an homotopy equivalence is ring iso with pullbacked module structure} that $g_! : H_*(X, C_*(\Omega X)) \to H_*(Y, g^*C_*(\Omega X))$ is a ring isomorphism. It remains to define a ring isomorphism $H_*(Y, g^*C_*(\Omega X)) \to H_*(Y, C_*(\Omega Y)).$
    Consider the map $$\ls{g} :\ls{Y} \to  g^*\ls{X} ,\ \ls{g}(\gamma)= (\gamma(0), g \circ \gamma).$$

  This is a morphism of fibrations over $\Omega Y$ and induces an isomorphism $$\ls{g}_* : H_*(\ls{Y}) \to H_*(g^* \ls{X}).$$

Corollary \ref{cor : equality of Ai morphism on total space gives equality of induced maps} shows that 

$$\widetilde{\ls{g}} : H_*(Y, C_*(\Omega Y)) \to H_*(Y,  g^*C_*(\Omega X)) $$ is an isomorphism.
Denote now $$g_{\#} = \widetilde{\ls{ g}}^{-1}  \circ g_! : H_*(X,C_*(\Omega X)) \overset{\sim}{\to} H_*(Y, C_*(\Omega Y)). $$

    We prove that $g_{\#}$ is an isomorphism of rings where the products on $H_*(X, \Omega X)$ and $H_*(Y, \Omega Y)$ are respectively the DG Chas-Sullivan products induced by the concatenations $m_X :  \ls{X} \ftimes{\ev}{\ev} \ls{X} \to \ls{X}$ and $m_Y :  \ls{Y} \ftimes{\ev}{\ev} \ls{Y} \to \ls{Y}$ which are morphisms of fibrations.\\

    We already know from Corollary \ref{cor : shriek map of an homotopy equivalence is ring iso with pullbacked module structure} that $$g_! : H_*(X, C_*(\Omega X)) \to H_*(Y, g^*C_*(\Omega X))$$ is an isomorphism of rings. It remains to prove that $$\widetilde{\ls{g}}^{-1}  : H_*(Y,g^*C_*(\Omega X)) \overset{\sim}{\to} H_*(Y, C_*(\Omega Y))$$ is a morphism of rings.

    Denote $\F = C_*(\Omega X)$, $\F^2 = C_*(\Omega X^2)$ and $\G = C_*(\Omega Y)$, $\G^2 = C_*(\Omega Y^2)$. The following diagram commutes

    $$\xymatrix{
    H_*(Y,g^*\F)^{\otimes 2} \ar[r]^-K & H_*(Y^2,  (g \times g)^*\F^2) \ar[r]^-{\Delta_!}  & H_*(Y, \Delta^*(g \times g)^*\F^2) \ar[r]^-{g^* \tilde{m}_X}  & H_*(Y, g^*\F) \\
    H_*(Y, \G)^{\otimes 2}  \ar[u]_{\widetilde{\ls{g}}^{\otimes 2}} \ar[r]^-K  & H_*(Y^2, \G^2) \ar[u]^{\widetilde{\ls{g} \times \ls{g}}} \ar[r]^-{\Delta_!}  & H_*(Y, \Delta^* \G^2) \ar[r]^-{\tilde{m}_Y} \ar[u]^{\Delta^*\widetilde{\ls{g} \times \ls{g}}}  & H_*(Y, \G) \ar[u]^-{\widetilde{\ls{g}}} .
    }$$

    Indeed, the commutativity of the square on
    
    \begin{itemize}
        \item the first column is a consequence of Lemma  \ref{lemme : induced maps on a Cartesian product commutes with K}.
        \item the second column is a consequence of Proposition \ref{prop : morphisme Ai commute avec direct et shriek}.
        \item the third column is a consequence of Corollary \ref{cor : equality of Ai morphism on total space gives equality of induced maps} and Corollary \ref{cor : composition of morphism of fibrations} since 

    $$m_X\circ( \ls{g} \times \ls{g}) = \ls{g} \circ m_Y.$$
    \end{itemize}

\end{proof}

\subsubsection{Pullback}

Let $(Y,\star_Y)$ be a pointed topological space and $F \hookrightarrow E \overset{p}{\to} Y$ be a fibration.
Let $f : X^n \to Y$ be a continuous map such that $f(\star_X) = \star_Y$. Let $m : E \ _p\times_p E \to E$ be a morphism of fibrations. Consider the pullback fibrations 

$$F \hookrightarrow f^*E \to X \textup{ and } F^2 \hookrightarrow f^*(E \ftimes{p}{p} E) \to X.$$

It is straightforward that $f^*m_* : f^*\left(E \ _p\times_p E\right) \to f^*E $ is also a morphism of fibrations. Therefore, there exists a product of degree $-n$

$$\CS : H_*(X, f^*C_*(F))^{\otimes 2} \to H_*(X,f^*C_*(F))$$
 
For instance :\begin{itemize}
    \item Let $X = \{\star\} \overset{i}{\subset} Y$ and consider $\Omega Y \hookrightarrow \ls{Y} \to Y$. Then, $\CS : H_*(\star, i^*C_*(\Omega Y))^{\otimes 2} \to H_*(\star,i^*C_*(\Omega Y))$ describes the Pontryagin product on $\Omega Y$.
    \item Let $Y = \ls{X}$ and $f : X \to \ls{X}$ be the inclusion of constant loops. Consider the fibration $\Omega \ls{X} \hookrightarrow \ls{\ls{X}} \to \ls{X}$. Then,  $\CS : H_*(X, f^*C_*(\Omega \ls{X}))^{\otimes 2} \to H_*(X,f^*C_*(\Omega \ls{X}))$ describes a product of degree $-n$ on pinched tori.
\end{itemize}

\subsection{DG Chas-Sullivan product for manifolds with boundary}

If $X$ has a boundary, take a Morse function without any critical points on the boundary and a pseudo-gradient pointing outward along $\partial_+ X \subset \partial X$ and inward along $\partial_- X\subset \partial X$ so that the boundary decomposes $\partial X = \partial_+ X \cup \partial_- X$ (see \cite[ Section 3.5]{AD14} for further explanations on Morse theory for a manifold with boundary and \cite[Section 5.3]{BDHO23} for the construction of Morse homology with DG coefficients for a manifold with boundary). 

Let $X,Y$ be pointed, oriented, compact and connected manifolds with boundary. Let $\F, \F'$ be two DG modules over $C_*(\Omega X)$ and $\G$ be a DG module over $C_*(\Omega Y)$. Let $\boldsymbol{\varphi} : \F \to \F'$ be a morphism of $\Ai$-modules.

The Künneth map $$\deffct{K}{C_*(X, \partial_+ X, \F) \otimes C_*(Y, \partial_+Y, \G)}{C_*(X \times Y, X \times \partial_+ Y \cup \partial_+X \times Y, \F \times \G)}{\alpha \otimes x \otimes \beta \otimes y}{(-1)^{|\beta||x|} (\alpha, \beta) \otimes(x,y)}$$ and  $$\tilde{\varphi} = \sum_n \varphi_{n+1} \m^{n} : H_*(X, \partial_+ X, \F) \to H_*(X, \partial_+ X, \F') $$ are defined in the same way, since all the trajectories between critical points avoid the boundary. 

Let $\psi : X \to Y$ be a continuous map. The shriek map $\psi_!$ has target $H_*(X,\partial_+ X, \psi^*\G)$ and factors through $H_*(Y, \G) \to H_*(Y, \partial_+ Y, \G) \to H_*(X,\partial_+ X, \psi^*\G)$ (see \cite{BDHO23} Remark 10.4.3). Therefore $\Delta_! : H_*(X^2, \partial_+ (X^2), \F^2) \to H_*(X,\partial_+ X, \Delta^*\F^2) $ is well-defined and has the same properties and the degree $-n$ product 

$$\CS : H_*(X,\partial_+ X, \F)^{\otimes 2} \to H_*(X,\partial_+ X, \F)$$

is defined in the same way and has the same properties under the same assumptions.

In particular, if $F \hookrightarrow E \overset{\pi}{\to} X$ and $m : E \ftimes{\pi}{\pi} E \to E$ is a morphism of fibrations, then $\CS$ induces a degree $-n$ product

$$\CS : H_*(E,E_+) \otimes H_*(E,E_+) \to H_*(E,E_+)$$

where we denoted $E_+ = \pi^{-1}(\partial_+ X).$

\section{Further directions}\label{section : Further direction}

In this paper, we laid out the necessary tools to study, using enriched Morse theory, products on total spaces of fibrations that \emph{intersect on the basis} and \emph{multiply on the fiber}. The Path-product defined and studied by \cite{Max23} is an example of such a product that does not fall into the category of products studied in this paper.

\subsection{Path products}

 Let $Y$ be a topological space, $X^n$ be a pointed, oriented, closed, connected manifold, and $f : X \to Y$ be a continuous map.\\

Define $$\mathcal{P}_{X,f} Y = \{(x,x',\alpha) \in X^2 \times \mathcal{P}Y, \ \alpha(0) = f(x) , \alpha(1)= f(x')\} $$ and $\ev_0, \ev_1 : \mathcal{P}_{X,f} Y \to X$, $\ev_0(x,x',\alpha) = x \in X$, $\ev_1(x,x',\alpha) = x' \in X$ the evaluation at the basepoint and endpoint. We will denote $\ev=(\ev_0,\ev_1) :  \mathcal{P}_{X,f} Y \to X^2$.
A degree $-n$ product $$\Lambda : H_*(\mathcal{P}_{X,f} Y)^{\otimes 2} \to H_*(\mathcal{P}_{X,f} Y)$$ has been defined and studied by \cite{Max23} if $Y$ is a closed manifold. This product, as the Chas-Sullivan product, is defined by intersecting on a space where the paths are concatenable and then concatenating.

Denote $\pi : \mathcal{P}_{X,f} Y \ftimes{\ev_1}{\ev_0} \mathcal{P}_{X,f} Y \to X^3, \ \pi(\gamma, \tau) = (\gamma(0), \gamma(1)=\tau(0), \tau(1))$.
Since $\mathcal{P}_{X,f} Y$ and \\ $\mathcal{P}_{X,f} Y \ftimes{\ev_1}{\ev_0} \mathcal{P}_{X,f}$ are the total spaces of the fibrations  $$\Omega Y \hookrightarrow \mathcal{P}_{X,f} Y \overset{\ev}{\rightarrow} X^2$$ and $$\Omega Y^2 \hookrightarrow \mathcal{P}_{X,f} Y \ftimes{\ev_1}{\ev_0} \mathcal{P}_{X,f} Y \overset{\pi}{\to} X^3, $$

we interpret this product in our setting by

$$\begin{array}{rll}
    \textup{PP}_{DG} : & H_i(X^2, C_*(\Omega Y)) \otimes H_j(X^2, C_*(\Omega Y))  &\\
    \overset{K}{\longrightarrow} & H_{i+j}(X^4, C_*(\Omega Y^2)) & \simeq H_{i+j}(\mathcal{P}_{X,f} Y \times \mathcal{P}_{X,f} Y)   \\
     \overset{D_!}{\longrightarrow} & H_{i+j-n}(X^3, \Delta^* C_*(\Omega Y^2)) & \simeq H_{i+j-n}(\mathcal{P}_{X,f} Y \ftimes{\ev_1}{\ev_0} \mathcal{P}_{X,f} Y) \\
     \overset{\tilde{m}}{\longrightarrow} & H_{i+j-n}(X^3, p^*C_*(\Omega Y)) & \simeq H_{i+j-n}(m(\mathcal{P}_{X,f} Y \ftimes{\ev_1}{\ev_0} \mathcal{P}_{X,f} Y)) \\
     \overset{p_{*}}{\longrightarrow} & H_{i+j-n}(X^2, C_*(\Omega Y)) & \simeq H_{i+j-n}(\mathcal{P}_{X,f} Y)
\end{array}$$

where $D : X^3 \to X^4, \ D(a,b,c) = (a,b,b,c)$ and $p : X^3 \to X^2, \ p(a,b,c) = (a,c)$ and $$m : \mathcal{P}_{X,f} Y \ftimes{\ev_1}{\ev_0} \mathcal{P}_{X,f} Y \to p^*\mathcal{P}_{X,f} Y$$ is the morphism of fibrations over $X^3$ induced by the concatenation of paths.

\begin{rem}
    Intuitively, given two chains $\sigma, \tau$ in $\mathcal{P}_{X,f}Y$, this product will intersect $\ev_{1,*}(\sigma)$ with $\ev_{0,*}(\tau)$ on $X$, concatenate at the intersection and forget the concatenation point.

    We can remark that, since we do not intersect along the diagonal $\Delta: X^2 \to X^4$, this product does not fall into the category of products studied in this paper. Nonetheless, the developed tools enable to define and study it. 
\end{rem}

Using techniques very similar to those used in Section \ref{section : Morse description and generalization of the Chas-Sullivan product} we prove in \cite{PPDG}

\begin{thm}
    The product $\textup{PP}_{DG}$ is associative, admits a neutral element and corresponds, via the Fibration Theorem, to the product $\Lambda$ defined by \cite{Max23}.
\end{thm}

\subsection{Further study on \texorpdfstring{$\Ai$}{Ai}-structures}

\subsubsection{\texorpdfstring{$\Ai$}{Ai}-algebra structures on enriched Morse complexes}

This introduction of $\Ai$-structures to study enriched Morse theory is a clear path to defining an $\Ai$-structures on enriched Morse chains as it has been defined in \cite{Abo11} and \cite{mazuir2022higheralgebraainftyomega} for the Morse cochains with coefficients in $\Z$. These constructions rely on studying the moduli space of perturbed Morse gradient trees $\overline{\mathcal{T}}^{\mathbb{X}}_t(y ; x_1, \dots, x_n)$ and particularly those of dimension $0$ and $1$.

The main difficulty in the enriched Morse setting would be to extend this work to define a fundamental class for the manifold with boundary and corners $\overline{\mathcal{T}}^{\mathbb{X}}_t(y ; x_1, \dots, x_n)$ in any dimension and understand how the orientations behave with respect to the boundary strata.

\subsubsection{Towards a Chas-Sullivan product \texorpdfstring{$\textup{CS}_{\Ai}$}{CSAi} with coefficients in an \texorpdfstring{$\Ai$}{Ai}-module}
Whenever $\F$ is an $\Ai$-module over $C_*(\Omega X)$ and there exists a morphism of $\Ai$-modules $m : \Delta^*(\F \otimes \F) \to \F$, there exists a product 

$$CS_{\Ai} : H_*(X, \F)^{\otimes 2} \to H_*(X, \F).$$

We will study, in future work, the question of associativity, commutativity and the existence of a neutral element for this product.

\bibliography{mybib}
\bibliographystyle{alpha}

Université de Strasbourg, Institut de recherche mathématique avancée, IRMA,
Strasbourg, France.

\emph{e-mail:} r.riegel@unistra.fr

\end{document}